%% file: averkov.tex
\documentclass[final,hidelinks,onefignum,onetabnum]{siamart220329}
\input{settings.tex}

\usepackage{lipsum}
\usepackage{amsfonts}
\usepackage{graphicx}
\usepackage{epstopdf}
\usepackage{algorithmic}
\usepackage{todonotes} 
\ifpdf
  \DeclareGraphicsExtensions{.eps,.pdf,.png,.jpg}
\else
  \DeclareGraphicsExtensions{.eps}
\fi

\newsiamremark{remark}{Remark}
\newsiamremark{hypothesis}{Hypothesis}
\crefname{hypothesis}{Hypothesis}{Hypotheses}
\newsiamthm{claim}{Claim}

\headers{Sparse convex relaxations in polynomial optimization}{A. Averkov, B. Peters, and S. Sager}

\title{Sparse convex relaxations \\ in polynomial optimization\thanks{Submitted to the editors DATE.   
\funding{This work was funded by the the Deut\-sche For\-schungs\-
ge\-mein\-schaft (DFG, German Research Foundation) - 314838170, GRK 2297
MathCoRe.}}}

\author{Gennadiy Averkov\thanks{Fakultät 1, Brandenburgische Technische Universität Cottbus-Senftenberg, Germany
  (\email{averkov@b-tu.de}).}
\and Benjamin Peters\thanks{Fakultät für Mathematik, Otto-von-Guericke Universität Magdeburg, Germany 
  (\email{benjamin@benjamin-peters.de}, \email{sager@ovgu.de}).}
\and Sebastian Sager\footnotemark[3] \thanks{Max Planck Institute for the Dynamics of Complex Technical Systems Magdeburg, Germany}}

\usepackage{amsopn}

\begin{document}

\maketitle

\begin{abstract}
We present a novel, general, and unifying point of view on sparse approaches to polynomial optimization. Solving polynomial optimization problems to global optimality is a ubiquitous challenge in many areas of science and engineering. Historically, different approaches on how to solve nonconvex polynomial optimization problems based on convex relaxations have been developed in different scientific communities. Here, we introduce the concept of monomial patterns. A pattern determines what monomials are to be linked by convex constraints in a convex relaxation of a polynomial optimization problem. This concept helps to understand existing approaches from different schools of thought, to develop novel relaxation schemes, and to derive a flexible duality theory, which can be specialized to many concrete situations that have been considered in the literature. We unify different approaches to polynomial optimization including polyhedral approximations, dense semidefinite relaxations, SONC, SAGE, and TSSOS in a self-contained exposition. We also carry out computational experiments to demonstrate the practical advantages of a flexible usage of pattern-based sparse relaxations of polynomial optimization problems. 
\end{abstract}





\input{1_intro.tex}

\input{2_modeling.tex}

\input{3_duality.tex}

\input{4_computations.tex}

\section{Conclusions and Outlook} \label{sec_summary}

We have presented different aspects of convexification and sparsity exploitation for primal and dual approaches to polynomial optimization in a self-contained review. The review contains novel proofs. For example, we proved the dual positivstellens\"atze from Putinar and Handelman using infinite-dimensional matrices together with a projection argument, which allowed to omit the introduction of moments. Our generalized framework is meant to facilitate the transfer of theoretical results to special cases.

Our numerical experiments indicate the huge potential of tailored relaxations based on pattern families. Structure exploitation may bring huge advantages in comparison to standard approaches both in terms of approximation quality and in computational runtime. Patterns have already been used for special cases. Ahmadi and coworkers  \cite{ahmadi2023sums} investigated a combination of the patterns $\N_2^n$ and axis-parallel chains for the unconstrained case $X = \R^n$. They were able to show equivalences of primal and dual formulation and exact representations for specific exponent sets $A$.

Ideally, one would like to develop an adaptive method that works best for any type of instances, i.e., with respect to exponents $\alpha \in A$ and coefficient vectors.
Detecting favorable structures a priori in an automatized, adaptive way, seems to be one of the future challenges in polynomial optimization. First ideas in this direction are based on graph algorithms \cite{Peters2021a} or integer programming \cite{Shaydurova2024}, but require further research. An integration into concepts from the global NLP community, in particular branch-and-bound frameworks, is also very promising.








\bibliographystyle{siamplain}
\bibliography{literature.bib}

\input{5_appendix.tex}


\end{document}

%% file: settings.tex
\usepackage{graphicx}%
\usepackage{multirow}%
\usepackage{mathrsfs}%
\usepackage{xcolor}%
\usepackage{textcomp}%
\usepackage{manyfoot}%
\usepackage{booktabs}%
\usepackage{subcaption}
\usepackage{longtable} 

\usepackage{cleveref}
\usepackage{extarrows}
\usepackage{tikz, pgfplots}
\pgfplotsset{compat=newest}
\usepgfplotslibrary{groupplots, statistics}
\usepackage{xspace}
\usepackage{nicefrac}
\usepackage{tabularx}
\usepackage{placeins}	
\usepackage{colortbl}	

\crefname{section}{Section}{Sections}
\crefname{example}{Example}{Examples}
\crefname{figure}{Figure}{Figures}

\crefname{subsection}{Subsection}{Subsections}

\newcommand{\yalmip}{\texttt{YALMIP}\xspace}
\newcommand{\cstssos}{\texttt{CS-TSSOS}\xspace}
\newcommand{\baron}{\texttt{BARON}\xspace}
\newcommand{\julia}{\texttt{JULIA}\xspace}
\newcommand{\mosek}{\texttt{MOSEK}\xspace}
\newcommand{\matlab}{\texttt{MATLAB}\xspace}

\newcommand{\sageopt}{\texttt{SAGEOPT}\xspace}
\newcommand{\python}{\texttt{PYTHON}\xspace}

\newcommand{\BF}{\operatorname{BF}}

\newcommand{\sprod}[2]{\langle {#1},{#2} \rangle}

\DeclareMathOperator{\cone}{cone}
\DeclareMathOperator{\conv}{conv}

\newcommand{\bfx}{\mathbf{x}}

\newcommand{\mF}{\mathcal{F}}

\newcommand{\mM}{\mathcal{M}}

\newcommand{\mP}{\mathcal{P}}

\newcommand{\N}{\mathbb{N}}

\newcommand{\Nnd}{\mathbb{N}^n_d}

\newcommand{\R}{\mathbb{R}}

\newcommand{\rmA}{A}

\newcommand{\rmF}{F}
\newcommand{\rmS}{S}
\newcommand{\Aex}{A_{\mathrm{ex}}}

\renewcommand{\epsilon}{\varepsilon}
\newcommand{\fex}{f^{\mathrm{ex}}}

\renewcommand{\R}{\mathbb{R}} 
\renewcommand{\N}{\mathbb{N}}
\newcommand{\cM}{\mathcal{M}} 
\newcommand{\cC}{\mathcal{C}} 
\newcommand{\cP}{\mathcal{P}} 
\newcommand{\cB}{\mathcal{B}} 
\renewcommand{\conv}{\operatorname{conv}} 
\renewcommand{\cone}{\operatorname{cone}} 
\newcommand{\Bx}{\operatorname{Box}}
\renewcommand{\BF}{\operatorname{BF}}
\newcommand{\MR}{\operatorname{MR}}

\newcommand{\vx}{\operatorname{vert}}
\newcommand{\cF}{\mathcal{F}}
\newcommand{\RLX}{\operatorname{RLX}}
\newcommand{\CRLX}{\operatorname{CRLX}}
\renewcommand{\sprod}[2]{\left\langle #1 , #2 \right\rangle }
 
\newcommand{\bS}{\mathbb{S}}
\newcommand{\GMC}{\operatorname{GMC}}
\newcommand{\tr}{\operatorname{tr}}

\newcommand{\triv}{\operatorname{triv}} 

\newtheorem{prop}{Proposition}[section]
\newtheorem{thm}[prop]{Theorem}
\theoremstyle{definition} 
\newtheorem{rem}[prop]{Remark}
\newtheorem{ex}[prop]{Example}
\newtheorem{defn}[prop]{Definition} 
\newtheorem{cor}[prop]{Corollary}

\newenvironment{nomenclature}{\begin{longtable}{@{}p{2.6cm}l}}{\end{longtable}}
\newcommand{\nclentry}[2]{#1 & #2 \\}

%% file: 1_intro.tex
\section{Introduction} \label{sec:intro}

Let $X$ be a basic closed semi-algebraic set defined as
\begin{equation} \label{eq_feasibleset}
	X = \{ x \in \R^n \colon g_1(x) \ge 0,\ldots, g_s(x) \ge 0, h_1(x) = 0,\ldots, h_t(x) = 0\}
\end{equation}
via $s$ polynomial inequality and $t$ polynomial equality constraints, where $s,t \ge 0$ are non-negative integers. 
In polynomial optimization, we are interested in mimizing a polynomial over the \emph{feasible set} $X$. We want to solve the problem 
\begin{equation}
	\tag{POP}\label{POP}
	\inf_{x \in X} f(x),
\end{equation}
where the objective $f$ and the polynomials used to describe $X$ are polynomials, i.e., elements of the ring $\R[x]$ of $n$-variate polynomials in the variables $x = (x_1,\ldots,x_n)$ with coefficients in $\R$. We use $x^\alpha : = x_1^{\alpha_1} \cdots x_n^{\alpha_n}$ for the monomials in $x$, where a monomial is determined by its exponent vector $\alpha  \in \N^n$. In this notation, $f(x) = \sum_{\alpha \in A} f_\alpha x^\alpha \in \R[x]_A$ is an $\R$-linear combination of monomials, where $A$ is a finite set of exponent vectors. Where convenient, we also use $f$ to denote the vector of coefficients of $f(x)$.
In general, we tried to use standard notation, avoid notational overhead, and introduce specifics when they are first used. A list of symbols is provided for convenience in the appendix. 

\subsection{The problem class polynomial optimization}

To get a first intuition for the problem class, we discuss polynomials and optimization from a general perspective, before we look into aspects of how to solve \eqref{POP}.

\subsubsection{Polynomials}

In polynomial optimization, the objective function and the inequalities and equalities describing $X$ are polynomials in real variables $x \in \R^n$. They form a huge and important subset in the more general class of nonlinear programming (NLP) problems.

Polynomial optimization problems of type \eqref{POP} appear in a variety of different areas and contexts, due to the ubiquity of polynomials and their properties. Polynomials arise directly in the mathematical modeling of real world systems, e.g., with a cubic function in the Fitzhugh-Nagumo model or with a quadratic function when air friction is modeled. They also arise from algebraic concepts. A famous example is the Rough-Hurwitz stability criterium which states that a time-invariant linear system is stable if the roots of it's characteristic polynomial have only negative real parts. Another example are symbolic reformulations and underestimators using polynomials, e.g., \cite{sherali2013reformulation,smith2001symbolic}. Even binary variables $x_i \in \{0,1\}$ can be modeled via concave quadratic constraints $x_i (1-x_i) =0$.
Concrete examples of applications in polynomial optimization, e.g., in discrete and combinatorial optimization, control systems and robotics, statistics, and electric power systems engineering are provided in \cite[Sec. 1.1]{ahmadi2019dsos}, \cite[Sec. 1.1]{laurent2009sums}, and \cite[Sec. 3.6]{MR3075433}. 

Yet another reason for the importance of polynomials is their universal approximation property. Weierstraß showed in 1885 that the subalgebra of polynomials is dense in the algebra of continuous functions (a property much celebrated and used nowadays for neural networks). The Taylor series of a function provides a convenient way to find polynomial approximations and error estimates for smooth functions. Polynomial regression and approximation are also computationally stable and efficient, e.g., using Tschebysheff polynomials and barycentric coefficients \cite{trefethen2019approximation}.
Therefore it is certainly no exaggeration to claim that mastering of the problem class \eqref{POP} would have a huge impact on many open challenges in science and engineering.

The structure of $X$ and the way it is given is of crucial important for particular approaches. In nonlinear programming, the convergence and performance of algorithms and also the conditions of optimality depend strongly on properties of constraints, such as the so-called constraint qualifications. In the approaches to polynomial optimization we are discussing here, they do not play a role, though. For \eqref{POP}, it is already interesting and nontrivial to address the special situations in which $X = \R^n$, which does not require any constraints, $X = [-1,1]^n$, modeled by the linear inequalities $-1 \le x_i \le 1$, or $X =\{-1,1\}^n$ modeled by the equalities $x_i(1-x_i)=0$. All those cases have a different flavor, and there are methods specific to those cases, which can be employed for solving a particular type of polynomial optimization problems. Nevertheless, independent of a particular version of \eqref{POP}, convexification is a common principle applied to all these versions. In the following we shall explain this universal principle without going into details regarding $X$, unless otherwise specified.

\subsubsection{Optimization}

A feasible point $x^*$ is called a global minimizer of \eqref{POP}, if there is no other point in $X$ with a lower objective function value. It is called a local minimizer, if there exists a neighborhood in $X$ around $x^*$ in which this property holds. It is way more difficult to find a global optimizer or even to verify global optimality of a given point. In contrast, local optimality can usually be verified locally via sufficient conditions of optimality by evaluating functions and higher-order derivatives in the candidate point.

Convex optimization, i.e., the minimization of a convex objective function over a convex feasible set $X$, has two important advantages compared to nonconvex optimization. The first advantage is that for strictly convex functions, the necessary first-order conditions of optimality are also sufficient conditions. The second, more important advantage is that a local minimizer of a convex problem is also a global minimizer. This can be seen by assuming the existence of two strict local minimizers, considering the connecting line of (feasible) points and constructing a contradiction to the local optimality of both from the convexity assumption on the objective function. For convex nonlinear problems, efficient algorithms have been developed in the last decades as shortly discussed in Section~\ref{sec_paradigms}. 

Many practical optimization problems are nonconvex, though. Sometimes, the determination of local optima is acceptable in an engineering context, based on the rationale that a descent-based algorithm initialized with the currently implemented practical solution will improve it towards a local optimium. Practitioners in industry might even argue that sometimes, moving from one local solution to a different one would imply additional undesirable costs. Examples are the necessary training of staff to control a complex process or related security issues.
In addition, finding global optimizers is simply often not realistically possible. Already local optimizers can be very difficult to find due to high computational costs, e.g., in the training of deep neural networks.

In many other contexts, apart of the intellectual desire to know what the very best solution is, global optimality is required. This is certainly true when optimization results are used for mathematical reasoning, such as the determination of bounds for kissing numbers \cite{Mittelmann2010}. Note that here often additional concepts such as interval analysis are necessary to obtain rigorous results \cite{Neumaier2004}.
Global optimization comes in several flavors and versions and is supported and developed by a wide range of sub-communities representing different philosophies. 
In our review, we are interested in the global solution of \eqref{POP} and shall develop a general convexification framework for this purpose.

\subsubsection{How to globally solve polynomial optimization problems} \label{sec_howto}

There are many excellent textbooks on polynomial optimization available \cite{marshall2008, lasserre2009moments, lasserre2011, wolkowicz2012handbook, MR3075433, gartner2012approximation, lasserre2015, nie2023moment, magron2023sparse}
that highlight the connection to positive polynomials, sum of squares (SOS), semidefinite programming (SDP), and related concepts from real algebra. 
Important contributions come also from a second community interested in global nonlinear programming (NLP), \cite{Neumaier2004}. Here, polyhedral relaxations and spatial branch and bound techniques \cite{Land1960,Little1963} have been suggested, e.g., \cite{Tawarmalani2005, tawarmalani2013convexification}, sometimes in the connection with additional concepts such as interval arithmetics \cite{Moore1966}. Often these approaches address also nonlinear optimal control problems, e.g., \cite{Adjiman1998,Esposito2000,Esposito2000a,Papamichail2002,Diedam2018}.

Many different approaches for the convexification of \eqref{POP} have been suggested in the literature. Among them are the well-known McCormick envelopes \cite{Mitsos2009,Scott2011,Boland2017}, i.e., the convexification of variables $x_1$ and $x_2$ and their product $x_1 x_2$. Other examples are truncated moment relaxation and its dual, the SOS relaxation, \cite{lasserre2011,laurent2009sums,marshall2008}, scaled-diagonally-dominant sum of squares (SDSOS) \cite{ahmadi2019dsos}, sums of non-negative circuit polynomials \cite{dressler2017,seidler2018experimental}, bound-factor products \cite{dalkiran2013boundfactor} and their dual Handelman's hierarchy \cite{handelman1988representing}, multilinear intermediates \cite{bao2015global}, polyhedral outer approximations \cite{tawarmalani2005polyhedral} as well as expression trees \cite{smith2001symbolic,sherali2013reformulation}.
These approaches are either based on primal or on dual formulations, as shall become clearer in the course of this review. It is one goal of this paper to present a unifying point of view that allows an easy understanding of the differences and similarities between these approaches.
Also, a variety of optimization solvers has evolved over the years, as discussed in Section~\ref{sec_computations}.

Polynomial optimization problems have obvious connections with real algebra, which in turn, is connected to semidefinite optimization. Since the seminal work of Jean-Bernard Lasserre, in the last two decades a community of experts working at the interface of global optimization, semidefinite optimization, real algebra, and real algebraic geometry has formed. This community seeks to understand the theoretical foundations of convexification as \emph{the} fundamental principle of global optimization and to use it in computations. It is interesting to observe the interplay between theory and computational practice in this context. The theory suggests tools of great generality supplied with extremely inefficient general algorithms. The computational practice exploits rather concrete convexifications (like McCormick envelopes), which can be deployed in general problems. Basically any kind of concretization of convexification ideas has led to a spin-off -- a sub-culture focussing and propagating this particular convexification technique and its application for solving \eqref{POP}. Examples are McCormick envelopes, other kinds of linear programming relaxations, Lasserre relaxations, SONC, and SAGE relaxations. 
With our review, we address three key similarities between these approaches from different sub-communities: first, using \emph{convexification} as a main tool, second, trying to impose \emph{sparse} constraints in practice, and third, establishing a \emph{duality theory} based on general principles from conic optimization. In our opinion, these similarities have not yet been articulated clearly enough. In our presentation, we demonstrate how the idea of sparse convexifications is implemented flexibly, without relying on specific modeling details. We illustrate the impact of convexification and sparsification concepts with numerical results.

\subsection{Convexification}

The (formal) convexification of problem \eqref{POP} consists of two elementary steps: lifting and describing the relationships between the variables $x$ and $v$ with convex constraints.

\subsubsection{Monomial Lifting} By introducing a new variable $v_\alpha = x^\alpha$ for each of the monomials that occur in \eqref{POP}, we linearize the objective, which becomes the linear function $(v_\alpha) _{ \alpha \in A} \mapsto \sum_{ \alpha \in A} f_\alpha v_\alpha$. We introduce the \emph{linearization} map $L_v$ for polynomials $f \in \R[x]$ by 
\[
L_v(f) := L_v \Bigl(\sum_{\alpha \in A} f_\alpha x^\alpha \Bigr) := \sum_{\alpha \in A} f_\alpha v_\alpha.
\]  Since  the monomial variable $v_0$ for the monomial $x^0 = 1$ is assigned the value $1$, we use $v_0$ as the constant $1$ if it is not specified otherwise. We call the linear inequality $L_v(f) \ge 0$ the linearization of the polynomial inequality $f \ge 0$. In \cite[Section~2.7]{lasserre2015} $L_v$ (without fixing $v_0=1$) is called the Riesz linearization functional. 
For example, the linearization of a polynomial $f:=(1-x_1)(1-x_2)$ is given by
$L_v ( f ) = L_v (1 - x_1 - x_2 + x_1 x_2) = v_{0,0} - v_{1,0} - v_{0,1} + v_{1,1} = 1 - v_{1,0} - v_{0,1} + v_{1,1}$.

Having a linear objective is favorable, but it requires the introduction of nonlinear equality constraints $v_\alpha = x^\alpha$, with $\alpha \in A$, additionally to the constraint $x \in X$. This stage of processing \eqref{POP} is the lifting stage, as we add new variables and optimize in a higher-dimensional space. The variable $v_\alpha$ is called the monomial variable for the monomial $x^\alpha$ (or, for the exponent vector $\alpha$). After this kind of reformulation, we have a vector $x= (x_1,\ldots,x_n)$ that varies in $X$ and determines the values of the monomial variables $v_\alpha$, $\alpha \in A$, on which the objective $L_v(f)$ depends. Thus, we do not need $x$ directly anymore, but we want to determine explicitly how the monomial variables, which are linked to each other through the $x$-variables, are interdependent. 

\subsubsection{Formal convexification} 

Problem \eqref{POP} is now reformulated as 
\[
\inf \Biggl\{ \sum_{\alpha \in A} f_\alpha v_\alpha \colon   x \in X,\  v_\alpha = x^\alpha \  ( \alpha \in A) \Biggr\}
\]
The next step is to project out the $x$-variables, keeping the monomial variables and introducing constraints that describe their relations. Since we want to arrive at a convex problem in the end, we describe the relation in terms of convex constraints. This intention can be formalized by replacing the constraint 
\begin{equation}\label{eq_curve}
	(v_\alpha)_{\alpha \in A} \in \{ (x^\alpha)_{\alpha \in A} \colon x \in X\},
\end{equation}
which is in fact a family of the constraints $v_\alpha = x^\alpha$, indexed by $\alpha \in A$,
with the convex constraint 
\begin{equation}\label{eq_momentbody}
	(v_\alpha)_{\alpha \in A}  \in \cM_A(X) := \conv \{ (x^\alpha)_{\alpha \in A} \colon x \in X\} .
\end{equation} 
We have a linear objective on a convex feasible set, so it is readily clear that \eqref{POP} is equivalent to the problem 
\begin{align}\tag{CVX-POP}\label{C-POPa}
	\displaystyle \inf \left\{ \sum_{\alpha \in A} f_\alpha v_\alpha \colon {v \in \cM_A(X)} \right\} 
\end{align}
in the sense that \eqref{POP} and \eqref{C-POPa} have the same optimal value. 
We emphasize that \eqref{C-POPa}, the convex problem which we call the monomial convexification of \eqref{POP}, is a formal problem.  By this we mean that \eqref{C-POPa} provides no indication on how to formulate  $\cM_A(X)$ in any of the available optimization paradigms. Instead, \eqref{C-POPa} simply declares our intention to reformulate or relax a nonconvex problem to a convex one. This process is called \emph{convexification} in the jargon of global optimization. The set $\cM_A(X)$ represents the (not yet known) convexification that conveys the (not yet known) convex constraints, through which the monomial variables $v_\alpha$, $\alpha \in A$, are related. Both in the theory of global optimization and real algebra, a major aim is to provide approaches to describe $\cM_A(X)$, exactly or approximately, within tractable paradigms of convex optimization. Linear, second order-cone, and semidefinite optimization are the most widely used paradigms of convex optimization for that purpose. The points of view on what \textit{tractable} means may vary among sub-communities. Theoretically inclined experts consider polynomial-time solvability of a class of optimization problems to be the standard definition of being tractable. In contrast, practitioners may apply a more refined distinction, as they also care about the conditioning of the problems, the rate of convergence and the exact order of polynomial running time. Hence, for example, in view of these remarks semidefinite optimization with a large size of semidefinite constraints can be considered to be tractable in the theoretically inclined community, but would be considered intractable in the practically inclined one. 

\subsection{Formal sparsification} \label{sec_sparsification}

We call a constraint \textit{dense}, if it involves a large fraction of the variables of the underlying optimization problem. In contrast, a \textit{sparse} constraint depends only on a small fraction of the optimization variables. More correctly, we are mainly interested in the difficulty of associated constraints and their impact on computational runtimes. E.g., when we use semidefinite modeling techniques, we measure the difficulty in terms of the sizes of the linear matrix inequalities, as discussed below in Section~\ref{sec_paradigms}. 

Whereas the question how to actually model the convex $\cM_A(X)$ is definitely \emph{the} core topic in global optimization, also the role of sparsity and approaches to impose sparsity are very important in practice. Suppose, for example, that we have developed an approach to model $\cM_A(X)$ exactly or approximately with some theoretical guarantees. This means that we have established a method that transfers the formal constraint $v \in \cM_A(X)$ into a set of concrete constraints,
\begin{align*} 
	\cM_A(X) \xlongrightarrow{\text{modeling method}} \text{Model of} \ \cM_A(X).
\end{align*} 
Our modeling method receives $A$ and a description of $X$ as an input and produces a model for $\cM_A(X)$. But what if, when $A$ is large and complicated, this model for $\cM_A(X)$ is not tractable, because it involves computationally hard-to-handle or ill-conditioned constraints? 
At a first glance, our modeling method is not helpful: applying it to \eqref{C-POPa} we obtain its concrete formulation, but we cannot solve it, because it is not tractable. But looking more precisely, we see that there is a remedy, and it relies on using sparse convex constraints. When we model $\cM_A(X)$, we want to bind \emph{all} of the variables $v_\alpha$ with $\alpha \in A$ by convex constraints. This is a tough task, because the number of variables may be large and their relations may be highly nontrivial. What we can do instead is creating smaller (and possible overlapping) groups of monomial variables and try to model the relations of variables within each single group. As one of the options, we can just use the \emph{same} modeling method for each group of variables. The latter is a high-level view on imposing sparsity, which refrains from details on how to choose the sparse sub-structures and what kinds of sparse sub-structures are favorable. The above strategy leads to what we call a  pattern relaxation of \eqref{POP} with respect to the pattern family $\{P_1,\ldots, P_N\}$, which covers $A$:
\begin{align}\tag{P-RLX}\label{P-RLXa}
	\displaystyle \inf \left\{ \sum_{\alpha \in A} f_\alpha v_\alpha \colon (v_\alpha)_{\alpha \in P_i} \in \mM_{P_i}(X) \ \text{for} \ i \in [N] \right\} 
\end{align}
This is a formal problem that declares the intention to find sparse constraints, and it offers flexibility in choosing patterns $P_i$ in a way that allows to achieve a desired balance between the approximation quality and the complexity of $\eqref{P-RLXa}$. 
By choosing certain specific ``shapes'' of $P_i$, we can also prescribe sparse constraints of a particular structure. This is also possible in a dual setting. For example and as explained in detail later in Section~\ref{sec_SAGE}, if we choose $P_i$ of the form $\{2\alpha, \alpha+ \beta, 2 \beta\}$, we end up with a SDSOS relaxation. If we choose the $P_i$ to be a simplicial circuit, we end up with a dual version of the SONC relaxation. There are many different approaches that can also be combined to formulate \eqref{P-RLXa}. 

We believe that our template how to derive convexifications of \eqref{POP} via the formal problems \eqref{C-POPa} and \eqref{P-RLXa} has a definite methodological advantage. While other sparse approaches are discussed inseparably of modeling methods, we clearly separate modeling and sparsity. In the course of developing a particular approach to solve \eqref{POP} we can change a modeling strategy, but keep the sparsity-imposing strategy unchanged or vice versa. In other words, the ways of imposing sparsity need not be presented together with the modeling techniques. Both tasks can be coordinated, but there is a certain amount of independence in carrying out these two tasks. As one of the highlights, we stress that the the key ideas of duality theory for sparse relaxations are independent of the concrete modeling approach. We are not aware of any other source where the duality for \eqref{P-RLXa} is presented in a model-independent way. 

The procedure of developing convexifications without imposing sparsity can be described via a diagram,
\begin{align*}
	X \xrightarrow{\text{monomial map}}  \{ (x^\alpha)_{\alpha \in A}  \colon x \in X \} \xlongrightarrow{\text{convexifying}}\cM_A(X)\xlongrightarrow{\text{lifting}}\cM_B(X). 
\end{align*}
The lifting part $\cM_A(X) \to \cM_B(X)$ is needed, when we want to pick a larger exponent set $B \supseteq A$ that we perceive to be more convenient to work with than the original set $A$. For example, if $n=1$, and we have $A = \{0,2,3,6\}$, which means that we optimize a  polynomial with the exponents $x^0, x^2, x^3, x^6$, we might not have a specific strategy that would help optimizing such a polynomial over $X$. Instead, we might consider $\cM_B(X)$ for 
 $B = \{0,1,2,3,4,5,6\}$, which would help us optimize all polynomials of degree at most $6$ over $X$. By providing a model for $\cM_B(X)$ and projecting $\cM_B(X)$ onto $\cM_A(X)$ we establish a strategy for optimizing polynomials over $X$ with the exponent vectors in $A$. 
Clearly, for solving the problem practically, we need to proceed with modeling and applying a particular convex optimization solver: 
\begin{align*} 
	\cM_B(X) \xlongrightarrow{\text{modeling}} \text{Model of} \ \cM_B(X) \xlongrightarrow{\text{solving}} 
\end{align*} 
When establishing a model, lifting can be used once again, since sometimes $\cM_B(X)$ is conveniently described as a projection of a higher-dimensional set. 

The procedure of developing a sparse convexification is a variation of the above procedure. In the sparse version, the set $B$ is the union of $P_1,\ldots,P_N$, and rather than trying to model $\cM_B(X)$ we model its $N$ projections, so that the respective diagram is as follows: 
\begin{align*}
	\cM_B(X) \xlongrightarrow{\text{projecting}}\cM_{P_i}(X) \xrightarrow{\text{modeling}} \text{Model of} \ \cM_{P_i}(X)
\end{align*}
applied for each $i \in [N]$. When models are established, we can move over to solving: 
\begin{align*}
	\text{Model of} \ \cM_{P_1}(X),\ldots, \text{Model of} \ \cM_{P_N}(X) \xrightarrow{\text{solving}} 
\end{align*} 

Figure~\ref{fig:PRLX} illustrates the procedure for the approximation of a cubic polynomial in one variable $x$ and exponents $A = \{1,2,3\}$ on $X=[0,1]$. Note that here we used $B=A$. In general, pattern relaxations can not be expected to be as tight as in this simple case.

The above strategy is an ``assembly line'' that brings us from \eqref{POP} to a lower bound on \eqref{POP}. It consists of the following steps and decisions: 
\begin{enumerate}
	\item What $B \supseteq A$ and which sets $P_1,\ldots, P_N$ covering $B$ do we use? 
	\item How do we model $\cM_{P_i}(X)$? In which optimization paradigm? 
	\item What algorithm is appropriate for the relaxation we have established? 
\end{enumerate} 

The focus of our exposition is on the first question. We shall structure our review according to linear programming, second-order cone programming, and semidefinite programming where appropriate. The question of the overarching optimization paradigm and algorithmic choices will only be shortly commented upon in the following section.

\begin{figure}[hbt!]
	\begin{subfigure}{\linewidth}
		\begin{center}
			\includegraphics[height=3.5cm]{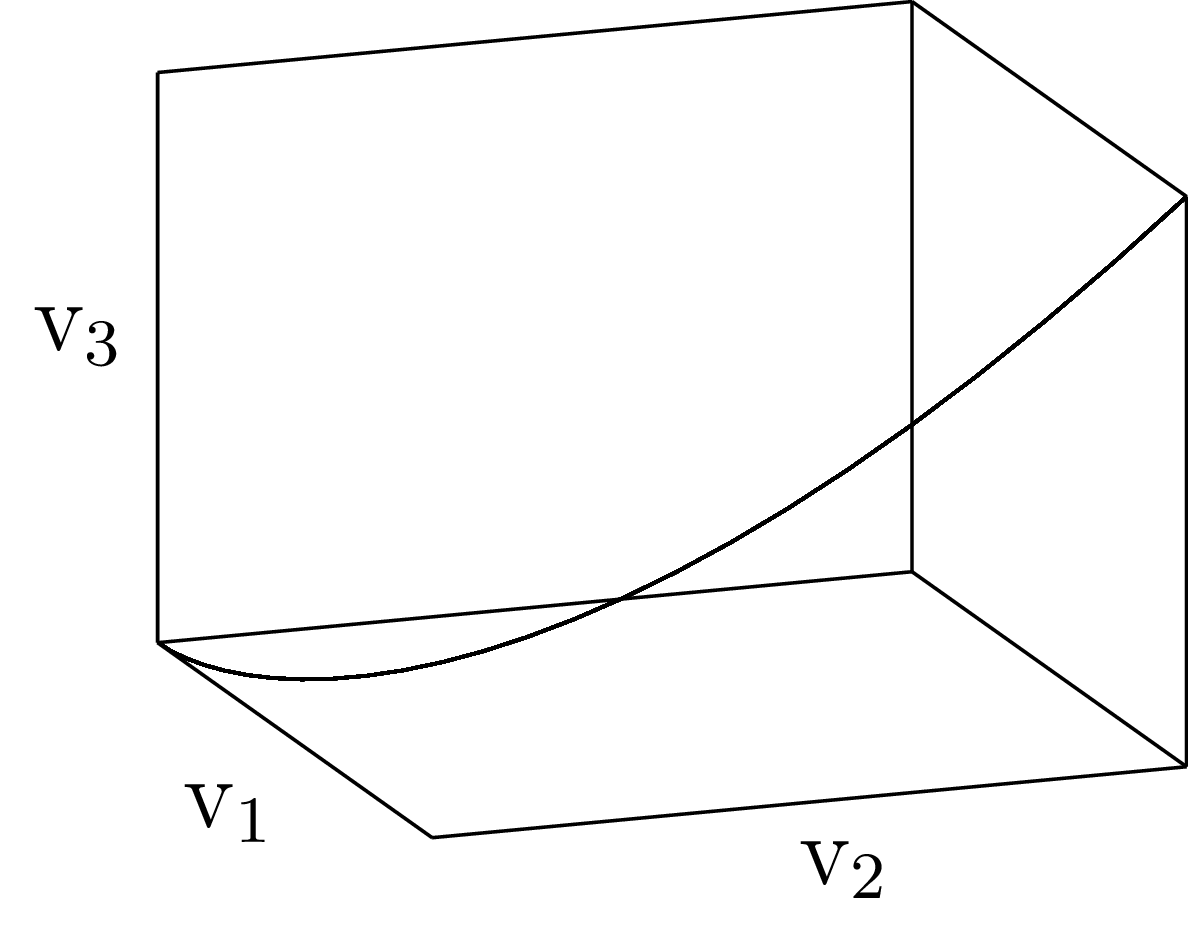}\qquad
			\includegraphics[height=3.5cm]{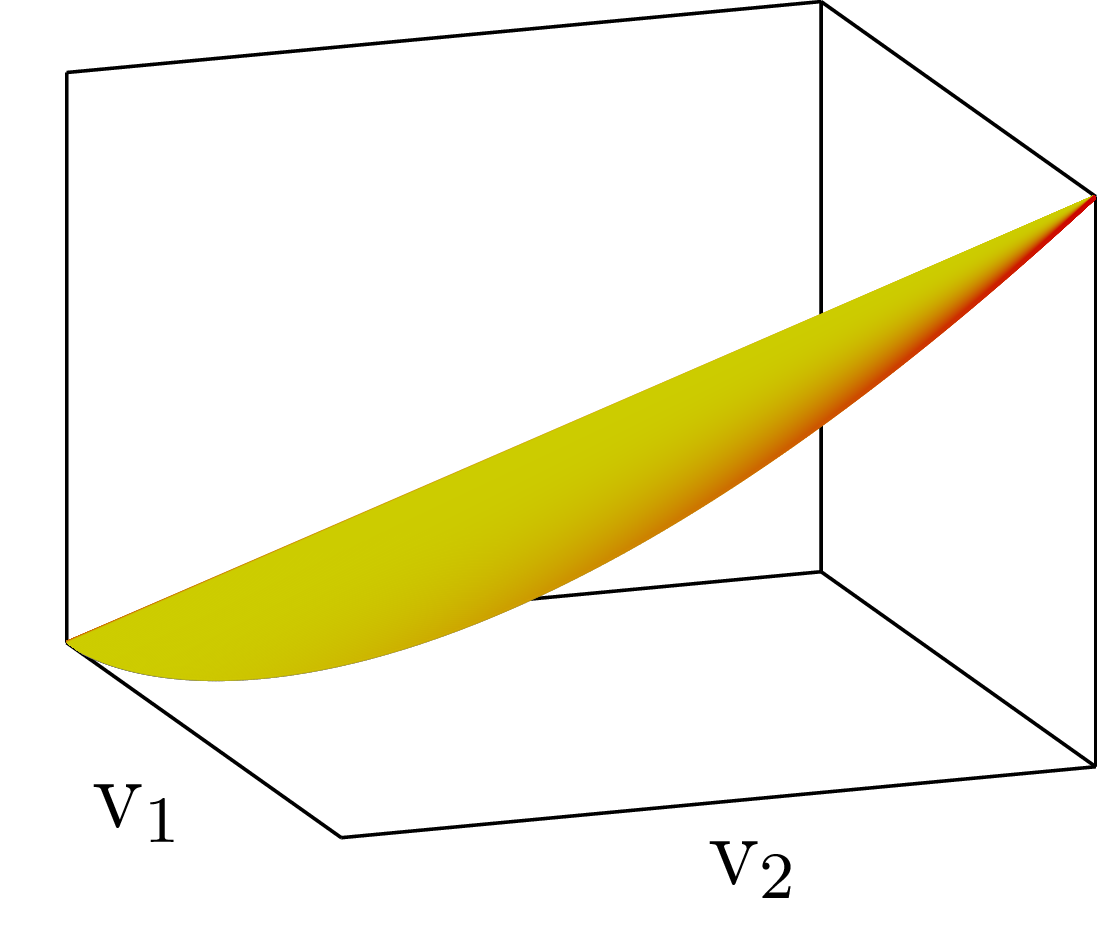} 
		\end{center}
	\end{subfigure}\par\bigskip\medskip
	\begin{subfigure}{\linewidth}
		\includegraphics[height=3.5cm]{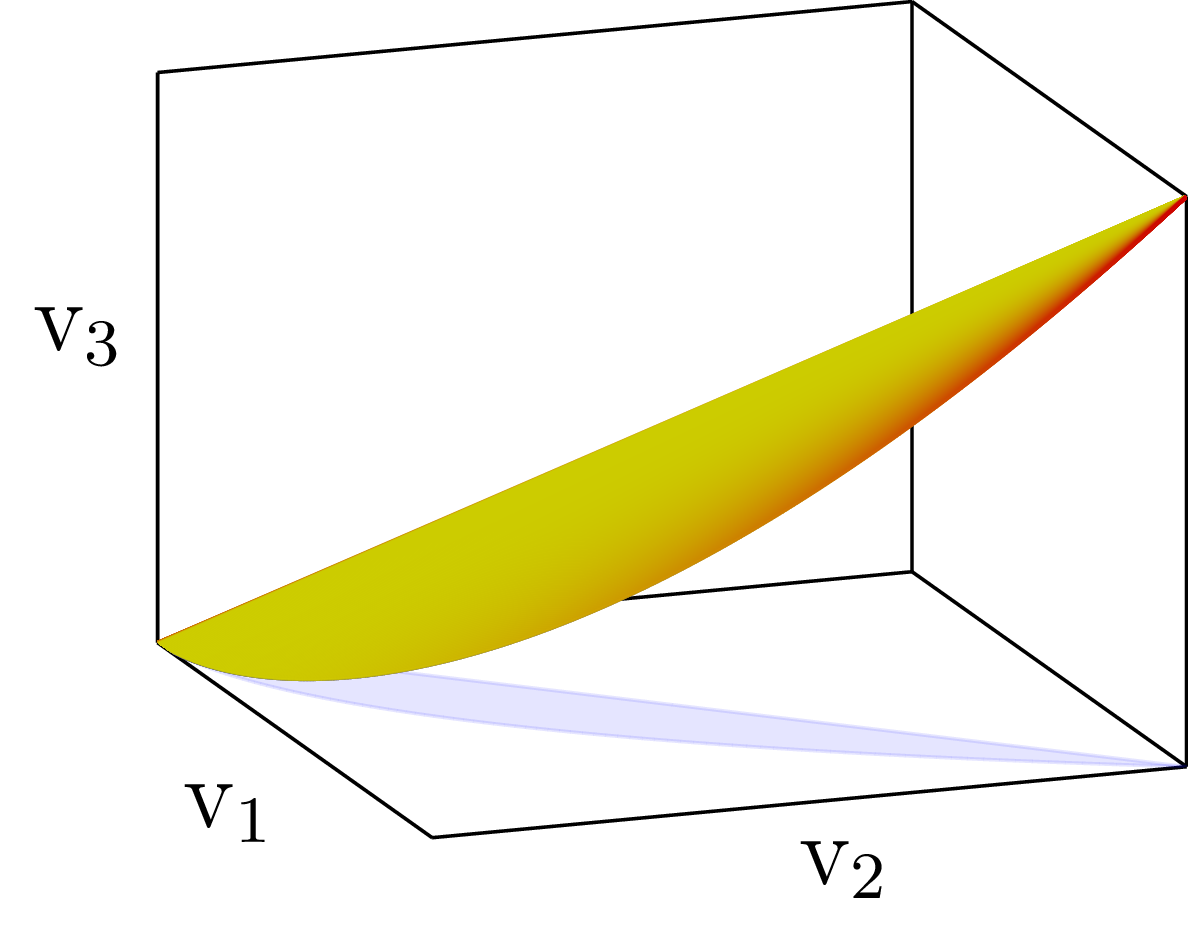}\hfill
		\includegraphics[height=3.5cm]{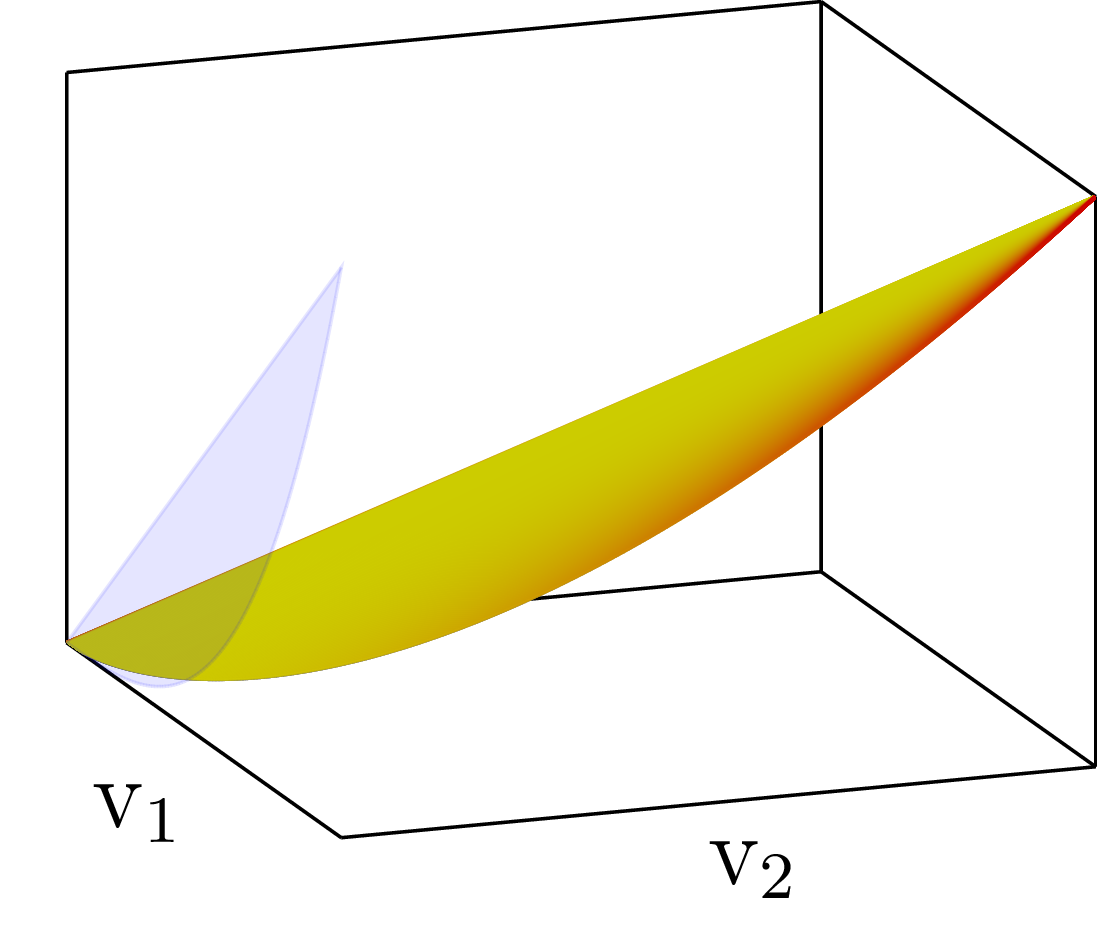}\hfill
		\includegraphics[height=3.5cm]{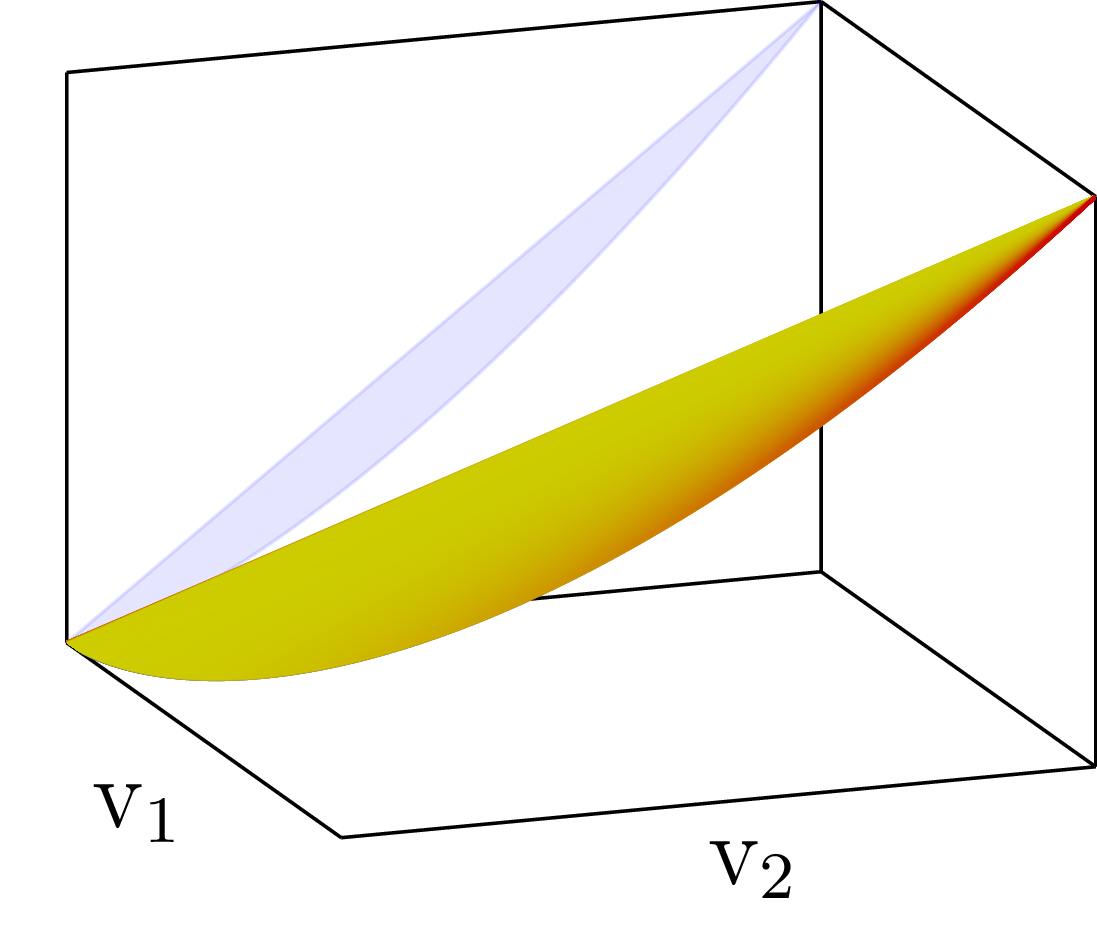}
	\end{subfigure}\par\bigskip\medskip
	\begin{subfigure}{\linewidth}
		\includegraphics[height=3.5cm]{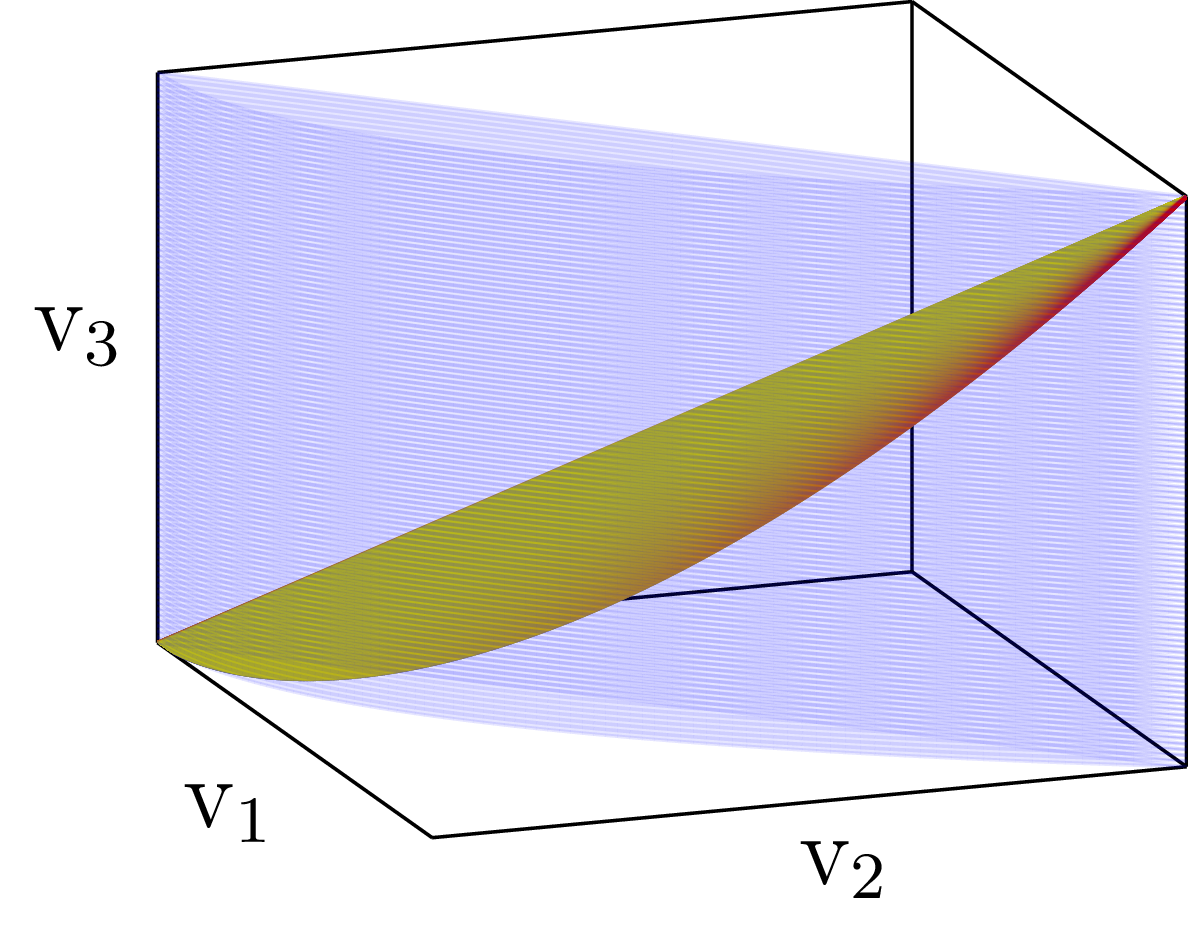}\hfill
		\includegraphics[height=3.5cm]{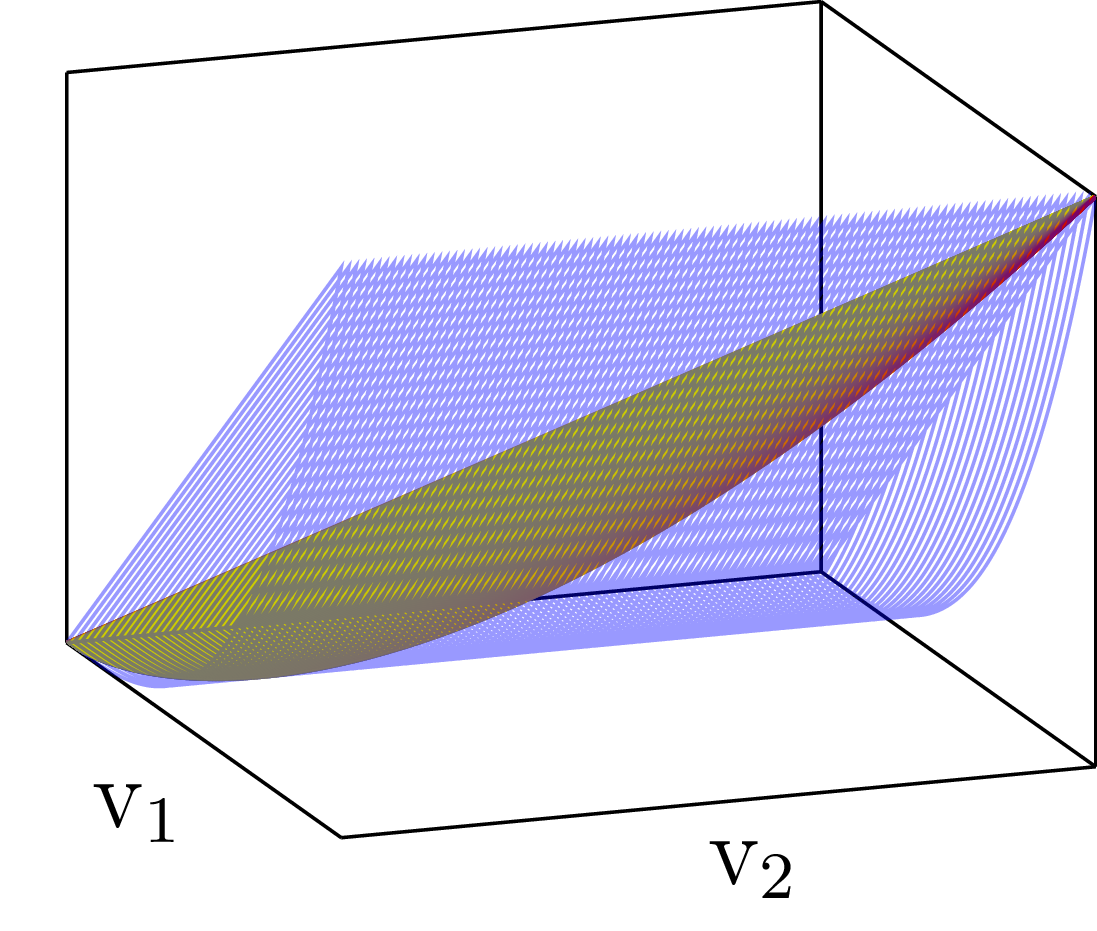}\hfill
		\includegraphics[height=3.5cm]{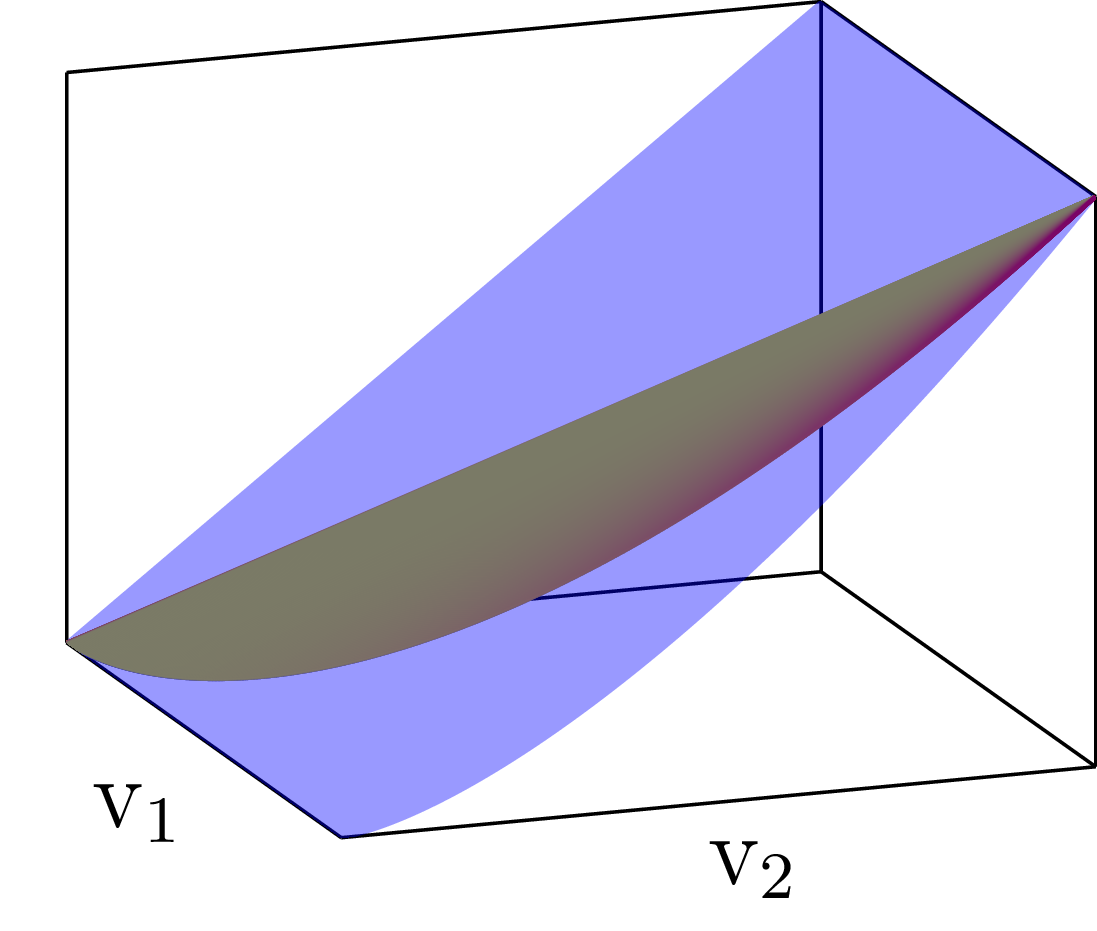}
	\end{subfigure}\par\bigskip\medskip
	\begin{subfigure}{\linewidth}
		\begin{center}
			\includegraphics[width=5.04cm]{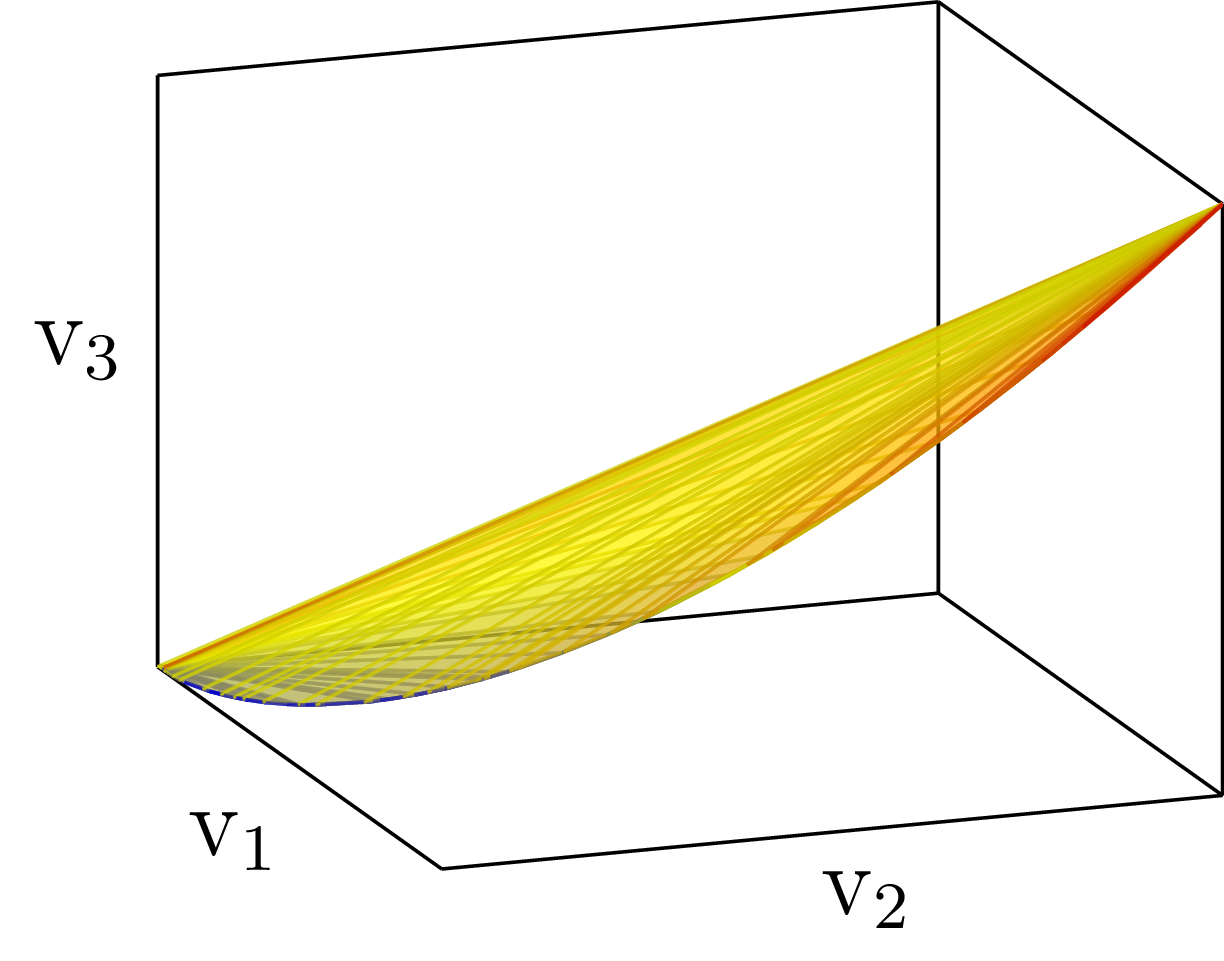}
		\end{center}
	\end{subfigure}
	\caption{Elementary steps towards a moment relaxation of a one-dimensional cubic polynomial. 
		\textbf{First row, left:} the curve $(x^\alpha)_{\alpha \in A}$ from \eqref{eq_curve} for $A=\{1,2,3\}$ and $X=[0,1]$, indicating the feasible values for the lifted, monomial variables $v$. 
		\textbf{Right:} the moment body $\mM_A(X) = \conv \{ (x^\alpha)_{\alpha \in A} \colon x \in X\}$, see \eqref{eq_momentbody}.
		\textbf{Second row, left to right:} projections $\mM_{P_i}(X)$ of $\mM_A(X)$ in blue for $P_i\in\mF := \{\{1,2\}, \,\{1,3\}, \,\{2,3\}\}$.
		\textbf{Third row, left to right:} liftings of $\mM_{P_i}(X)$ into $\R^3$ for all $P_i\in\mF$.
		\textbf{Last row:} the (intersected) feasible region $(v_\alpha)_{\alpha \in P_i} \in \mM_{P_i}(X) \ \text{for} \ i \in [3]$ of problem \eqref{P-RLXa}.
	}\label{fig:PRLX}
\end{figure}

\subsection{Modeling choices and optimization paradigms} \label{sec_paradigms}

As already mentioned in Section~\ref{sec_howto}, there are two schools of thought on computationally dealing with $\cM_A(X)$, to which we shall refer shortly as SDP and NLP communities, respectively. 
SDP communities introduce a sequence of semidefinitely representable sets $( S_d )_{d \ge \deg f} $ that approximate $\cM_A(X)$ and converge to $\cM_A(X)$ as $d \to \infty$. Fixing a specific $d$, \eqref{C-POPa} can be relaxed by replacing $\cM_A(X)$ with $S_d$, and semidefinite programming algorithms can be used to solve the resulting relaxation. However, as $d$ grows, the complexity of the underlying semidefinite relaxation may grow immensely. In practice, only the lowest level of the hierarchy is usually used in computations. 
The NLP community advocates the use of linear programming models. Linear programming models may be less suitable for finding tight approximations of the sets $\cM_A(X)$, because $\cM_A(X)$ may not be polyhedral and linear constraints can only describe polyhedra exactly. However, linear models work extremely well in branch-and-bound frameworks, because the dual simplex method for linear programming can be warm-started. 
There are additional approaches targeted towards other optimization paradigms. Usually these are other forms of cone programming, such as second-order cone or copositive programming. We do not intend to argue in favor of either of these schools. Also, it is beyond the scope of this review to look into the algorithmic details.

Our intention is to show the similarities of the convexification theory independent of the choice of a particular modeling approach and to integrate the available modeling techniques into the unifying convexification theory. 
In this sense we want to highlight one important similarity here. In all approaches, ultimately convex optimization problems have to be solved -- either as subproblems in a divide-and-conquer approach, or as a reformulation of the original problem. For these convex optimization problems descent based algorithms can be applied, because a local optimum is also a global optimum. Often an interior point method is used, sometimes also active set based approaches as the simplex method.
Now it is crucial from a practical point of view to understand the impact of different relaxations on the computational runtime. In general, this is very difficult to assess because of the unknown iteration numbers until convergence. However, the computational cost per iteration as an important contributor to the overall runtime is somewhat easier to estimate and a good indicator. In all Quasi-Newton methods -- primal or primal-dual interior point, active set based, or parametric programming -- the computational cost per iteration is usually dominated by the costs to solve systems of linear equations. 

The main motivation for sparsification as discussed in Section~\ref{sec_sparsification} is that sparsity can be exploited in the numerical linear algebra to solve systems of linear equations. 
E.g., the \mosek handbook \cite{mosek} emphasizes that ``having many small matrix variables is more efficient than one big matrix variable''. Some details on the specific case of interior point methods for large--scale cone programming are provided in \cite{Andersen2011}.
The main reason for this is that it is computationally cheaper to solve several smaller systems of linear equations than fewer and larger ones. You would always prefer to invert 1000 matrices of size 10 by 10 in comparison to inverting one matrix of size 10000 by 10000. This is due to the cubic runtime of the Cholesky method to obtain a LU decomposition of a symmetric matrix. Numerical results for the solution of SDPs of varying size indicate the approximate cubic runtime increase in the SDP size \cite[Table1.5]{Andersen2011}.
Considering matrices of size $n$ by $n$, a block decomposition into $N$ matrices of size $m$ would thus be asymptotically highly preferable if $N m^3 \ll n^3$. 
Of course there are many more details. Linear matrix inequalities may contain several matrix variables, and matrix variables may feature in several inequalities, and all of them may have different sizes. Sparsity may not only be exploited via block decompositions, but in a variety of ways by exploiting specifics of the model or of the optimization paradigm \cite{Andersen2011}. In some nonlinear optimization solvers, BFGS updates are directly applied to an approximation of the inverse of the Hessian, reducing the effort from cubic to quadratic for a matrix vextor multiplication.

Still, as a rule of thumb the rough $N m^3 \ll n^3$ improvement motivates why it is so interesting to look at the intersection of many patterns $\cM_{P_i}(X)$ of reduced size compared to $\cM_A(X)$. One illustrative example was shown in Figure~\ref{fig:PRLX}, although here no significant savings can be expected for the small dimensions $N=n=3, m=2$. As a main take-away we summarize that in all approaches to solving \eqref{POP}, the runtime of algorithms can be reduced drastically by creating and exploiting sparsity. In practice, this often makes the difference between intractable and computationally feasible and is a driving force in polynomial optimization reseach.

\subsection{Duality} There are two dual points of view on polynomial optimization: the convexification point of view we have just discussed and the non-negativity certification point of view. Non-negativity certification is a far-reaching generalization of the ideas behind Farkas' lemma and duality in linear programming to the case of polynomial optimization. 
When one presents the theory of polynomial optimization, one either starts with the convexification and then passes through non-negativity certification via dualization or one does it the other way around, starting with the questions of non-negativity certification and then passing to convexification via dualization. We believe that it is instructive to have a general presentation of the duality theory for sparse approaches, purified from the modeling details. Such a presentation allows to understand the core ideas and to gain insight regarding how different approaches are related to each other. For many specific choices of patterns, duality theorems were formulated independently of each other. The general duality underpinning all of these special cases is the duality of \eqref{P-RLXa} to the bound-certification problem 
\[
\sup \{ \lambda \in \R \colon f - \lambda = g_1 + \cdots + g_N, \ g_i \in \mP(X)_{P_i} \ i \in [N] \}. 
\]
Here, the non-negativity of $f-\lambda$ on $X$ is to be certified as a sum of sparse polynomials with supports contained in $P_1,\ldots,P_N$ which are  non-negative on $X$. Analogously to how we discussed modeling of \eqref{P-RLXa}, there is an aspect of modeling $\cP(X)_{P_i}$ coming up in relation to the bound-certification problem. Usually, a model of $\cM(X)_{P_i}$ can be dualized to obtain a model of $\cP(X)_{P_i}$, applying the duality of the respective optimization paradigm (e.g., linear, second-order cone or semidefinite-programming duality). 
Special cases, and a detailed discussion of duality will follow in Section~\ref{sec_conic}.

\subsection{Contributions} 

We emphasize that none of the propositions in this review is genuinely novel or surprising, although we formulated them in a self-consistent manner using basis tools from convex analysis and linear algebra. In particular the proofs for Theorems~\ref{dual:putinar} and \ref{dual:handelman} are more compact than in other references we are aware of.
The main contribution lies in generalization and abstraction of existing ideas. The unifying framework allows to derive novel approaches to convexification and/or sparsification, but also to derive theoretical results applicable to all existing and future methods.
We illustrate the potential of convexification and sparsification with numerical experiments.

\subsection{Structure of the paper} 
In Section~\ref{sec_specific} we provide an overview of two basic convexification principles, based on linear and semidefinite programming, respectively, and provide concrete suggestions on how to flexibly sparsify the  convexification approaches. 
The second major part of this paper addresses in Section~\ref{sec_conic} the conic viewpoint to pattern relaxations, duality theory for pattern relaxations, and the theory of lifted conic formulations in the context of polynomial optimization. This part provides a general systematic treatment of duality for sparse approaches to polynomial optimization. It is independent from Section~\ref{sec_specific}.
In Section~\ref{sec_computations} we present results of numerical experiments for specific patterns.

%% file: 2_modeling.tex
\section{Modeling principles for convexification} \label{sec_specific}

In this section we shall review different techniques for convexification and sparsification of $\{ (x^\alpha)_{\alpha \in A} \colon x \in X\}$, guided by different principles. We start with some definitions and a look at semidefinite and linear programming, before we continue with sparsification.

Let $\bS^m$ be the space of symmetric matrices of size $m$ over reals and $\bS_+^m$ the convex cone of positive semidefinite matrices  in $\bS^m$. For a matrix $M \in \bS^m$, we denote the condition $M \in \bS_+^m$ as $M \succeq 0$. 

If $M : \R^N \to \bS^m$ is an affine map, then the condition $M(v) \succeq 0$ is called a linear matrix inequality (LMI) of size $m$. In other words, $M(v)$ is a symmetric matrix with the entries being affine functions in $v$ and the condition $M(v) \succeq 0$ requires positive semidefiniteness of $M(v)$. Semidefinite optimization is optimization of a linear objective function subject to finitely many LMIs. If all the LMIs have size $2$, then we get a special case of semidefinite opimization, called second-order cone optimization. If all LMIs have size $1$, we obtain linear programming as an even more special case of semidefinite optimization. Semidefinite, second-order cone, and linear optimiation have a well-developed algorithmic theory, with the underlying algorithms implemented in a huge number of solvers and used in myriads of applications. 

\begin{defn}
	The set defined by finitely many LMIs is called a spectrahedron. A set which is the image of a spectrahedron under a linear transformation is called a projected spectrahedron or a semidefinitely representable set. 
\end{defn} 

\subsection{Semidefinite modeling principles} \label{sec_sdp}

We first look at semidefinite modeling principles.

\subsubsection{Semidefinite convexifications from inference rules}

Our objective is to lay out ways of deriving convex inequalities valid on $\cM_A(X)$ for the feasible set $X$ defined in \eqref{eq_feasibleset}.
Observe that, if $p(x) \ge 0$ is a polynomial inequality valid for $x \in X$, then $L_v(p) \ge 0$ is a linear inequality valid  for $v \in \cM_A(X)$. Furthermore, if $M(x)$ is a matrix whose entries are polynomials in $\R[x]_A$ and the condition $M(x) \succeq 0$ is true for $x \in X$, then the LMI $L_v (M(x)) \succeq 0$ is true for $v \in \cM_A(X)$, where $L_v(M(x))$ is the entry-wise linearization of the matrix $M(x)$. 

Let $M_B(x) = x^B (x^B)^\top$. For every $x \in \R^n$, we have easy inference rules: 
\begin{align}
  p_1(x) \ge 0,\ldots, p_t(x) \ge 0 & & &  \Rightarrow & & p_1(x)\cdots p_t(x) M_B(x) \succeq 0, \label{prod:lmi:rule} 
\\ p_1(x) \ge 0,\ldots, p_t(x) \ge  0 & & & \Rightarrow & & p_1(x) \cdots p_t(x) \ge 0,\label{prod:rule} 
\\	p(x)  \ge 0 &  &  & \Rightarrow & & p(x) M_B(x) \succeq 0, \label{cond:lmi:rule} 
\\	1 \ge 0	&  &  &  \Rightarrow & & M_B(x) \succeq 0, \label{uncond:lmi:rule} 
\\  p(x) = 0 & & & \Rightarrow & & p(x) q(x) = 0   \label{equation:rule} 
\end{align} 
where $p,p_1,\ldots,p_t, q \in \R[x]$ and $B$ is a finite subset of $\N^n$. The inequalities $g_i(x) \ge 0$ and the equations $h_j(x)=0$ that define $X$ are our fundamental knowledge concerning feasibility. Application of the inference rules gives us further, redundant knowledge, which turns out to be useful for convexification purposes. Rules (\ref{prod:lmi:rule}--\ref{uncond:lmi:rule}) deal with inequalities and differ in their degree of generality, with the rule \eqref{prod:lmi:rule} subsuming the rules (\ref{prod:rule}--\ref{uncond:lmi:rule}). \eqref{prod:lmi:rule} is a special case of \eqref{prod:rule} with $B=\{0\}$, \eqref{cond:lmi:rule} is a special case of \eqref{prod:lmi:rule} with $t=1$ and \eqref{uncond:lmi:rule} is a special case of \eqref{cond:lmi:rule} with $p=1$. 

Inserting the polynomial inequalities and equations defining $X$ into the premises of (\ref{prod:lmi:rule}--\ref{uncond:lmi:rule}) and linearizing the conclusions of these rules using $L_v$, we obtain linear and semidefinite constraints that are valid for the moment body of $X$. In this way different outer relaxations of the moment body $\cM_A(X)$ can be derived as projected spectrahedra. If one only uses the rule \eqref{prod:rule}, one obtains an outer approximation by a polyhedron.

The more general the rule is, the harder the implied constraints for $\cM_A(X)$ may become. For example, if the product $p_1(x) \cdots p_t(x)$ in rules \eqref{prod:lmi:rule} and \eqref{prod:rule} has many factors, then the respective linearization will involve many monomial variables. If an exponent set $B$ involved in rules \eqref{prod:lmi:rule} and \eqref{cond:lmi:rule} is large, then the respective linearization will be an LMI of a large size. That is, more general rules allow a more exact approximation of $\cM_A(X)$, but come at a higher computational price as discussed in Section~\ref{sec_paradigms}.

\begin{rem}
	Let us estimate the number of monomial variables involved in the linearization derived from 
	(\ref{prod:lmi:rule}--\ref{equation:rule}). Since in our context linearization is simply substitution of monomials with monomial variables, we need to track the relation between the monomials involved in the premises and the conclusions of our inference rules. 
	\begin{itemize} 
		\item  If $p_i(x) \in \R[x]_{A_i}$, then the entries of the matrix $p_1(x) \cdots p_t(x) M_B(x)$ in the conclusion of \eqref{prod:lmi:rule} belong to $\R[x]_{A_1 + \cdots + A_t + B + B}$. 
		\item For \eqref{prod:rule}, one has $p_1(x) \cdots p_t(x) \in \R[x]_{A_1+ \cdots + A_t}$.
	  \item For \eqref{cond:lmi:rule}, if $p(x) \in \R[x]_C$, the entries of $p(x) M_B(x)$ belong to $\R[x]_{C+B+B}$.  
		\item For \eqref{uncond:lmi:rule}, the entries of $M_B(x)$ belong to $\R[x]_{B+B}$.
	  \item For \eqref{equation:rule}, if $p(x) \in \R[x]_C$ and $q(x) \in \R[x]_D$, then $p(x) q(x) \in \R[x]_{C+D}$. 
	\end{itemize} 
	Deliberately choosing exponent sets $A_i, B, C,D$ of a manageable size and appropriate shape, we can control the computational complexity associated with the processing of the repsective linearizations. That is, we can make the constraints sparse and we can produce  LMIs of a desirable (and not too large) size. This is an overarching idea of most of the existing convexifications that rely on rules (\ref{prod:lmi:rule}--\ref{equation:rule}). 
\end{rem} 

\subsubsection{Positivstellens\"atze and their duals} 

Real algebra and real algebraic geometry allow us to understand in which cases the convexifications form the above inference rules approximate $\cM_A(X)$ arbitrarily well or even describe $\cM_A(X)$ exactly. From the practical perspective, however, one definitely needs to keep track of the complexity of the convex constraints resulting from those rules. The inferences \eqref{cond:lmi:rule} and \eqref{uncond:lmi:rule} are the favorite choices for linearization in the semidefinite optimization and real-algebra communities. These choices are backed by Putinar's Positivstellensatz \cite{putinar1993positive}, which uses so-called sum of squares certificates of positivity. A polynomial in $\R[x]$ is called SOS, if it admits a representation $f_1^2 + \cdots + f_t^2$ with $f_1,\ldots,f_t \in \R[x]$. We denote by $\Sigma$ the set of all SOS polynomials in $\R[x]$ and call it the infinite-dimensional SOS cone. 

\begin{thm}[Putinar]  \label{putinar:thm} 
	Let $X = \{ x \in \R^n \colon g_1(x) \ge 0,\ldots, g_s(x) \ge 0\}$ be a compact semialgebraic set contained in the ball $\{ x \colon x_1^2 + \cdots + x_n^2 \le \rho\}$ of some radius $\rho \in \R_{>0}$. Let $g_0 (x):=1$ and $g_{s+1}(x) := \rho - (x_1^2 + \cdots + x_n^2)$. Then every polynomial $f$ which is strictly positive on $X$ belongs to $f = g_0 \Sigma + \cdots + g_{s+1} \Sigma$. 
\end{thm}

In the following theorem, we use the infinite symmetric positive semidefinite matrix $M_{\N^n}  = x^{\N^n}(x^{\N^n})^\top$. For a symmetric matrix $M = (m_{i,j})_{i,j \in I} \in \R^{I \times I}$ with an infinite index set $I$, we define positive semidefinteness condition $M \succeq 0$ as the validity of $(m_{i,j} )_{i,j\in B}$ for every finite $B \subseteq I$. We use the following notation: if $y \in \R^n$ and $B \subseteq \N^n$, then for $A \subseteq B$ the vector $y_A$ is the projection of $y$ on the coordinates indexed by $A$: $y_A = (y_\alpha)_{\alpha \in A} \in \R^A$. The set $Y_A$ denotes the projection of a set $Y$ on the $A$ coordinates, i.e., $Y_A  = \{ y_A \colon y \in Y \}$.

\begin{thm}[Dual Putinar] \label{dual:putinar} 
	In the notation of Theorem~\ref{putinar:thm}, for every finite set $A \subseteq \N^n \setminus \{0\}$, one has 
	\begin{equation} 
		\cM_A(X) = \left\{ v \in \R^{\N^n} \colon v_0 = 1, \ L_v(g_i M_{\N^n} ) \succeq 0 \ \text{for all} \ i=0,\ldots,s+1 \right\}_A. \label{dual:putinar:eq} 
	\end{equation} 
\end{thm} 
\begin{proof} 
	The set $\cM_A(X)$ is the closed convex hull of $\{ (p^\alpha)_{\alpha \in A} \colon p \in X\}$. Hence, for showing the part ``$\subseteq$'' of the set equality \eqref{dual:putinar:eq}, it is enough to check that each $w = (p^\alpha)_{\alpha \in A}$ with $p \in X$ belongs to the right-hand side of \eqref{dual:putinar:eq}. For $v = (p^\alpha)_{\alpha \in \N^n}$ we have $v_A = w$, $v_0 =1$ and $L_v( g_i M_{\N^n} ) = g_i(p) M_{\N^n}(p)$, where $g_i(p) \ge 0$ and $M_{\N^n}(p) = v v^\top$ is an infinite rank-one positive semidefinite matrix. This shows that $w$ belongs to the right-hand side of \eqref{dual:putinar:eq}.
	
	To conclude the proof of \eqref{dual:putinar:eq}, we consider an arbitrary $w \in \R^A \setminus \cM_A(X)$ and show that such $w$ does not belong to the right-hand side of \eqref{dual:putinar:eq}. By separation theorem, there exists $c \in \R^{A \cup \{0\}}$ such that $\sum_{\alpha \in A } w_\alpha c_\alpha + c_0 < 0$ and $\sum_{\alpha \in A} u_\alpha c_\alpha + c_0 \ge 0$ for every $u \in \cM_A(X)$. Taking $u_\alpha = p^\alpha$ with $p \in X$, we see that the polynomial $\sum_{\alpha \in A} x^\alpha c_\alpha + c_0$ is non-negative on $X$. Fix $\epsilon>0$ small enough to satisfy $\sum_{\alpha \in A} w_\alpha c_\alpha + c_0 + \epsilon < 0$. Since $f = \sum_{\alpha \in A} c_\alpha x^\alpha  + c_0 + \epsilon$ is strictly positive on $X$, by Theorem~\ref{putinar:thm} we conclude that $f = g_0 \sigma_0 + \cdots + g_{s+1} \sigma_{s+1}$, where $\sigma_0,\ldots,\sigma_{s+1} \in \Sigma$. Consider any vector $v \in \R^{\N^n}$ satisfying $v_A = w$. We have 
	\[
		0 > \sum_{\alpha \in A} c_\alpha w_\alpha + c_0 + \epsilon  = L_v(f) = \sum_{i=0}^s L_v (g_i \sigma_i). 
	\]
	Hence $L_v(g_i \sigma_i) < 0$ for some $i \in \{0,\ldots,s\}$. Since $\sigma_i$ is a sum of squares, for some of these squares, which we denote as $q^2$, one has $L_v(g_i q^2) < 0$. Let $q$ be given as $q = s^\top x^B$ for some finite $B \subseteq \N^n$ and some $s \in \R^B$. Then $g_i q^2 = s^\top g_i M_B s$ and by this, $L_v (g_i q^2) = s^\top L_v( g_i M_B) s <0$. The latter shows that $g_i M_B$ is not positive semidefinite, which implies that the condition $g_i M_{\N^n} \succeq 0$ is violated. Consequently, $w$ is not in the right-hand side of \eqref{dual:putinar:eq}. 
\end{proof} 

\begin{rem} 
A natural question arises about the meaning of the set on the right-hand side of \eqref{dual:putinar:eq}, of which we take the projection onto $\R^A$. Or in other words, what would be the case $A = \N^n$ of Theorem~\ref{dual:putinar}? On the left-hand side as $\cM_{\N^n}(X)$ one can expect some generalization of the closed convex hull operation for the set  $\{ x^{\N^n} \colon x \in X\}$. It turns out that the correct generalization is the system of ``complete moment vectors'' with respect to a probability measure concentrated on $X$. Such a measure-theoretic generalization of Theorem~\ref{dual:putinar} is nice to have, but it comes with the price of being more technical, as measure-theoretic technicalities have to be addressed to properly define $\cM_{\N^n}(X)$.  We would like to stress that  Theorem~\ref{dual:putinar}, which does not involve any measure-theoretic machinery, is already enough to establish the convergence of the hierarchy of moment relaxations as outlined in the following section.
\end{rem} 

\subsubsection{Semidefinite hierarchies of moment relaxations}

We can establish sequences of approximations of polynomial optimization problems. The approximations are based on either linear or semidefinite convex problems and have convergence guarantees from the theorems of the previous section. We discuss these hierarchies and approximations with sparsity in mind. 

Consider the setting of Theorem~\ref{dual:putinar}, assuming that we know an $A_i$ with $g_i \in \R[x]_{A_i}$ for $i=0,\ldots,s+1$.  Each constrained $L_v( g_i M_{\N^n}) \succeq 0$ can be ``cropped'' to a constraint $L_v(g_i M_{B_i}) \succeq 0$, by fixing a finite set $B_i \subseteq \N^n$. This an LMI of a finite size $|B_i|$. In order to take into consideration of all the monomial variables $v_\alpha$ with $\alpha \in A$ in constraints it suffices to have $B_0 + B_0 \supseteq A$, since the constraint $L_v( g_0 M_{B_0} )  = L_v(M_{B_0}) \succeq 0$ involves the variables $v_\beta$ with $\beta \in B_0 + B_0$. Hence, each choice of $B_0,\ldots,B_{s+1}$ produces an outer approximation $\MR(B_0,\ldots,B_{s+1})_A$ of $\cM_B(X)$, where with $P := \bigcup_{i=0}^{s+1} (A_i + B_i)$ one has
\[
	\MR(B_0,\ldots,B_{s+1}) := \{ v \in \R^{P} \colon L_v( g_i M_{B_i}) \succeq 0 \ \text{for all} \ i =0,\ldots,s+1\}.
\]

That is, $\MR(B_0,\ldots,B_{s+1})_A$ is a projected spectrahedron in $\R^A$ that approximates the moment body $\cM_A(X)$ and is given as a projection of a spectrahedron  in $\R^P$. 
When all of the $B_i$ converge to $\N^n$, then $\MR(B_0,\ldots,B_{s+1})$ converges to $\cM_A(X)$ in the so-called Hausdorff metric. We recall that for compact sets $K,L$, the Hausdorff metric is defined by $\delta(K,L) = \max \{ \max_{q \in L} \min_{p \in K} \| p - q\|, \max_{p \in K} \min_{q \in L} \| p -q\|\}$. When $K \subseteq L$, this is the minimal $\rho \ge 0$ such that every point of $K$ is at distance at most $\rho$ to some point of $L$. 
\begin{cor} \label{putinar:approx:thm} 
	In the notation of Theorem~\ref{dual:putinar}, if $(B_{0,t})_{t \in \N}$ with $i=0,\ldots,s+1$ are sequences of finite subsets of $\N^n$ that are increasing in $t$ with respect to inclusion and satisfying $\bigcup_{t \in \N} B_{i,t} = \N^n$,  then 
	\[
		\lim_{t \to \infty} \MR(B_{0,t}, \ldots, B_{s+1,t})_A = \cM_A(X)
	\]
	holds in the sense of the Hausdorff metric. 
\end{cor} 
\begin{proof} 
	Using $\bigcup_{t \in \N} B_{i,t} = \N^n$ and Theorem~\ref{dual:putinar}, we obtain 
	\begin{equation} \label{intersection:seq} 
		\bigcap_{t \in \N} \MR(B_{0,t},\ldots,B_{s+1,t} )_A  = \cM_A(X). 
	\end{equation} 
	Since $B_{i,t}$'s are increasing in $t$ with respect to inclusion, the set  $\MR(B_{0,t},\ldots,B_{s+1,t} )_A$ is decreasing in $t$ with respect to inclusion. Thus, applying Lemma~1.8.2 from \cite{schneider2014} to \eqref{intersection:seq}, we get the assertion. 
\end{proof} 

\begin{rem} 
Lasserre introduced the canonical choice $B_{i,t} = \N_{2 (d-d_i + t)  }^n$ in the case, where $A = \N_{2 d}^n$, $B_i = \N_{2 d_i}$ and $d \ge d_i$. For this choice $A_i + B_{i,t} = \N_{ 2 (d+t)}^n$ so that the spectrahedron $\MR(B_{0,t},\ldots,B_{s+1,t} )$ is a subset of $\R^{\N_{d+t}^n}$. 
\end{rem} 

\begin{rem}
Theorem~\ref{putinar:approx:thm} alludes to an iterative method for solving the problems $\inf_{x \in X} f(x)$, where $f \in \R[x]_A$. Such a method  would iteratively increment $t$ and solve a semidefinite optimization problem on $R_t:=\MR(B_{0,t},\ldots,B_{s+1,t})$ in each iteration in order to provide an improved lower bound on $\inf_{x \in X} f(x)$, which would converge to the true value as $t \to \infty$.  However, in practice, such a method would be extremely impractical. Traditionally, one uses interior-point methods for solving SDPs which are difficult to warm-start. In addition, the size of the LMIs describing $R_t$ get larger (and quite dramatically so for large $n$) as $t$ increases, leading to increased run times as discussed in Section~\ref{sec_paradigms}. 
Instead of practically using Theorem~\ref{putinar:approx:thm}, one could treat it as a framework result showing that, in principle, an arbitrarily precise approximation of $\inf_{x \in X} f(x)$ is possible, albeit usually at a very high cost. The fact that the price to be paid for accuracy is very high is  not very surprising, because solving polynomial optimization problems to global optimality is known to be computationally hard. 

The practical conclusion we make here is that we should choose sets $B_0,\ldots,B_{s+1}$ that define the approximation $\MR(B_0,\ldots,B_{s+1})_A$ carefully. They should not be too large and the structure of the particular problem $\inf_{x \in X} f(x)$ should be exploited, such that the resulting SDPs can be solved in reasonable time. 
\end{rem} 

\begin{rem}
A natural modification of the relaxation $R(B_0,\ldots,B_{s+1})_A$ is obtained by replacing the semidefinite constraint $L_v(g_i M_{B_i}) \succeq 0$ with a family of constraints $L_v(g_i M_B)$ where $B$ belongs to a given family $\cB_i$ of finite subsets of $\N^n$. In this way, one large $B_i$ can be replaced by smaller $B$'s taken from $\cB_i$. This new relaxation might provide a looser approximation, but it it may become more manageable from a computational perspective. The relaxation is $\MR(\cB_0,\ldots,\cB_{s+1})_A$, where 
	\[
		\MR(\cB_0,\ldots,\cB_{s+1}) := \{ v \in \R^P \colon L_v(g_i M_B) \succeq 0  \ \text{for all} \ i=0,\ldots,s+1,  \ \text{and} \ B \in \cB_i\}. 
	\]
\end{rem} 

\subsection{Linear programming models} 

Now we look at different convexification approaches that result in linear programming models. Note that we often assume specific structures $X = K$ of the feasible set such as $K$ being a polytope or $K=\Bx(l,u)$.

\subsubsection{Bound-factor relaxations} 

Another useful positivstellensatz goes back to Handelman \cite{handelman1988representing}.

\begin{thm}[Handelman] \label{handelman:thm} 
	Let $K$ be a non-empty polytope given as $K = \{ x \in \R^n \colon g_1(x) \ge 0,\ldots, g_s(x) \ge 0 \}$, where $ g = (g_1,\ldots,g_s)$ are polynomials of degree one. Then every polynomial $f \in \R[x]$ strictly positive on $K$ admits a representation 
	\[
	f= \sum_{\beta \in B} c_\beta g^{\beta},
	\]
	where $B \subseteq \N^s$ is a finite set and $c_\beta \ge 0$ for every $\beta \in B$. 
\end{thm} 

\begin{thm}[Dual Handelman] \label{dual:handelman}
	In the notation of Theorem~\ref{handelman:thm}, for every finite set $A \subseteq \N^n \setminus \{0\}$, on has 
	\begin{equation} \label{dual:handelman:eq}
	\cM_A(K) = \{ v \in \R^{\N^n} \colon v_0 = 1, \ L_v( g^\beta) \ge 0 \ \text{for all} \ \beta \in \N^s\}_A. 
	\end{equation} 
\end{thm} 

\begin{proof}
The proof is analogous to the proof of the dual Putinar theorem~\ref{dual:putinar}, but somewhat simpler, because it involves linear constraints rather than the semidefinite ones.
To show the inclusion ``$\subseteq$'', we pick $p$, $w$ and $v$ in the same way as in the proof of dual Putinar (note that $K$ in the dual Handelman corresponds to $X$ in the dual Putinar). Then $L_v(g^\beta) = g(v)^\beta \ge 0$, since $g_i(v) \ge 0$ for each $i \in [s]$ in view of the fact that $v$ belongs to $P$. This shows that $w$ is in the right-hand side of \eqref{dual:handelman:eq}.

To show the the inclusion ``$\supseteq$'', we assume that $w$ is in $\R^A \setminus \cM_A(K)$. Then, as in the proof of the dual Putinar, we fix a separating hyperplane determined by the vector $c$, introduce an $\epsilon>0$ such that the polynomial $f(x)$ (depending on $c$ and $\epsilon$) is strictly positive on $K$. By Handelman,
$f = \sum_{\beta \in B} d_\beta g^\beta$, where $B$ is a finite subset of $\N^s$ and $d_\beta$ are non-negative coefficients. 

For every $v \in \N^n$ with $v_B = w$, we have 
\begin{align*}
0 &> \sum_{\alpha \in A} c_\alpha w_\alpha + \epsilon \\
& = L_v(f) = \sum_{\beta \in B} d_\eta L_v(g^\beta). 
\end{align*}

This implies that for some $\beta \in B$ we have $L_v(g^\beta)<0$. Consequently, $v$ is not in the right-hand side set of \eqref{dual:handelman:eq}.
\end{proof}

Here, we work under the assumptions of Theorem~\ref{handelman:thm}. Consider $0 \in A_i \subseteq \N_1^n$ such that $g_i \in \R[x]_{A_i}$. We can also fix $A_i = \N_1^n$, which is the case needed when the coefficients of $g_i$ are generic, but in more specific situations, we might prefer to take a smaller $A_i$. We introduce the polyhedral relaxation $\BF(B)$ of $\cM_A(K)$ as
\begin{equation} \label{eq_BF}	
\BF(B):= \{ v \in \R^P \colon L_v( g^\beta ) \ge 0  \ \text{for all} \ \beta \in B\}
\end{equation}
depends on the choice of a finite set $B \subseteq \N^s$ and 
\[
P := \bigcup_{\beta \in B} ( \underbrace{A_1 + \cdots + A_1}_{\beta_1} + \cdots + \underbrace{A_s + \cdots + A_s}_{\beta_s} ). 
\]
We call $\BF(B)$ the bound-factor polyhedron for $B$. 
Analogously to how we analyzed established semidefinite relaxations in Corollary~\ref{putinar:approx:thm}, one can show that $\BF(B)$ can approximate $\cM_A(K)$ arbitrarily well. In particular, $\lim_{d \to \infty} \BF(\N_d^s)_A = \cM_A(K)$ with the limit in the sense of the Hausdorff metric. 

\begin{rem}
In contrast to linear programming relaxations, approximation by semidefinite relaxations is achievable for a much broader class of moment bodies $\cM_A(X)$, where $X$ is not necessarily a polytope. In some cases (for example, for univariate problems) one can even guarantee an exact description of $\cM_A(X)$. In sharp contrast to this, an exact polyhedral description is impossible with linear programming relaxations if $\cM_A(X)$ is not a polyhedron. On the other hand, linear programming relaxations are better suited for applications in branch-and-bound frameworks, because the dual simplex method can start in phase 2 with the solution of the parent node. 
\end{rem} 

\subsubsection{Exact polyhedral models} \label{sec_polyhedralbodies}

Since global NLP heavily relies on linear programming, it uses patterns with polyhedral moment bodies. Below, we present some such patterns. 

\begin{prop} 
	\label{prop:multilin}
	Let $I \subseteq \{0,1\}^n$ and $K=\Bx(l,u)$. Then the moment body $\cM_I(K)$ is a polytope given by 
	\(
	\cM_I(K) = \cM_I ( \vx(K)). 
	\)
\end{prop} 
\begin{proof} 
	Clearly, $\vx(K) = \{l_1,u_1\} \times \cdots \times \{l_n,u_n\}$. 
	By rearranging $x_i (u_i - l_i) = x_i (u_i - l_i)$, one has 
	\[
	x_i = \underbrace{ \frac{u_i - x_i}{u_i - l_i} }_{a_{i,l_i}}  l_i + \underbrace{ \frac{x_i - l_i}{u_i - l_i}}_{a_{i,u_i}} u_i. 
	\]
	with $a_{i,l_i}, a_{i,u_i} \ge 0$, for every $x \in K$, and $a_{i,l_i} + a_{i,u_i} = 1$. Representing $x_1,\ldots,x_n$ as above and using the fact that $m(x_1,\ldots,x_n) := x^I$ is an affine function in each of the $x_i$'s, we obtain 
	\[
	m(x)   = \sum_{s_1 \in \{l_1,u_1\},\ldots, s_n \in \{l_n,u_n\} } a_{1,s_1} \cdots a_{n,s_n} \, m(s_1,\ldots,s_n),
	\]
	where all $2^n$ products $a_{1,s_1} \cdots a_{1,s_n}$ are non-negative, for $x \in K$, and 
	\[
	\sum_{s_1 \in \{l_1,u_1\},\ldots, s_n \in \{l_n,u_n\} } a_{1,s_1} \cdots a_{n,s_n} = (a_{1,l_1} + a_{1,u_1}) \cdot \ldots \cdot (a_{n,l_n} + a_{n,u_u}) = 1. 
	\] 
	This shows $\cM_I(K) = \cM_I(\vx(K))$, which immediately implies that $\cM_I(K)$ is a polytope. 
\end{proof} 

\begin{rem}
	\label{rem:multilin} 
	Since $\cM_I(K) = \cM_I(\vx(K))$, the condition in Proposition~\ref{prop:multilin} that $v \in \cM_I(K)$ can be formulated using the linear constraints 
	\begin{align*}
		v &  = \sum_{p \in \vx(K)} \lambda_p \, m(p),
		\\ 1 & = \sum_{p \in \vx(K)} \lambda_p
	\end{align*} 
	with the non-negative auxiliary variables $\lambda_p \ge 0$.
\end{rem} 

\begin{rem} \label{mccormick} 
	The rectangle $K = [l_1,u_1] \times [l_2,u_2]$ is given by $g = (g_1,g_2,g_3,g_4) = ( x_1 -l_1, u_1 -x_1, x_2 - l_2, u_2 -x_2)$ with $g_i \in \R[x]_{A_i}$ and $A_1 = A_2 = \{(0,0), (1,0)\}$ and $A_3 = A_4 = \{(0,0),(0,1)\}$. In this case, it is known that for 
	\[
	B = \{ (1,0), (0,1)\}^2 = \{ (1,0,1,0), (1,0,0,1), (0,1,1,0), (0,1,0,1) \},
	\]
	the bound-factor polyhedron $\BF(B)$ coincides with the moment body $\cM_A(K)$ for $A = \{ 0,1\}^2$. The four inequalities describing $\BF(B)$ are $L_v( g_1 g_3) \ge 0, L_v(g_1 g_4) \ge 0$, $L_v(g_2 g_3) \ge 0$, and $l_v(g_3 g_4) \ge 0$. They are known as McCormick inequalities. 
\end{rem} 


\begin{ex} 
	Consider  $\cM_{\{0,1\}^2}(K)$ for $K = [-1,1]^2$. By Proposition~\ref{prop:multilin}, this moment body is a polytope with the vertices $x^{\{0,1\}^2}$, where $x \in \{-1,1\}^2$, so that $v \in \cM_{\{0,1\}^2}$ can be formulated as the constraint 
	\[
	\begin{pmatrix} 
		1 \\ v_{1,0} \\ v_{0,1} \\ v_{1,1} 
	\end{pmatrix}  = \lambda_1  \begin{pmatrix} 1 \\ -1 \\ -1 \\ 1\end{pmatrix} + \lambda_2 \begin{pmatrix} 1 \\ 1 \\ -1 \\ -1\end{pmatrix} + \lambda_3  \begin{pmatrix} 1 \\ -1 \\ 1 \\ -1\end{pmatrix} + \lambda_4 \begin{pmatrix} 1 \\ 1 \\ 1 \\ 1\end{pmatrix}
	\]
	using auxiliary non-negative variables $\lambda_1,\ldots,\lambda_4 \ge 0$. Using McCormick's inequalities we can also formulate $\cM_{\{0,1\}^2}(K)$ in the original space by 
	\begin{align*} 
		L_v( (1-x_1) (1-x_2) ) := 1 - v_{1,0} - v_{0,1} + v_{1,1} & \ge 0, 
		\\		L_v ( ( 1+ x_1) (1-x_2)) := 1 + v_{1,0} -v_{0,1} - v_{1,1} & \ge 0, 
		\\		L_v ( (1-x_1) (1 + x_2)) := 1 - v_{1,0} + v_{0,1} - v_{1,1} & \ge 0, 
		\\		L_v (1 + x_1)(1+x_2)) := 1+ v_{1,0} + v_{0,1} + v_{1,1} & \ge 0. 
	\end{align*} 
\end{ex} 

As a next step, we generalize Proposition~\ref{prop:multilin} by passing from $I = \{0,1\}^n$ to a more general set $P$: 

\begin{prop} 
	\label{prop:multilin:2} 
	Let $\alpha \in \N^n$, $P \subseteq \{0,\alpha_1\} \times \cdots \times \{0,\alpha_n\}$, and $K=\Bx(l,u)$. Then the moment body $\cM_P(K)$ is a polytope. 
\end{prop} 
\begin{proof} 
	There is no loss of generality in assuming $\alpha_i>0$ for every $i \in [n]$. 
	We can describe $P$ as  $P = \Gamma I$, for some  $I \subseteq \{0,1\}^n$ and the diagonal matrix  $\Gamma \in \N^{n \times n}$ with the diagonal entries $\alpha_1,\ldots,\alpha_n$. By Proposition~\ref{transf:patterns} from Section~\ref{deriving:from:mon:subs}, one has $\cM_P(K) = \cM_{\Gamma I} (K) = \cM_I({K^\Gamma})$, where $K^\Gamma = \Bx(l_\Gamma,u_\Gamma)$. By Proposition~\ref{prop:multilin}, $\cM_I(K^\Gamma)$ is a polytope. 
\end{proof} 

\begin{rem}
	\label{rem:multilin:2} 
	Remark~\ref{rem:multilin} and the proof of Proposition~\ref{prop:multilin:2} yield a way to provide an explicit formulation of $\cM_P(K)$ with $P \subseteq \{0,\alpha_1\} \times \cdots \times \{0,\alpha_n\}$. 
\end{rem} 

Motivated by Proposition~\ref{prop:multilin:2}, we call a pattern $P \subseteq \{0,\alpha_1 \} \times \cdots \times \{0,\alpha_n\}$ mutlilinear. 

We also generalize Proposition~\ref{prop:multilin:2} as follows. 
\begin{prop} 
	\label{prop:multilin:3} 
	Let $ \Gamma = (\gamma(1),\ldots,\gamma(k)) \in \R^{n \times k}$ be a matrix, for which the monomials $x^{\gamma(i)}$ with $i \in [k]$ are variable-independent. Then for any $P \subseteq \Gamma \{0,1\}^k$, the moment body $\cM_P(K)$ is a polytope. 
\end{prop} 
\begin{proof} 
	This is a direct consequence of Propositions~\ref{transf:patterns} and \ref{prop:multilin}. 
\end{proof} 


\subsubsection*{Patterns from expression trees} \label{sec_expression}

We describe how expression-tree convexifications from global NLP can be phrased in terms of patterns, when applied to \eqref{POP}. Each algebraic expression is made up of a certain set of elementary operations, such as powers, linear combinations, or products of expressions. A decomposition of an algebraic expression into these operations can be visualized using an algebraic expression tree. This is a rooted tree with nodes labeled by terms occurring in the expression. Each term is built up from its child terms using elementary operations. The underlying convexification is obtained by introducing a variable for each sub-expression occurring in the tree and providing convex constraints that link the variables of every node and its child nodes. A polynomial $f = \sum_{\alpha \in A} f_\alpha x^\alpha$ can thus be expressed as a linear combination of monomials, where each monomial $x^\alpha$ is expressed as a product of powers of the variables. If we do not want to use products with arbitrary number of factors as elementary operations, we can fix a bracketing in the product  $x_1^{\alpha_1} \cdots x_n^{\alpha_n}$, to specify the order in which multiplications are carried out. 

So, if $f = 2 x_1^2 x_3 - 3  x_1 (x_2 x_3^4) + 7 x_1 (x_2 x_3)$, we have the corresponding tree

\medskip
\tikzstyle{TreeNode} = [draw, rounded corners, align=center, minimum height = 18pt]
\begin{center} 
	\begin{tikzpicture}{scale=0.5}
		\node[TreeNode] (f) at (4,3) {$f =2 x_1^2 x_3 - 3  x_1 (x_2 x_3^4) + 7 x_1 (x_2 x_3) $} ;
		\node[TreeNode] (m1) at (0,2) { $x_1^2 x_3$ } ;
		\node[TreeNode] (m2) at (4,2) { $x_1 (x_2 x_3^4)$ } ;
		\node[TreeNode] (m3) at (8,2) { $x_1 (x_2 x_3)$ } ;
		\node[TreeNode] (a) at (-1,1) { $x_1^2$ } ;
		\node[TreeNode] (b) at (1,1) { $x_3$ } ;
		\node[TreeNode] (c) at (3,1) { $x_1$ } ;
		\node[TreeNode] (d) at (5,1) { $x_2 x_3^4$ } ;
		\node[TreeNode] (e) at (7,1) { $x_1$ } ;
		\node[TreeNode] (g) at (9,1) { $x_2 x_3$ } ;
		\node[TreeNode] (h) at (-1,0) { $x_1$ } ;
		\node[TreeNode] (i) at (4,0) { $x_2$ } ;
		\node[TreeNode] (j) at (6,0) { $x_3^4$ } ;
		\node[TreeNode] (k) at (8,0) { $x_2$ } ;
		\node[TreeNode] (l) at (10,0) { $x_3$ } ;		 
		\node[TreeNode] (m) at (6,-1) { $x_3$ } ;		 
		\draw[->,thick] (f) -- (m1); 
		\draw[->,thick] (f) -- (m2);
		\draw[->,thick] (f) -- (m3);
		\draw[->,thick] (m1) -- (a); 
		\draw[->,thick] (m1) -- (b); 
		\draw[->,thick] (m2) -- (c); 
		\draw[->,thick] (m2) -- (d); 
		\draw[->,thick] (m3) -- (e); 
		\draw[->,thick] (m3) -- (g); 
		\draw[->,thick] (a) -- (h);
		\draw[->,thick] (d) -- (i);		 
		\draw[->,thick] (d) -- (j);
		\draw[->,thick] (g) -- (k);
		\draw[->,thick] (g) -- (l);
		\draw[->,thick] (j) -- (m);
	\end{tikzpicture} 
\end{center} 
We assign monomial variables to all expressions apart from the root node $f$. The convexification for the root node $f$ and the variables $v_{2,1,0}$, $v_{1,1,4}$, and $v_{1,1,1}$ assigned to its children is $f = 2 v_{2,1,0} - 3 v_{1,1,4} + 7 v_{1,1,1}$. Every non-root node of a tree gets assigned a monomial variable. Taking a non-root node together with its children determines a pattern. For example, the node $x_1 (x_2 x_3^4)$ and its children form the pattern $\{ (1,1,4), (1,0,0), (0,1,4)\}$, while the node $x_1^2$ and its only child $x_1$ determine the pattern $\{ (2,0,0), (1,0,0)\}$. As a consequence of Proposition~\ref{prop:multilin:3}, we see that for three-element patterns $P$ the respective moment body $\cM_P(K)$ is a polytope. For the two-element patterns $P = \{ (2,0,0) , (1,0,0)\}$ and $P = \{(0,0,4), (0,0,1)\}$ the respective moment bodies $\cM_P(K)$ are not polytopes, but they are two-dimensional bodies, for which an outer polyhedral approximation can be constructed rather easily. Hence, for this example, we establish a sparse polyhedral relaxation of \eqref{POP} using  \eqref{P-RLXa} with a pattern family arising from an expression tree. 

There is yet another possibility. If we allow products with arbitrary many factors as elementary operations, then each monomial $x^{\alpha}$ can be written as a product of powers of the variables. For the above example, we obtain the expression tree 
\begin{center} 
	\begin{tikzpicture}{scale=0.5}
		\node[TreeNode] (f) at (4,3) {$f =2 x_1^2 x_3 - 3  x_1 x_2 x_3^4 + 7 x_1 x_2 x_3 $} ;
		\node[TreeNode] (m1) at (0,2) { $x_1^2 x_3$ } ;
		\node[TreeNode] (m2) at (4,2) { $x_1 x_2 x_3^4$ } ;
		\node[TreeNode] (m3) at (8,2) { $x_1 x_2 x_3$ } ;
		\node[TreeNode] (a) at (-1,1) { $x_1^2$ } ;
		\node[TreeNode] (b) at (1,1) { $x_3$ } ;
		\node[TreeNode] (c) at (3,1) { $x_1$ } ;
		\node[TreeNode] (d) at (5,1) { $x_3^4$ } ;
		\node[TreeNode] (e) at (7,1) { $x_1$ } ;
		\node[TreeNode] (g) at (9,1) { $x_3$ } ;
		\node[TreeNode] (h) at (-1,0) { $x_1$ } ;
		\node[TreeNode] (i) at (4,1) { $x_2$ } ;
		\node[TreeNode] (j) at (5,0) { $x_3$ } ;
		\node[TreeNode] (k) at (8,1) { $x_2$ } ;
		\draw[->,thick] (f) -- (m1); 
		\draw[->,thick] (f) -- (m2);
		\draw[->,thick] (f) -- (m3);
		\draw[->,thick] (m1) -- (a); 
		\draw[->,thick] (m1) -- (b); 
		\draw[->,thick] (m2) -- (c); 
		\draw[->,thick] (m2) -- (d); 
		\draw[->,thick] (m2) -- (i); 
		\draw[->,thick] (m3) -- (e); 
		\draw[->,thick] (m3) -- (g); 
		\draw[->,thick] (a) -- (h);
		\draw[->,thick] (d) -- (j);
		\draw[->,thick] (m3) -- (k);
	\end{tikzpicture} 
\end{center} 
The difference in constructing patterns for this tree is that some of the nodes determine patterns of the form $\{ \alpha_1 e_1,\ldots, \alpha_n e_n , \alpha \}$. For example, the node $x_1 x_2 x_3^4$ together with its children defines the pattern $ \{ (1,0,0), (0,1,0), (0,0,4), (1,1,4)\}$. For such patterns, Proposition~\ref{prop:multilin:2} guarantees that the respective moment body $\cM_P(K)$ is a polytope. Hence, the comments for the previous expression tree still apply. 

\subsection{Sparsifying dense approaches} 
\label{deriving:from:mon:subs} 

We want to carry over Lasserre's semidefinite relaxations of the moment body $\cM_{\N_d^n}(X)$ body to a relaxation of $\cM_P(X)$ for some smaller sets $P$, which can be used as patterns in \eqref{P-RLXa}. 

For a pattern $P \subseteq \N^n$, we already introduced the notation $x^P := ( x^\alpha)_{\alpha \in P}$. Given a matrix $\Gamma = (\gamma(1),\ldots,\gamma(k)) \in \N^{n \times k}$, we use the similar notation with $x^{\Gamma} := ( x^{\gamma(1)},\ldots, x^{\gamma(k)})$ along with the vectors $l_\Gamma := ( l_{\gamma(1)},\ldots, l_{\gamma(k)})$ and $u_\Gamma := (u_{\gamma(1)},\ldots,u_{\gamma(k)})$. 
For a set $X$, we introduce the set $X^\Gamma := \{ x^\Gamma \colon x \in X\}$,  which is the image of $X$ under the map $x \mapsto x^\Gamma$. 

Assuming that $\Gamma$ has rank $k$, one can now define the pattern $\Gamma P := \{ \Gamma \alpha \colon \alpha \in P\}$, which we call the image of $P$ under $\Gamma$. The exponents in  $P$ are in one-to-one correspondence with the exponents in $\Gamma P$, since $\Gamma$ has rank $k$. 
We now map the $k$-dimensional patterns $\N_d^k$ to patterns $\Gamma \N_d^k$. Clearly, one can generate many copies of such patterns by considering many different choices of $\Gamma$, and then use such patterns within \eqref{P-RLXa}. 

We call monomials $x^{\gamma(1)},\ldots, x^{\gamma(k)}$ variable-independent if no variable $x_i$ occurs in more than one of these $k$ monomials. This means the supports of the exponent vectors $\gamma(1),\ldots,\gamma(k)$ are pairwise disjoint. 

\begin{prop} \label{transf:patterns} 
	Let $\Gamma = (\gamma(1),\ldots,\gamma(k)) \in \N^{n \times k}$ be a rank $k$ matrix and consider a pattern $P \subseteq \N^k$. Then for the moment body of the pattern $\Gamma P$, one has 
	\[
	\cM_{\Gamma P}(X) = \cM_{P} (X^\Gamma). 
	\]
	If $X = K = \Bx(l,u)$ and if $x^{\gamma(1)},\ldots, x^{\gamma(k)}$ are variable-independent monomials, then $K^\Gamma = \Bx(l_\Gamma,u_\Gamma)$.
\end{prop} 
\begin{proof} 
	Since $\Gamma$ has rank $k$, the map $\alpha \mapsto \Gamma \alpha$ from $P$ to $\Gamma P$ is bijective. The pattern  $P$ gives the system of monomials $y^P$ in the variables $y = (y_1,\ldots,y_k)$, whereas the pattern $\Gamma P$ gives the system of monomials $x^{\Gamma P}$. One can obtain $x^{\Gamma P}$ by substituting $y = x^{\Gamma}$ into $y^P$. In the coordinate form, the substitution can be written as $y_i = x^{\gamma(i)}$. One has 
	\begin{align*}
		\cM_{\Gamma P}(X) & \, = \conv \{ ( (x^\Gamma)^\alpha )_{\alpha \in P} : x \in X \}  \\ & = \conv \{ ( y^\alpha)_{\alpha \in P} \colon y \in X^\Gamma \} \\ & = \cM_P(X^\Gamma). 
	\end{align*} 
	Assume now that $X = K = \Bx(l,u)$ and that the  monomials $x^{\gamma(1)}, \ldots, x^{\gamma(k)}$ are variable-independent. In this case, by letting $x$ vary in  $K=\Bx(l,u)$, we can fix the values of the $k$ monomials $x^{\gamma(i)}$ independently. Since $x^{\gamma(i)}$ takes values in $[l_{\gamma(i)},u_{\gamma(i)}]$, we obtain  $K^\Gamma = \Bx(l_\Gamma,u_\Gamma)$. 
\end{proof} 

If  $\Gamma = (\gamma(1),\ldots,\gamma(k)) \in \N^{n \times k}$ is such that the monomials $x^{\gamma(1)},\ldots,x^{\gamma(k)}$ are variable-independent, then we call the pattern $\Gamma \N_d^k$ a truncated sub-monoid pattern. If $k=1$, then we call $\Gamma \N_d^k$  a chain. Thus, a chain is a pattern of the form $\{0,\gamma, \ldots, d \gamma \}$, where $\gamma \in \N^n$. 

Applying Proposition~\ref{transf:patterns} to the moment body $\cM_{\Gamma \N_d^k}(X)$ of a truncated submonoid pattern, we can transfer the relaxation techniques for $\cM_{\N_d^n}(X)$ to the body $\cM_{\Gamma \N_d^k}(X)$. In principle, both bound-factor relaxations and Lasserre's semidefinite relaxation can be chosen. In Section~\ref{sec_computations} we shall use Lasserre's relaxation, thus we only spell out the resulting Lasserre-type moment relaxation of $\cM_{\Gamma \N_{2 d}^k}(X)$ (we replace $d$ by $2d$, since we work with even degrees in Lasserre's approach). For box constraints we have $K^\Gamma = \Bx(l_\Gamma, u_\Gamma)$. The Lasserre-type relaxation of $M_{\Gamma \N_{2 d}^k}$ is the following set 
\begin{align*}
	\MR_{\Gamma,d}(K)  := \{ v \in \R^{\Gamma \N_{2d}^k} \colon 
	& L_v (M_{\Gamma \N_d^k} ) \succeq 0 \ \text{and} 
	\\ & L_v(h_i M_{ \Gamma \N_{d-1}^k} ) \succeq 0 \ \text{for all} \ i \in [k] \}
\end{align*}
where $h_i := (x^{\gamma(i)} - l_{\gamma(i)} ) (u_{\gamma(i)} - x^{\gamma(i)})$.
We call $\MR_{\Gamma,d}(K)$ the semidefinite relaxation of $\cM_{\Gamma \N_k^d}(K)$ and use it in our computations.  

\begin{ex}
	Consider the box $K = [-1,1] \times [-2,2]$. We choose $d=1$ and $\Gamma = \begin{pmatrix} 2 & 0 \\ 0 & 2 \end{pmatrix}$. Then $\Gamma \N_{2 d}^n = \Gamma \N_2^2 = \{ (0,0), (2,0), (0,2), (4,0), (2,2), (0,4)\}$ is the pattern. The monomials $x_1^2$ and $x_2^2$ range in the segments $[0,1]$ and $[0,4]$ for $x \in K$. That is, $l_{2,0} = 0, u_{2,0} = 1$, $l_{0,2} = 0, u_{0,2} = 4$. Consequently, $\MR_{\Gamma,d}(K)$ is given by the LMIs
	\begin{align*} 
		L_v( \begin{pmatrix} 1 & x_1^2 & x_2^2 
			\\ x_1^2 &  x_1^4 & x_1^2 x_2^2 
			\\ x_2^2 & x_1^2 x_2^2 & x_2^4
		\end{pmatrix} ) := \begin{pmatrix} 1 & v_{2,0} & v_{0,2} 
			\\ v_{2,0} & v_{4,0} & v_{2,2} 
			\\ v_{0,2} & v_{2,2} & v_{0,4} 
		\end{pmatrix} & \succeq 0, 
		\\  L_v (x_1^2 (1 - x_1^2) ) := v_{2,0} - v_{4,0} & \ge 0, 
		\\ L_v ( x_2^2 ( 4 - x_2^2) ) := 4 v_{0,2} - v_{0,4} & \ge 0.  
	\end{align*} 
\end{ex} 

\begin{ex} 
	Consider the box $K = [-1,2] \times [1,2]$ and the chain pattern $P= \{ 0, \gamma, 2 \gamma \}$ with $\gamma = (1,2)$. The Lasserre-type relaxation for this choice is $\MR_{\Gamma,d}(K)$ with $\Gamma = \begin{pmatrix} 1 \\ 2\end{pmatrix}$ and $d = 1$. One has $x^\gamma = x_1 x_2^2$. When $x \in K$, one has $-1 \le x_1 \le 2$ and $1 \le x_2^2 \le 4$. 
	It follows that the least  value for $x^\gamma = x_1 x_2^2$ is attained for $x_1 = -1, x_2^2 = 4$ and the largest value for $x^\gamma$ is attained for $x_1 = 2, x_2^2 =4$. That is, $l_\gamma = -4$ and $u_\gamma = 8$.
	Hence, $\MR_{\Gamma,d}(K)$ is given by the LMIs
	\begin{align*} 
		L_v (\begin{pmatrix} 1 & x_1 x_2^2 
			\\ x_1 x_2^2 & x_1^2 x_2^4 
		\end{pmatrix} ) :=  \begin{pmatrix} 1 & v_{1,2} \\ v_{1,2} & v_{2,4} \end{pmatrix} & \succeq 0, 
		\\ L_v ( (x_1 x_2^2 + 4)(8 - x_1 x_2^2) ) := 32 + 4 v_{1,2} + v_{2,4} & \ge 0. 
	\end{align*}  
\end{ex} 

\subsection{Shifting patterns} \label{sec_shifting}

For all considered patterns we can also construct a shifted version, which may be useful in computations. Note that if a pattern $P$ does not contain the zero exponent, we can pass from from $P$ to $P \cup \{0\}$ without any significant changes in the relaxation, because the constraint $v_{P \cup \{0\}} \in \cM_{P \cup \{0\}}(X)$ is equivalent to $v_P \in \cM_P(X)$, $v_0=1$. 

\begin{prop}
	\label{prop:shift} 
	Consider a pattern $P \subseteq \N^n$ with $0 \in P$
	and an exponent vector $\eta \in \N^n \setminus \{0\}$ such that the two monomials $x^{\eta}$ and $x^{\alpha}$ are variable-independent for each choice of $\alpha \in P$. Then the pattern $\eta + P$ is given by 
	\[
	\cM_{\eta+P}(X) = \{   ( v_{\eta + \beta} )_{\beta \in P} \in   \cone \cM_P(X) \colon  l_\eta \le v_\eta \le u_\eta \}. 
	\]
\end{prop} 
\begin{proof} 
	To clarify the asserted equality, we note that although  $\cone \cM_P(X)$ is a subset of $\R^{P}$ and $\cM_{\eta+P}(X)$ is a subset of $\R^{\eta +P}$, we can identify  $\R^{P}$ and $\R^{\eta + P}$ in view of the bijection $\beta \leftrightarrow \eta + \beta$ between the index sets $P$ and $\eta + P$. 
	
Due to variable-independence of the monomials $x^\eta$ and $x^\alpha$, the choice of values of $x^\eta$ is independent on the choice of $m(x):=(x^\alpha)_{\alpha \in P}$, as $x$ varies in $X$. Hence 
	\begin{align*} 
		\cM_{\eta + P} (X) &  = \conv \{ t x^\alpha : t \in [l_\eta,u_\eta] , x \in X \} 
		\\ & = \conv \{ t u : t \in [l_\eta ,u_\eta], u \in \cM_P(X) \}, 
		\\ & = \conv \{ t u : t \in \{l_\eta, u_\eta\}, u \in \cM_P(X) \}. 
	\end{align*} 
	Thus, $\cM_{\eta + P}(X)$ is the convex hull of two parallel cross-sections of $\cone \cM_P(X)$ at ``heights'' $l_\eta$ and $u_\eta$, respectively. 
	This readily gives the assertion. 
\end{proof} 

Proposition~\ref{prop:shift} gives a way to establish formulations of shifted patterns, in particular, shifted truncated sub-monoids and shifted chains. Since we employ the conic-hull operation, we need to homogenize constraints. Both linear and semidefinite constraints can be homogenized. We give a simple example, illustrating how shifting works with respect to linear constraints. 
\begin{ex} 
	Consider the pattern $\{0,1\}^2 \times \{1\}$, which is a shift of $P = \{0,1\}^2 \times \{1\}$ by the vector $\eta = (0,0,1)$ and let us choose $X = [0,1]^2 \times [1,2]$. The moment body $\cM_P(X)$ is a polytope defined by the McCormick inequalities $L_v( x_1 x_2) := v_{1,1,0} \ge 0,$ $L_v(x_1 (1-x_2)) := v_{1,0,0} - v_{1,1,0} \ge 0,$ $L_v( (1-x_1) x_2) := v_{0,1,0} - v_{1,1,0} \ge 0,$  $L_v ( (1-x_1) (1-x_2)) := 1 - v_{1,0,0} - v_{0,1,0} + v_{1,1,0} \ge 0$. When we homogenize, the $1$ in the last inequality gets replaced by $v_{0,0,0}$. Finally, we shift the indices by adding $\eta= (0,0,1)$ and add the constraint $1 \le v_{0,0,1} \le 2$, resulting in the linear inequalities 
	\begin{align*}
		v_{1,1,1} & \ge 0, 
		\\ v_{1,0,1} - v_{1,1,1} & \ge 0,
		\\ v_{0,1,1} - v_{1,1,1} & \ge 0, 
		\\ v_{0,0,1} - v_{1,0,1} - v_{0,1,1} + v_{1,1,1} & \ge 0,
		\\ 1 \le v_{0,0,1} & \le 2
	\end{align*} 
	that define $\cM_{ \{0,1\}^2 \times \{1\} } ([0,1]^2 \times [1,2])$. 
\end{ex} 

\begin{ex} 
	We give a small example that shows how to derive a semidefinite formulation of a  shifted chain. 
	
	Consider the shifted chain $\eta + P = \{ (0,1), (1,1), (2,1) , (3,1), (4,1)\}$ with  $\eta = (0,1)$ and $P = \{(0,0),(1,0),(2,0),(3,0),(4,0)\}$. We choose $X = [-1,1] \times [1,2]$. The moment body $\cM_P(X)$ is described by the two LMIs
	\begin{align*} 
		L_v ( \begin{pmatrix} 1 & x_1 & x_1^2
			\\ x_1 & x_1^2 & x_1^3
			\\ x_1^2 & x_1^3 & x_1^4 \end{pmatrix}  ) := \begin{pmatrix} 1 & v_{1,0} & v_{2,0} \\ v_{1,0} & v_{2,0} & v_{3,0} 
			\\ v_{2,0} & v_{3,0} & v_{4,0} \end{pmatrix} & \succeq 0, 
		\\ L_v ( (1-x_1^2) \begin{pmatrix} 1 & x_1 \\ x_1 & x_1^2 \end{pmatrix}  ) : = 
		\begin{pmatrix} 1 - v_{2,0} & v_{1,0} - v_{3,0} 
			\\ v_{1,0} -  v_{3,0} & v_{2,0} - v_{4,0} 
		\end{pmatrix} & \succeq 0 
	\end{align*} 
	Replacing $1$ by $v_{0,0}$, then shifting the indices by $(0,1)$ and introducing the constraint $1 \le v_{0,1} \le 2$, we arrive at the following description of $\cM_{\eta+ P}(X)$: 
	\begin{align*}
		\begin{pmatrix} v_{0,1} & v_{1,1} & v_{2,1} \\ v_{1,1} & v_{2,1} & v_{3,1} 
			\\ v_{2,1} & v_{3,1} & v_{4,1} \end{pmatrix} & \succeq 0, 
		\\ 
		\begin{pmatrix} v_{0,1} - v_{2,1} & v_{1,1} - v_{3,1} 
			\\ v_{1,1} -  v_{3,1} & v_{2,1} - v_{4,1} 
		\end{pmatrix} & \succeq 0,
		\\ 1 \le v_{0,1} & \le 2.
	\end{align*} 
\end{ex}

\subsection{Summary and Visualization of Patterns} \label{sec_visualization}

We have seen a variety of different patterns that can be used in our general convexification and sparsification framework.
It is possible to combine the above approaches to obtain a pattern family $F$. This can be done problem-specifically, by exploiting properties of the underlying exponent set $A$ to obtain a benefitial relation between computational cost and tightness of the relaxation. Examples of such customized pattern families combining multilinear patterns and shifted chains are visualized below and evaluated computationally in Section~\ref{sec_structuredresults}.

One additional advantage of the abstracted concept of monomial patterns is that it allows a general way of visualizing which monomials and which auxiliary lifted variables are involved.

\begin{ex}\label{rmk:pattern-plot}
We consider different exponent sets for $n=2$ for visualization,
\begin{align*}
A_1  &:= \{(0,0), (0,3), (3,0), (3,3)\}, \\
A_2  &:= \{(0,0), (1,1), (2,2), (3,3), (4,4), (5,5), (6,6)\}, \\
A_3  &:= \{(0,0), (0,3), (0,6), (2,0), (2,3), (4,0)\}, \\ 
A_4  &:= \{(4,0), (4,1), (4,2), (4,3), (4,4), (4,5)\}, \\ 
\Aex &:= \{(0,2), (1,1), (2,3), (2,4), (4,0), (5,5)\}.
\end{align*}
Figure~\ref{fig_patterns} shows a visualization for them and some of the patterns discussed above.
In addition, three custom pattern families $F^1$, $F^2$, and $F^3$ based on combinations of multilinear patterns, chains, and shifted chains applied to $\Aex$ are also visualized in the last row of Figure~\ref{fig_patterns}. They will be used for numerical results in Section~\ref{sec_structuredresults}.
\end{ex}

\begin{figure}[h!!!]
 \input{figures/fig-multilinear-pattern.tex}
 \input{figures/fig-tree-bf-moment.tex}
 \input{figures/fig-chain-pattern.tex}	
 \input{figures/fig-shifted-chain-pattern.tex}
 \input{figures/fig-intro-pattern-configs.tex}
 \caption{
Visualization of involved variables $v_\alpha$ in examplary patterns. The title of each subplot refers to the corresponding set $A$ from Example~\ref{rmk:pattern-plot}. Exponents $\alpha \in A \subset \N^2$ are depicted as red squares. Pattern-specific auxiliary exponents $\alpha \in \cup_i P_i$ are depicted as blue dots. Curves connect all exponents $\alpha \in P_i$. As a rule of thumb, the largest cardinality $| P_i |$ has a strong influence on the overall runtime to solve \eqref{P-RLXa}.\\
\textbf{First row:} multilinear patterns with $P = \{0,1\}^2$, Section~\ref{sec_polyhedralbodies}\\ 
\textbf{Second row:} expression tree, Section~\ref{sec_expression}, bound-factor product, Section~\ref{sec_polyhedralbodies}, and moment relaxation, Section~\ref{sec_sdp} \\
\textbf{Third row:} different truncated submonoids, Section~\ref{deriving:from:mon:subs}\\
\textbf{Fourth row:} different shifted chains, Section~\ref{sec_shifting} \\
\textbf{Fifth row:} Families $F^1$, $F^2$, $F^3$ combining multilinear and shifted chains patterns\\
}
\label{fig_patterns}
\end{figure}
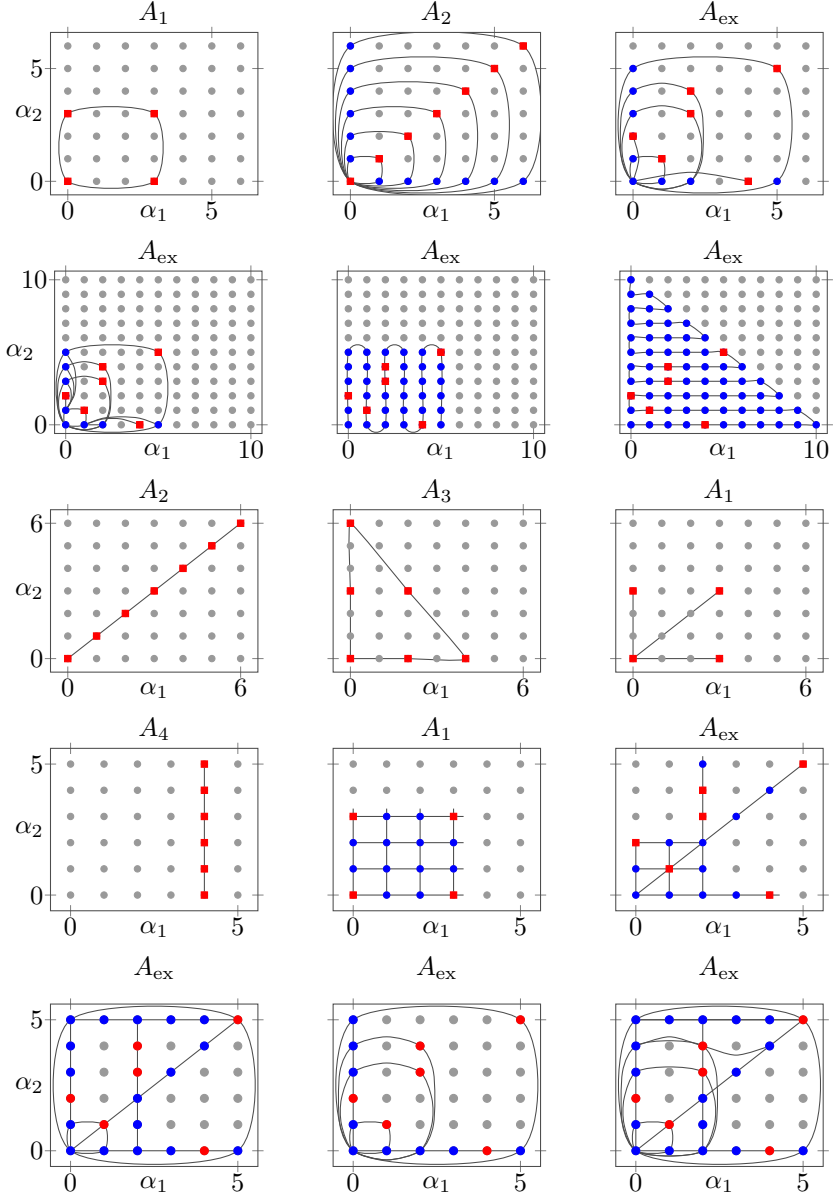
\clearpage

%% file: figures/fig-multilinear-pattern.tex
\begin{center}
\begin{tikzpicture}
\begin{groupplot}[
    axis line style = {black!70!white},
	group style     = {group size=3 by 1,
	xticklabels at  = edge bottom,
	yticklabels at  = edge left,
	ylabels at      = edge left,
	xlabels at      = edge bottom
	},
	width=\textwidth*(1/3),
	xtick              = {0,5},
	ytick              = {0,5},
	xmin               =-0.6,
	xmax               = 6.6,
    ymin               =-0.6,
    ymax               = 6.6,
    tick align         = center,
    xlabel             = $\alpha_1$,
	ylabel             = $\alpha_2$,
	y tick label style = {xshift=0.3em},
	x tick label style = {yshift=0.3em},
	ylabel style       = {rotate =-90, xshift= 1.2em},
	xlabel style       = {             yshift= 1.2em}
    ]
    \nextgroupplot[title={$\rmA_1$},title style={yshift=-1.5ex}]
    
    \foreach \x in {0,...,6}{
    \foreach \y in {0,...,6}{
    \addplot[only marks,mark size=1.2pt,black!40!white] coordinates {(\x,\y)} ;
    }}
    \addplot[only marks,mark size=1.2pt, mark=square*, red] coordinates {(3,3)} ;
    \addplot[only marks,mark size=1.2pt, mark=square*, red] coordinates {(3,0)} ;
    \addplot[only marks,mark size=1.2pt, mark=square*, red] coordinates {(0,3)} ;
    
    \addplot[only marks,mark size=1.2pt, mark=square*, red] coordinates {(0,0)} ;
    
    \addplot[smooth cycle, black!70!white] coordinates {
            (3,0)
            (3,3)
            (0,3)
            (0,0)
        }; 
 
    \nextgroupplot[title={$\rmA_2$},title style={yshift=-1.5ex}]
    
    \foreach \x in {0,...,6}{
    \foreach \y in {0,...,6}{
    \addplot[only marks,mark size=1.2pt,black!40!white] coordinates {(\x,\y)} ;
    }}
    \addplot[only marks,mark size=1.2pt, mark=square*, red] coordinates {(0,0)} ;   
    \addplot[only marks,mark size=1.2pt, mark=square*, red] coordinates {(1,1)} ;
    \addplot[only marks,mark size=1.2pt, mark=square*, red] coordinates {(2,2)} ;
    \addplot[only marks,mark size=1.2pt, mark=square*, red] coordinates {(3,3)} ;
    \addplot[only marks,mark size=1.2pt, mark=square*, red] coordinates {(4,4)} ;
    \addplot[only marks,mark size=1.2pt, mark=square*, red] coordinates {(5,5)} ;
    \addplot[only marks,mark size=1.2pt, mark=square*, red] coordinates {(6,6)} ;
    
    \addplot[smooth cycle, black!70!white] coordinates {
            (1,0)
            (1,1)
            (0,1)
            (0,0)
        }; 
    \addplot[smooth cycle, black!70!white] coordinates {
            (2,0)
            (2,2)
            (0,2)
            (0,0)
        };        
    \addplot[smooth cycle, black!70!white] coordinates {
            (3,0)
            (3,3)
            (0,3)
            (0,0)
        };            
    \addplot[smooth cycle, black!70!white] coordinates {
            (4,0)
            (4,4)
            (0,4)
            (0,0)
        };    
    \addplot[smooth cycle, black!70!white] coordinates {
            (5,0)
            (5,5)
            (0,5)
            (0,0)
        };
    \addplot[smooth cycle, black!70!white] coordinates {
            (6,0)
            (6,6)
            (0,6)
            (0,0)
        };
    \addplot[only marks,mark size=1.2pt,blue] coordinates {(1,0)} ; 
    \addplot[only marks,mark size=1.2pt,blue] coordinates {(0,1)} ;
    
    \addplot[only marks,mark size=1.2pt,blue] coordinates {(2,0)} ; 
    \addplot[only marks,mark size=1.2pt,blue] coordinates {(0,2)} ;

    \addplot[only marks,mark size=1.2pt,blue] coordinates {(3,0)} ; 
    \addplot[only marks,mark size=1.2pt,blue] coordinates {(0,3)} ;    
    
    \addplot[only marks,mark size=1.2pt,blue] coordinates {(4,0)} ; 
    \addplot[only marks,mark size=1.2pt,blue] coordinates {(0,4)} ;
    
    \addplot[only marks,mark size=1.2pt,blue] coordinates {(5,0)} ; 
    \addplot[only marks,mark size=1.2pt,blue] coordinates {(0,5)} ;   
    
    \addplot[only marks,mark size=1.2pt,blue] coordinates {(6,0)} ; 
    \addplot[only marks,mark size=1.2pt,blue] coordinates {(0,6)} ; 
       
    \nextgroupplot[title={$\Aex$},title style={yshift=-1.5ex}]
    
    \nextgroupplot[title={Expression Tree}]
    \foreach \x in {0,...,6}{
    \foreach \y in {0,...,6}{
    \addplot[only marks,mark size=1.2pt,black!40!white] coordinates {(\x,\y)} ;
    }}
    
    \addplot[only marks,mark size=1.2pt,blue] coordinates {(0,1)} ;
    \addplot[only marks,mark size=1.2pt,blue] coordinates {(1,0)} ;
    \addplot[smooth cycle, black!70!white] coordinates {
            (1,0)
            (1,1)
            (0,1)
            (0,0)
        };     
    
    \addplot[only marks,mark size=1.2pt,blue] coordinates {(0,3)} ;
    
    \addplot[only marks,mark size=1.2pt,blue] coordinates {(0,4)} ;
    \addplot[only marks,mark size=1.2pt,blue] coordinates {(2,0)} ;
    
    \addplot[only marks,mark size=1.2pt,blue] coordinates {(0,5)} ;
    
    \addplot[only marks,mark size=1.2pt,blue] coordinates {(0,2)} ;
    \addplot[only marks,mark size=1.2pt,blue] coordinates {(5,0)} ;

    \addplot[only marks,mark size=1.2pt, mark=square*, red] coordinates {(1,1)} ;
    \addplot[only marks,mark size=1.2pt, mark=square*, red] coordinates {(2,3)} ;
    \addplot[only marks,mark size=1.2pt, mark=square*, red] coordinates {(2,4)} ;

    \addplot[only marks,mark size=1.2pt,blue] coordinates {(0,0)} ;
       
    \addplot[smooth cycle, black!70!white] coordinates {
            (2,0)
            (2.4,1)
            (2.4,3)
            (2,4)
            (1,4.35)
            (0,4)
            (-0.4,3)
            (-0.4,1)
            (0,0)
            (1,-0.35)
        };
     
     \addplot[smooth, black!70!white] coordinates {
            (0,2)
            (0.2,1)
            (0,0)
        };
           
    \addplot[smooth cycle, black!70!white] coordinates {
            (2,0)
            (2.35,1)
            (2.35,2)
            (2,3)
            (1,3.35)
            (0,3)
            (-0.35,2)
            (-0.35,1)
            (0,0)
            (1,-0.35)
        };
        
    \addplot[smooth cycle, black!70!white] coordinates {
            (5,0)
            (5,5)
            (0,5)
            (0,0)
        }; 
     
     \addplot[smooth, black!70!white] coordinates {
            (4,0)
            (2,0.4)
            (0,0)
        };   
    \addplot[only marks,mark size=1.2pt, mark=square*, red] coordinates {(5,5)} ;
    \addplot[only marks,mark size=1.2pt, mark=square*, red] coordinates {(4,0)} ;
    \addplot[only marks,mark size=1.2pt, mark=square*, red] coordinates {(0,2)} ;
    \end{groupplot}
\end{tikzpicture}	
\end{center}

%% file: figures/fig-tree-bf-moment.tex
\begin{center}
\begin{tikzpicture}
\begin{groupplot}[
    axis line style = {black!70!white},
	group style     = {group size=3 by 1,
	xticklabels at  = edge bottom,
	yticklabels at  = edge left,
	ylabels at      = edge left,
	xlabels at      = edge bottom
	},
	width=\textwidth*(1/3),
	xtick              = {0,10},
	ytick              = {0,10},
	xmin               =-0.6,
	xmax               = 10.6,
    ymin               =-0.6,
    ymax               = 10.6,
    tick align         = center,
    xlabel             = $\alpha_1$,
	ylabel             = $\alpha_2$,
	y tick label style = {xshift=0.3em},
	x tick label style = {yshift=0.3em},
	ylabel style       = {rotate =-90, xshift= 1.2em},
	xlabel style       = {             yshift= 1.2em}
    ]
    \nextgroupplot[title={$\Aex$},title style={yshift=-1.5ex}]
    
    \foreach \x in {0,...,10}{
    \foreach \y in {0,...,10}{
    \addplot[only marks,mark size=1.2pt,black!40!white] coordinates {(\x,\y)} ;
    }}

    \addplot[only marks,mark size=1.2pt,blue] coordinates {(0,1)} ;
    \addplot[only marks,mark size=1.2pt,blue] coordinates {(1,0)} ;
    \addplot[smooth cycle, black!70!white] coordinates {
            (1,0)
            (1,1)
            (0,1)
            (0,0)
        };

    \addplot[only marks,mark size=1.2pt,blue] coordinates {(0,3)} ;

    \addplot[only marks,mark size=1.2pt,blue] coordinates {(0,4)} ;
    \addplot[only marks,mark size=1.2pt,blue] coordinates {(2,0)} ;
    
    \addplot[only marks,mark size=1.2pt,blue] coordinates {(0,5)} ;
    
    \addplot[only marks,mark size=1.2pt,blue] coordinates {(0,2)} ;
    \addplot[only marks,mark size=1.2pt,blue] coordinates {(5,0)} ;
  
    \addplot[only marks,mark size=1.2pt, mark=square*, red] coordinates {(1,1)} ;
    \addplot[only marks,mark size=1.2pt, mark=square*, red] coordinates {(2,3)} ;
    \addplot[only marks,mark size=1.2pt, mark=square*, red] coordinates {(2,4)} ;

    \addplot[only marks,mark size=1.2pt,blue] coordinates {(0,0)} ;
       
    \addplot[smooth cycle, black!70!white] coordinates {
            (2,0)
            (2.4,1)
            (2.4,3)
            (2,4)
            (1,4.35)
            (0,4)
            (-0.4,3)
            (-0.4,1)
            (0,0)
            (1,-0.35)
        };
     \addplot[smooth, black!70!white] coordinates {
            (0,4)
            (0.4,3)
            (0.4,2)
            (0,1)
        };         
     \addplot[smooth, black!70!white] coordinates {
            (2,0)
            (1,0)
        };    
    \addplot[smooth cycle, black!70!white] coordinates {
            (2,0)
            (2.35,1)
            (2.35,2)
            (2,3)
            (1,3.35)
            (0,3)
            (-0.35,2)
            (-0.35,1)
            (0,0)
            (1,-0.35)
        };
     \addplot[smooth, black!70!white] coordinates {
            (0,3)
            (0.25,2.5)
            (0.3,2)
            (0.25,1.5)
            (0,1)
        };          
        
    \addplot[smooth cycle, black!70!white] coordinates {
            (5,0)
            (5,5)
            (0,5)
            (0,0)
        }; 
     \addplot[smooth, black!70!white] coordinates {
            (5,0)
            (4,0.45)
            (3,0.55)
            (2,0.45)
            (1,0)
        };       
     \addplot[smooth, black!70!white] coordinates {
            (0,5)
            (0.45,4)
            (0.55,3)
            (0.45,2)
            (0,1)
        };
      \addplot[smooth, black!70!white] coordinates {
            (0,2)
            (0,1)
        };  
     \addplot[smooth, black!70!white] coordinates {
            (4,0)
            (3,0.4)
            (2,0.4)
            (1,0)
        };   
    \addplot[only marks,mark size=1.2pt, mark=square*, red] coordinates {(5,5)} ;
    \addplot[only marks,mark size=1.2pt, mark=square*, red] coordinates {(4,0)} ;
    \addplot[only marks,mark size=1.2pt, mark=square*, red] coordinates {(0,2)} ;
 
    \nextgroupplot[title={$\Aex$},title style={yshift=-1.5ex}]
    \foreach \x in {0,...,10}{
    \foreach \y in {0,...,10}{
    \addplot[only marks,mark size=1.2pt,black!40!white] coordinates {(\x,\y)} ;
    }}
    \foreach \x in {0,...,5}{
    \foreach \y in {0,...,5}{
    \addplot[only marks,mark size=1.2pt,blue] coordinates {(\x,\y)} ;
    }}
    
    \addplot[only marks,mark size=1.2pt, mark=square*, red] coordinates {(1,1)} ;
    \addplot[only marks,mark size=1.2pt, mark=square*, red] coordinates {(2,3)} ;
    \addplot[only marks,mark size=1.2pt, mark=square*, red] coordinates {(2,4)} ;
    \addplot[only marks,mark size=1.2pt, mark=square*, red] coordinates {(5,5)} ;
    \addplot[only marks,mark size=1.2pt, mark=square*, red] coordinates {(4,0)} ;
    \addplot[only marks,mark size=1.2pt, mark=square*, red] coordinates {(0,2)} ;
 
    \addplot[smooth, black!70!white] coordinates {
        (0,0)
        (0,5)
        (1,5)
        (1,0)
        (2,0)
        (2,5)        
        (3,5)
        (3,0)
        (4,0)
        (4,5)
        (5,5)
        (5,0)
      };

    \nextgroupplot[title={$\Aex$},title style={yshift=-1.5ex}]
    
    \foreach \x in {0,...,10}{
    \foreach \y in {0,...,10}{
    \addplot[only marks,mark size=1.2pt,black!40!white] coordinates {(\x,\y)} ;
    }}
    
    \foreach \x in {0,...,10}{
    \addplot[only marks,mark size=1.2pt,blue] coordinates {(\x,0)} ;
    }
    \foreach \x in {0,...,9}{
    \addplot[only marks,mark size=1.2pt,blue] coordinates {(\x,1)} ;
    }
    \foreach \x in {0,...,8}{
    \addplot[only marks,mark size=1.2pt,blue] coordinates {(\x,2)} ;
    }
    \foreach \x in {0,...,7}{
    \addplot[only marks,mark size=1.2pt,blue] coordinates {(\x,3)} ;
    }    
    \foreach \x in {0,...,6}{
    \addplot[only marks,mark size=1.2pt,blue] coordinates {(\x,4)} ;
    } 
    \foreach \x in {0,...,5}{
    \addplot[only marks,mark size=1.2pt,blue] coordinates {(\x,5)} ;
    }     
    \foreach \x in {0,...,4}{
    \addplot[only marks,mark size=1.2pt,blue] coordinates {(\x,6)} ;
    } 
    \foreach \x in {0,...,3}{
    \addplot[only marks,mark size=1.2pt,blue] coordinates {(\x,7)} ;
    } 
    \foreach \x in {0,...,2}{
    \addplot[only marks,mark size=1.2pt,blue] coordinates {(\x,8)} ;
    }    
    \foreach \x in {0,...,1}{
    \addplot[only marks,mark size=1.2pt,blue] coordinates {(\x,9)} ;
    }
    \foreach \x in {0,...,0}{
    \addplot[only marks,mark size=1.2pt,blue] coordinates {(\x,10)} ;
    }        
    \addplot[only marks,mark size=1.2pt, mark=square*, red] coordinates {(1,1)} ;
    \addplot[only marks,mark size=1.2pt, mark=square*, red] coordinates {(2,3)} ;
    \addplot[only marks,mark size=1.2pt, mark=square*, red] coordinates {(2,4)} ;
    \addplot[only marks,mark size=1.2pt, mark=square*, red] coordinates {(5,5)} ;
    \addplot[only marks,mark size=1.2pt, mark=square*, red] coordinates {(4,0)} ;
    \addplot[only marks,mark size=1.2pt, mark=square*, red] coordinates {(0,2)} ;
    
    \addplot[smooth, black!70!white] coordinates {
            (0,0)
            (9,0)
            (10,0)
            (9,1)
            (8,1)
            (1,1)
            (0,1)
            (0,2)
            (1,2)
            (7,2)
            (8,2)
            (7,3)
            (6,3)
            (1,3)
            (0,3)
            (0,4)
            (1,4)
            (5,4)
            (6,4)
            (5,5)
            (4,5)
            (1,5)
            (0,5)
            (0,6)
            (1,6)
            (3,6)
            (4,6)
            (3,7)
            (2,7)
            (1,7)
            (0,7)
            (0,8)
            (1,8)
            (2,8)
            (1,9)
            (0,9)
            (0,10)
        };   
    \end{groupplot}
\end{tikzpicture}	
\end{center}

%% file: figures/fig-chain-pattern.tex
\begin{center}
\begin{tikzpicture}
\begin{groupplot}[
    axis line style={black!70!white},
	group style={group size=3 by 1,
	xticklabels at = edge bottom,
	yticklabels at = edge left,
	ylabels at=edge left,
	xlabels at=edge bottom,
	},
	width=\textwidth*(1/3),
	xtick              = {0,6},
	ytick              = {0,6},
	xmin               =-0.6,
	xmax               = 6.6,
    ymin               =-0.6,
    ymax               = 6.6,
    tick align         = center,
    xlabel             = $\alpha_1$,
	ylabel             = $\alpha_2$,
	y tick label style = {xshift=0.3em},
	x tick label style = {yshift=0.3em},
	ylabel style       = {rotate =-90, xshift= 1.2em},
	xlabel style       = {             yshift= 1.2em}
    ]
    
    \nextgroupplot[title={$A_2$},title style={yshift=-1.5ex}]
    
    \foreach \x in {0,...,6}{
    \foreach \y in {0,...,6}{
    \addplot[only marks,mark size=1.2pt,black!40!white] coordinates {(\x,\y)} ;
    }}
    \addplot[only marks,mark size=1.2pt, mark=square*, red] coordinates {(0,0)} ;
    \addplot[only marks,mark size=1.2pt, mark=square*, red] coordinates {(1,1)} ;
    \addplot[only marks,mark size=1.2pt, mark=square*, red] coordinates {(2,2)} ;
    \addplot[only marks,mark size=1.2pt, mark=square*, red] coordinates {(3,3)} ;
    \addplot[only marks,mark size=1.2pt, mark=square*, red] coordinates {(4,4)} ;
    \addplot[only marks,mark size=1.2pt, mark=square*, red] coordinates {(5,5)} ;
    \addplot[only marks,mark size=1.2pt, mark=square*, red] coordinates {(6,6)} ;

    \addplot[smooth, black!70!white] coordinates {
            (0,0)
            (6,6)};

    \nextgroupplot[title={$A_3$},title style={yshift=-1.5ex}]

    \foreach \x in {0,...,6}{
    \foreach \y in {0,...,6}{
    \addplot[only marks,mark size=1.2pt,black!40!white] coordinates {(\x,\y)} ;
    }}
    
    \addplot[only marks,mark size=1.2pt, mark=square*, red] coordinates {(0,0)} ;
    \addplot[only marks,mark size=1.2pt, mark=square*, red] coordinates {(0,3)} ;
    \addplot[only marks,mark size=1.2pt, mark=square*, red] coordinates {(0,6)} ;
    \addplot[only marks,mark size=1.2pt, mark=square*, red] coordinates {(2,0)} ;
    \addplot[only marks,mark size=1.2pt, mark=square*, red] coordinates {(4,0)} ;
    \addplot[only marks,mark size=1.2pt, mark=square*, red] coordinates {(2,3)} ;

    \addplot[smooth cycle, black!70!white, tension=0.2] coordinates {
            (0,0)
            (0,3)
            (0,6)
            (2,3)
            (4,0)
            (2,0)
            (0,0)
        };
 
    \nextgroupplot[title={$A_1$},title style={yshift=-1.5ex}]
    
    \foreach \x in {0,...,6}{
    \foreach \y in {0,...,6}{
    \addplot[only marks,mark size=1.2pt,black!40!white] coordinates {(\x,\y)} ;
    }}
    \addplot[only marks,mark size=1.2pt, mark=square*, red] coordinates {(0,0)} ;
    \addplot[only marks,mark size=1.2pt, mark=square*, red] coordinates {(3,3)} ;
    \addplot[only marks,mark size=1.2pt, mark=square*, red] coordinates {(3,0)} ;
    \addplot[only marks,mark size=1.2pt, mark=square*, red] coordinates {(0,3)} ;

    \addplot[smooth, black!70!white] coordinates {
            (0,0)
            (3,3)
        }; 
    \addplot[smooth, black!70!white] coordinates {
            (0,0)
            (0,3)
        }; 
    \addplot[smooth, black!70!white] coordinates {
            (0,0)
            (3,0)
        };

\end{groupplot}
\end{tikzpicture}	
\end{center}	

%% file: figures/fig-shifted-chain-pattern.tex
\begin{center}
\begin{tikzpicture}
\begin{groupplot}[
    axis line style={black!70!white},
	group style={group size=3 by 1,
	xticklabels at = edge bottom,
	yticklabels at = edge left,
	ylabels at=edge left,
	xlabels at=edge bottom,
	},
	width=\textwidth*(1/3),
	xtick              = {0,5},
	ytick              = {0,5},
	xmin               =-0.6,
	xmax               = 5.6,
    ymin               =-0.6,
    ymax               = 5.6,
    tick align         = center,
    xlabel             = $\alpha_1$,
	ylabel             = $\alpha_2$,
	y tick label style = {xshift=0.3em},
	x tick label style = {yshift=0.3em},
	ylabel style       = {rotate =-90, xshift= 1.2em},
	xlabel style       = {             yshift= 1.2em}
    ]
    
    \nextgroupplot[title={$A_4$ },title style={yshift=-1.5ex}]
    
    \foreach \x in {0,...,5}{
    \foreach \y in {0,...,5}{
    \addplot[only marks,mark size=1.2pt,black!40!white] coordinates {(\x,\y)} ;
    }}
    \addplot[only marks,mark size=1.2pt, mark=square*, red] coordinates {(4,0)} ;
    \addplot[only marks,mark size=1.2pt, mark=square*, red] coordinates {(4,1)} ;
    \addplot[only marks,mark size=1.2pt, mark=square*, red] coordinates {(4,2)} ;
    \addplot[only marks,mark size=1.2pt, mark=square*, red] coordinates {(4,3)} ;
    \addplot[only marks,mark size=1.2pt, mark=square*, red] coordinates {(4,4)} ;
    \addplot[only marks,mark size=1.2pt, mark=square*, red] coordinates {(4,5)} ;   

    \addplot[smooth, black!70!white] coordinates {
            (4,0)
            (4,5)
        };
    
    \nextgroupplot[title={$A_1$},title style={yshift=-1.5ex}]
    
    \foreach \x in {0,...,5}{
    \foreach \y in {0,...,5}{
    \addplot[only marks,mark size=1.2pt,black!40!white] coordinates {(\x,\y)} ;
    }}
    \addplot[only marks,mark size=1.2pt, mark=square*, red] coordinates {(0,0)} ;
    \addplot[only marks,mark size=1.2pt, mark=square*, red] coordinates {(3,3)} ;
    \addplot[only marks,mark size=1.2pt, mark=square*, red] coordinates {(3,0)} ;
    \addplot[only marks,mark size=1.2pt, mark=square*, red] coordinates {(0,3)} ;
    
    \addplot[only marks,mark size=1.2pt,blue] coordinates {(1,0)} ;
    \addplot[only marks,mark size=1.2pt,blue] coordinates {(2,0)} ;
    
    \addplot[only marks,mark size=1.2pt,blue] coordinates {(0,1)} ;
    \addplot[only marks,mark size=1.2pt,blue] coordinates {(1,1)} ;
    \addplot[only marks,mark size=1.2pt,blue] coordinates {(2,1)} ;
    \addplot[only marks,mark size=1.2pt,blue] coordinates {(3,1)} ;
    
    \addplot[only marks,mark size=1.2pt,blue] coordinates {(0,2)} ;
    \addplot[only marks,mark size=1.2pt,blue] coordinates {(1,2)} ;
    \addplot[only marks,mark size=1.2pt,blue] coordinates {(2,2)} ;
    \addplot[only marks,mark size=1.2pt,blue] coordinates {(3,2)} ;
    
    \addplot[only marks,mark size=1.2pt,blue] coordinates {(1,3)} ;
    \addplot[only marks,mark size=1.2pt,blue] coordinates {(2,3)} ;

    \addplot[smooth, black!70!white] coordinates {
            (0,0)
            (0,3.3)
        }; 
        
    \addplot[smooth, black!70!white] coordinates {
            (1,0)
            (1,3.3)
        }; 
        
    \addplot[smooth, black!70!white] coordinates {
            (2,0)
            (2,3.3)
        }; 
        
    \addplot[smooth, black!70!white] coordinates {
            (3,0)
            (3,3.3)
        }; 
        
    \addplot[smooth, black!70!white] coordinates {
            (0  ,0)
            (3.3,0)
        };        

    \addplot[smooth, black!70!white] coordinates {
            (0  ,1)
            (3.3,1)
        };  

    \addplot[smooth, black!70!white] coordinates {
            (0  ,2)
            (3.3,2)
        };  
        
    \addplot[smooth, black!70!white] coordinates {
            (0  ,3)
            (3.3,3)
        };  
	
    \nextgroupplot[title={$\Aex$},title style={yshift=-1.5ex}]
    
    \nextgroupplot[title={Expression Tree}]
    \foreach \x in {0,...,5}{
    \foreach \y in {0,...,5}{
    \addplot[only marks,mark size=1.2pt,black!40!white] coordinates {(\x,\y)} ;
    }}
    
    \addplot[only marks,mark size=1.2pt, mark=square*, red] coordinates {(1,1)} ;
    \addplot[only marks,mark size=1.2pt, mark=square*, red] coordinates {(2,3)} ;
    \addplot[only marks,mark size=1.2pt, mark=square*, red] coordinates {(2,4)} ;
    \addplot[only marks,mark size=1.2pt, mark=square*, red] coordinates {(5,5)} ;
    \addplot[only marks,mark size=1.2pt, mark=square*, red] coordinates {(4,0)} ;
    \addplot[only marks,mark size=1.2pt, mark=square*, red] coordinates {(0,2)} ;    

    \addplot[only marks,mark size=1.2pt,blue] coordinates {(2,5)} ;    
    \addplot[only marks,mark size=1.2pt,blue] coordinates {(2,2)} ;
    \addplot[only marks,mark size=1.2pt,blue] coordinates {(2,1)} ;
    \addplot[only marks,mark size=1.2pt,blue] coordinates {(2,0)} ;
    \addplot[smooth, black!70!white] coordinates {
            (2,0)
            (2,5.3)
        };
        
    \addplot[only marks,mark size=1.2pt,blue] coordinates {(1,0)} ;
    \addplot[only marks,mark size=1.2pt,blue] coordinates {(3,0)} ;
    \addplot[smooth, black!70!white] coordinates {
            (0  ,0)
            (4.3,0)
        };  
        
    \addplot[only marks,mark size=1.2pt,blue] coordinates {(0,0)} ;
    \addplot[only marks,mark size=1.2pt,blue] coordinates {(3,3)} ;
    \addplot[only marks,mark size=1.2pt,blue] coordinates {(4,4)} ;
    \addplot[smooth, black!70!white] coordinates {
            (0,0)
            (5,5)
        };

    \addplot[only marks,mark size=1.2pt,blue] coordinates {(0,1)} ;
    \addplot[only marks,mark size=1.2pt,blue] coordinates {(1,2)} ;

    \addplot[smooth, black!70!white] coordinates {
            (0,0)
            (0,2)
        };
    \addplot[smooth, black!70!white] coordinates {
            (1,0)
            (1,2)
        };
        
    \addplot[smooth, black!70!white] coordinates {
            (0,1)
            (2,1)
        };
    \addplot[smooth, black!70!white] coordinates {
            (0,2)
            (2,2)
        };  
\end{groupplot}
\end{tikzpicture}	
\end{center}	

%% file: figures/fig-intro-pattern-configs.tex
\begin{center}
\begin{tikzpicture}
\begin{groupplot}[
    axis line style = {black!70!white},
	group style     = {group size=3 by 1,
	xticklabels at  = edge bottom,
	yticklabels at  = edge left,
	ylabels at      = edge left,
	xlabels at      = edge bottom
	},
	width=\textwidth*(1/3),
	xtick              = {0,5},
	ytick              = {0,5},
	xmin               =-0.6,
	xmax               = 5.6,
    ymin               =-0.6,
    ymax               = 5.6,
    tick align         = center,
    xlabel             = $\alpha_1$,
	ylabel             = $\alpha_2$,
	y tick label style = {xshift=0.3em},
	x tick label style = {yshift=0.3em},
	ylabel style       = {rotate =-90, xshift= 1.2em},
	xlabel style       = {             yshift= 1.2em}
    ]
    \nextgroupplot[title={$\Aex$}]

    \foreach \x in {0,...,10}{
    \foreach \y in {0,...,10}{
    \addplot[only marks,mark size=1.5pt,black!40!white] coordinates {(\x,\y)} ;
    }}
    \addplot[only marks,mark size=1.5pt,red] coordinates {(1,1)} ;
    \addplot[only marks,mark size=1.5pt,red] coordinates {(2,3)} ;
    \addplot[only marks,mark size=1.5pt,red] coordinates {(2,4)} ;
    \addplot[only marks,mark size=1.5pt,red] coordinates {(5,5)} ;
    \addplot[only marks,mark size=1.5pt,red] coordinates {(4,0)} ;
    \addplot[only marks,mark size=1.5pt,red] coordinates {(0,2)} ;    
 
    \addplot[only marks,mark size=1.5pt,blue] coordinates {(4,4)};
    \addplot[only marks,mark size=1.5pt,blue] coordinates {(3,3)};
    \addplot[only marks,mark size=1.5pt,blue] coordinates {(2,2)};
    \addplot[only marks,mark size=1.5pt,blue] coordinates {(0,0)};
    
    \addplot[only marks,mark size=1.5pt,blue] coordinates {(2,0)} ;
    \addplot[only marks,mark size=1.5pt,blue] coordinates {(2,1)} ;
    \addplot[only marks,mark size=1.5pt,blue] coordinates {(2,5)} ;
    
    \addplot[only marks,mark size=1.5pt,blue] coordinates {(0,1)} ;
    \addplot[only marks,mark size=1.5pt,blue] coordinates {(1,0)} ;
    \addplot[only marks,mark size=1.5pt,blue] coordinates {(1,5)} ;
    \addplot[only marks,mark size=1.5pt,blue] coordinates {(4,5)} ;
    \addplot[only marks,mark size=1.5pt,blue] coordinates {(3,5)} ;
    
    \addplot[only marks,mark size=1.5pt,blue] coordinates {(3,0)} ;
    \addplot[only marks,mark size=1.5pt,blue] coordinates {(5,0)} ;
    \addplot[only marks,mark size=1.5pt,blue] coordinates {(0,4)} ;
    \addplot[only marks,mark size=1.5pt,blue] coordinates {(0,3)} ;
    \addplot[only marks,mark size=1.5pt,blue] coordinates {(0,5)} ;
    
    \addplot[smooth, black!70!white] coordinates {
            (0,0)
            (5,5)
        }; 
    \addplot[smooth, black!70!white] coordinates {
            (0,0)
            (0,5)
        };
    \addplot[smooth, black!70!white] coordinates {
            (0,0)
            (5,0)
        };
    \addplot[smooth, black!70!white] coordinates {
            (2,0)
            (2,5)
        };
    \addplot[smooth, black!70!white] coordinates {
            (0,5)
            (5,5)
        };
    \addplot[smooth cycle, black!70!white] coordinates {
            (1,0)
            (1,1)
            (0,1)
            (0,0)
        }; 
    \addplot[smooth cycle, black!70!white] coordinates {
            (5,0)
            (5,5)
            (0,5)
            (0,0)
        };
    \nextgroupplot[title={$\Aex$}]
    
    \foreach \x in {0,...,10}{
    \foreach \y in {0,...,10}{
    \addplot[only marks,mark size=1.5pt,black!40!white] coordinates {(\x,\y)} ;
    }}
    \addplot[only marks,mark size=1.5pt,red] coordinates {(1,1)} ;
    \addplot[only marks,mark size=1.5pt,red] coordinates {(2,3)} ;
    \addplot[only marks,mark size=1.5pt,red] coordinates {(2,4)} ;
    \addplot[only marks,mark size=1.5pt,red] coordinates {(5,5)} ;
    \addplot[only marks,mark size=1.5pt,red] coordinates {(4,0)} ;
    \addplot[only marks,mark size=1.5pt,red] coordinates {(0,2)} ;    
 
    \addplot[only marks,mark size=1.5pt,blue] coordinates {(0,0)};
    
    \addplot[only marks,mark size=1.5pt,blue] coordinates {(2,0)} ;
    
    \addplot[only marks,mark size=1.5pt,blue] coordinates {(0,1)} ;
    \addplot[only marks,mark size=1.5pt,blue] coordinates {(1,0)} ;
    
    \addplot[only marks,mark size=1.5pt,blue] coordinates {(3,0)} ;
    \addplot[only marks,mark size=1.5pt,blue] coordinates {(5,0)} ;
    \addplot[only marks,mark size=1.5pt,blue] coordinates {(0,4)} ;
    \addplot[only marks,mark size=1.5pt,blue] coordinates {(0,3)} ;
    \addplot[only marks,mark size=1.5pt,blue] coordinates {(0,5)} ;
    
    \addplot[smooth, black!70!white] coordinates {
            (0,0)
            (0,5)
        };
    \addplot[smooth, black!70!white] coordinates {
            (0,0)
            (5,0)
        };
    \addplot[smooth cycle, black!70!white] coordinates {
            (1,0)
            (1,1)
            (0,1)
            (0,0)
        };
    \addplot[smooth cycle, black!70!white] coordinates {
            (2,0)
            (2.4,1)
            (2.4,3)
            (2,4)
            (1,4.35)
            (0,4)
            (-0.4,3)
            (-0.4,1)
            (0,0)
            (1,-0.35)
        };
    \addplot[smooth cycle, black!70!white] coordinates {
            (2,0)
            (2.35,1)
            (2.35,2)
            (2,3)
            (1,3.35)
            (0,3)
            (-0.35,2)
            (-0.35,1)
            (0,0)
            (1,-0.35)
        };
    \addplot[smooth cycle, black!70!white] coordinates {
            (5,0)
            (5,5)
            (0,5)
            (0,0)
        };            
    \nextgroupplot[title={$\Aex$}] 
    \foreach \x in {0,...,10}{
    \foreach \y in {0,...,10}{
    \addplot[only marks,mark size=1.5pt,black!40!white] coordinates {(\x,\y)} ;
    }}
    \addplot[only marks,mark size=1.5pt,red] coordinates {(1,1)} ;
    \addplot[only marks,mark size=1.5pt,red] coordinates {(2,3)} ;
    \addplot[only marks,mark size=1.5pt,red] coordinates {(2,4)} ;
    \addplot[only marks,mark size=1.5pt,red] coordinates {(5,5)} ;
    \addplot[only marks,mark size=1.5pt,red] coordinates {(4,0)} ;
    \addplot[only marks,mark size=1.5pt,red] coordinates {(0,2)} ;    
 
    \addplot[only marks,mark size=1.5pt,blue] coordinates {(4,4)};
    \addplot[only marks,mark size=1.5pt,blue] coordinates {(3,3)};
    \addplot[only marks,mark size=1.5pt,blue] coordinates {(2,2)};
    \addplot[only marks,mark size=1.5pt,blue] coordinates {(0,0)};
    
    \addplot[only marks,mark size=1.5pt,blue] coordinates {(2,0)} ;
    \addplot[only marks,mark size=1.5pt,blue] coordinates {(2,1)} ;
    \addplot[only marks,mark size=1.5pt,blue] coordinates {(2,5)} ;
    
    \addplot[only marks,mark size=1.5pt,blue] coordinates {(0,1)} ;
    \addplot[only marks,mark size=1.5pt,blue] coordinates {(1,0)} ;
    \addplot[only marks,mark size=1.5pt,blue] coordinates {(1,5)} ;
    \addplot[only marks,mark size=1.5pt,blue] coordinates {(4,5)} ;
    \addplot[only marks,mark size=1.5pt,blue] coordinates {(3,5)} ;
    
    \addplot[only marks,mark size=1.5pt,blue] coordinates {(3,0)} ;
    \addplot[only marks,mark size=1.5pt,blue] coordinates {(5,0)} ;
    \addplot[only marks,mark size=1.5pt,blue] coordinates {(0,4)} ;
    \addplot[only marks,mark size=1.5pt,blue] coordinates {(0,3)} ;
    \addplot[only marks,mark size=1.5pt,blue] coordinates {(0,5)} ;
    
    \addplot[smooth, black!70!white] coordinates {
            (0,0)
            (0,5)
        };
    \addplot[smooth, black!70!white] coordinates {
            (0,0)
            (5,0)
        };
    \addplot[smooth cycle, black!70!white] coordinates {
            (1,0)
            (1,1)
            (0,1)
            (0,0)
        };
    \addplot[smooth, black!70!white] coordinates {
            (0,5)
            (5,5)
        };
    \addplot[smooth, black!70!white] coordinates {
            (0,5)
            (5,5)
        }; 
    \addplot[smooth, black!70!white] coordinates {
            (0,0)
            (5,5)
        };  
    \addplot[smooth, black!70!white] coordinates {
            (2,0)
            (2,5)
        };
    \addplot[smooth, black!70!white] coordinates {
            (0,4)
            (1,4.35)
            (2,4)
            (3,3.65)
            (4,4)
        };            
    \addplot[smooth cycle, black!70!white] coordinates {
            (2,0)
            (2.4,1)
            (2.4,3)
            (2,4)
            (1,4.2)
            (0,4)
            (-0.4,3)
            (-0.4,1)
            (0,0)
            (1,-0.2)
        };
    \addplot[smooth cycle, black!70!white] coordinates {
            (2,0)
            (2.35,1)
            (2.35,2)
            (2,3)
            (1,3.2)
            (0,3)
            (-0.35,2)
            (-0.35,1)
            (0,0)
            (1,-0.2)
        };
    \addplot[smooth cycle, black!70!white] coordinates {
            (5,0)
            (5,5)
            (0,5)
            (0,0)
        };                
    \end{groupplot}
\end{tikzpicture}	
\end{center}

%% file: 3_duality.tex
\section{Conic pattern relaxations and their duality} \label{sec_conic}

We now focus on a different aspect of polynomial optimization -- duality. Complementary, but also based on the concept of sparsification, we shall see how patterns corresponding to some popular dual approaches in polynomial optimization arise from a unifying, general framework.

\subsection{Conic pattern relaxations} 

We introduce a conic analog of \eqref{P-RLXa}. To this end, analogously to how we defined  the moment body $\cM_A(X)$ using the convex hull, we introduce the \emph{moment cone} $\cC_A(X) :=\overline {  \cone \{ x^A \colon x \in X\} }$, defined using the conic hull. In the conic setting, we need to change our treatment of $v_0$. In the discussion below, $v_0$ is a variable and not the constant $1$. In particular, the use of the functional $L_v$ is modified accordingly. 

We give an example in which the relation between the moment body and the moment cone is straightforward and another example, where the relation is intricate.  The two situations differ by the condition whether the set of exponent vectors contains the zero exponent. 

\begin{ex} 
	Let $X = \R$. 
	For $A = \{0,1,2\}$, one has 
	\[
	\cC_A(X) = \left\{ \begin{pmatrix} v_0 \\ v_1 \\ v_2 \end{pmatrix} \in \R^A \colon \begin{pmatrix} v_0 & v_1 \\ v_1 & v_2 \end{pmatrix} \succeq 0 \right\},
	\]  
	while 
	\[
	\cM_A(X) = \left\{ \begin{pmatrix} 1 \\ v_1 \\ v_2 \end{pmatrix} \in \R^A \colon \begin{pmatrix} 1 & v_1 \\ v_1 & v_2 \end{pmatrix} \succeq 0 \right\}
	\]
	is the cross-section of $\cC_A(X)$ by the plane $v_0=1$. So, here $\cC_A(X)$ is a homogenization of $\cM_A(X)$ and $\cM_A(X)$ is a de-homogenization of $\cC_A(X)$. That is, the relation between $\cC_A(X)$ and $\cM_A(X)$ is straightforward. 
\end{ex}

\begin{ex} 	
	Consider the case $A = \{2,3,4\}$ and $X = \R$. 
	Obviously, a polynomial $f (x)\in \R[x]_A$ is non-negative if and only if $f(x) / x^2 \in \R[x]_{\{0,1,2\}} $ is non-negative. This implies that non-negative linear functionals on $\cC_A(X)$ correspond to non-negative linear functionals on $\cC_{\{0,1,2\}}(X)$. Thus, by the separation theorem for convex cones, $\cC_A(X)$ is a copy of $\cC_{\{0,1,2\}}(X)$ so that one has 
	\[
	\cC_A(X) = \left\{ \begin{pmatrix} v_2 \\ v_3 \\ v_4 \end{pmatrix}  \in \R^A \colon \begin{pmatrix} v_2& v_3 \\ v_3 & v_4 \end{pmatrix} \succeq 0 \right\} 
	\]
	The moment body $\cM_A(X)$ can be obtained by taking the cross-section of $\cC_{\{0\} \cup A}(X)$ by the hyperplane $v_0=1$. By standard results on the moment problem for the real line $\R$, we know that the moment cone $\cC_{\{0,1,2,3,4\}}(X)$ is defined by $L_v(M_{\{0,1,2\}} )  \succeq 0$. Consequently, since $\cC_{\{0\} \cup A}$ is a projection of $\cC_{\{0,1,2,3,4\}}(X)$, we have 
	\[
	\cC_{\{0\} \cup A}(X) = \left\{ \begin{pmatrix} v_0\\ v_2 \\ v_3 \\ v_4 \end{pmatrix} \in \R^{\{0\} \cup A} \colon  \underbrace{\begin{pmatrix} v_0 & v_1 & v_2 \\ v_1 & v_2 & v_3 \\ v_2 & v_3 & v_4 \end{pmatrix}}_{=:L_v(M_{\{0,1,2\}})} \succeq 0 \ \text{for some} \ v_1 \in \R \right\} 
	\]
	The semidefinite condition defining $\cC_{ \{0\} \cup A}(X)$ is an instance of a positive semidefinite matrix-completion problem for a chordal graph. We refer to Chapter~10 of \cite{vandenberghe2015chordal} for details. Hence, by \cite[Theorem~10.1]{vandenberghe2015chordal}, this semidefinite condition can be phrased via the positive semidefiniteness of the sub-matrices in the matrix defining $\cC_{\{0\} \cup A}(X)$ that do not involve $v_1$. Consequently, 
	\[
	\cC_{\{0\} \cup A}(X) = \left\{ \begin{pmatrix} v_0\\ v_2 \\ v_3 \\ v_4 \end{pmatrix} \in \R^{\{0\} \cup A} 
	\colon 
	\begin{pmatrix} v_0 & v_2 \\ v_2 & v_4 \end{pmatrix} \succeq 0, \  \begin{pmatrix} v_2& v_3 \\ v_3 & v_4 \end{pmatrix} \succeq 0 \right\} 
	\]
	By fixing $v_0=1$ we obtain a description of $\cM_A(X)$ by two LMIs with matrices of size two: 
	\[
	\cM_A(X) = \left\{ \begin{pmatrix} v_2 \\ v_3 \\ v_4 \end{pmatrix} \in \R^A \colon  
	\begin{pmatrix} 1 & v_2 \\ v_2 & v_4 \end{pmatrix} \succeq 0, \  \begin{pmatrix} v_2& v_3 \\ v_3 & v_4 \end{pmatrix} \succeq 0
	\right\} 
	\]
	We thus see that the description of $\cM_A(X)$ contains an LMI not present in the description of $\cC_A(X)$. 
\end{ex} 

For a finite exponent set $A$ with $0 \in A$, we  formulate the conic convexification of the problem $\inf_{x \in X} f(x)$ of optimizing a polynomial $f \in \R[x]_A$ over a set $X \subseteq \R^n$, as the problem $\inf \{ L_v(f) \colon v \in \R^A, \ v_0 = 1, \ v \in \cC_A(X) \}$. This is nothing but \eqref{C-POPa}, since the constraints $v_0=1, v \in \cC_A(X)$ are nothing but the constraint $v \in \cM_A(X)$. But when we move on to a  sparse relaxation, the conic sparse relaxation differs from \eqref{P-RLXa} in general. As with \eqref{P-RLXa}, we fix a pattern family $\cF = \{P_1,\ldots,P_N\}$ such that $B:= P_1 \cup \cdots \cup P_N$ contains $A$ as a subset. We call
\begin{equation}
	\inf \bigl\{ L_v(f) \colon v_0 = 1, v_{P_i} \in \cC_{P_i}(X)  \ \text{for all} \ i \in [N] \bigr\} \tag{C-P-RLX} \label{C-P-RLX}
\end{equation}
the \emph{conic pattern relaxation} of the problem $\inf_{x \in X} f(x)$ for the pattern family $\cF = \{P_1,\ldots,P_N\}$. When all $P_i$ contain $0$, \eqref{C-P-RLX} is nothing by \eqref{P-RLXa}, but when some of the $P_i$'s do not contain $0$, \eqref{C-P-RLX} is a coarser convexification than \eqref{P-RLXa} in general. 

Example~\ref{ex_difference} illustrates the difference between \eqref{P-RLXa} and \eqref{C-P-RLX}. 
\begin{ex} \label{ex_difference}
	Consider $A = \{0,1,2,3,4\}$, $X = \R$ and the pattern family $\cF = \{ P_1, P_2\}$ with $P_1 = \{0,1,2\}$ and $P_2 = \{2,3,4\}$. For this choice, \eqref{P-RLXa}, formulated as  
	\[
	\inf \left\{ \sum_{i=0}^4 f_i v_i \colon (v_0,v_1,v_2) \in \cM_{P_1}(X), \ (v_2,v_3,v_4) \in \cM_{P_2}(X) \right\},
	\] 
	gives the exact optimal value of $\inf_{x \in X} f(x)$ when $f \in \R[x]_A$ satisfies $f_0= f_1=f_2=0$. In contrast to this, \eqref{C-P-RLX}, formulated as 
	\[
	\inf \left\{ \sum_{i=0}^4 f_i v_i \colon v_0 = 1, (v_0,v_1,v_2) \in \cC_{P_1}(X), \ (v_2,v_3,v_4) \in \cC_{P_2}(X) \right\},
	\] gives the optimal value $-\infty$ or $0$ when $f_0=f_1 = f_2= 0$. To see this, note that $(v_2, v_3, v_4) \in \cC_{P_2}(X)$ is a positively homogeneous constraint so that validity implies the validity of $ (\lambda v_2, \lambda v_3, \lambda v_4) \in \cC_{P_2}(X)$ for each $\lambda \in \R_+$. Consequently, $\inf \{ ( f_2 (\lambda v_2) + f_3 (\lambda v_3) + f_4(\lambda v_4) \colon \lambda \in \R_+ \} $ is $-\infty$ or $0$, depending on whether $f_2 v_2 + f_3 v_3 + f_4 v_4$ is negative or not. 
\end{ex} 

\eqref{C-P-RLX} is a coarser convexification of $\inf_{x \in X} f(x)$ than \eqref{P-RLXa}, but due to the direct presence of cones in \eqref{C-P-RLX}, it has an advantage of being directly connected with conic optimization. In addition, the cones $\cC_{P_i}(X)$ involved in \eqref{C-P-RLX} might have a simpler description than the respective moment bodies. We introduce the notation 
\[
\CRLX_\cF(X) := \{ v \in \R^B \colon v_{P_i} \in \cC_{P_i}(X) \ \text{for all} \in [N]\},
\]
for the cone defining the conic moment relaxation. This cone is a relaxation of the 
cone $\cC_A(X)$ in the sense of the inclusion $\cC_A(X) \subseteq \CRLX_\cF(X)_A$. 
The problem \eqref{C-P-RLX} can be written as  $\inf \{ L_v(f) \colon v_0 = 1, v \in \CRLX_\cF(X) \}$. 

\subsection{Conic optimization and conic duality} 
\label{sect:conic:duality} 

We provide a brief overview of conic optimization and conic duality, which we present in a form tailored to our purposes. 
Dual cones can be formulated for general Euclidean spaces. Here we use a definition directly addressing the real vector space and standard scalar product.
\begin{defn} 
	The \emph{dual cone} of a set $C \subseteq \R^m$  is the closed convex cone 
	\[
	C^\ast := \{ y \in \R^m \colon  \sprod{y}{x} \ge 0 \ \text{for all} \ x \in C\},
	\] 
	where $\sprod{x}{y} := x^\top y$ is the standard scalar product of $x$ and $y$. 
\end{defn} 

By separation theorems, for every closed convex cone $C$, one has $C^{\ast\ast} = C$. 

%
Conic programming is a general template of optimization problems with a linear objective function and so-called conic constraints. We define a \emph{conic problem}  as a problem of the form 
\begin{equation}
	\inf \{ \sprod{c}{x} \colon x \in C, \ b - A x \in D\} \tag{CP} \label{CP}
\end{equation}
with the linear objective function $x  \in \R^n \mapsto \sprod{c}{x}$ given by $c \in \R^n$, the conic constraint $x \in C$ on the vector $x$ of decision variables, defined using a closed convex cone $C \subseteq \R^n$, and a linear conic constraint $b - A x \in D$, defined using a vector $b \in \R^m$, a matrix $A \in \R^{m \times n}$ and a closed convex cone $D \subseteq \R^m$. If we choose $D = \{0\}$, the constraint $b - A x \in D$ amounts to the system of linear equations $A x = b$. If we choose $D = \R_+^m$, the constraint $b - A x \in D$ becomes a system of linear inequalities $A x \le b$. If we choose $C = \R_+^n$, the constraint $x \in C$ is the non-negativity constraint on the $x$ variables. These observations shows that linear programming in its basic versions is captured as a special case of \eqref{CP}. The dual problem of \eqref{CP} is defined as 
\begin{equation}
	\sup \{ \sprod{b}{y} \colon c -  A^\top y \in C^\ast,  \ y \in D^\ast \} \tag{CP-D} \label{CP-D}
\end{equation}


Rather than presenting a general conic-duality theorem relating \eqref{CP} and \eqref{CP-D}, we focus on its special case, which we need for applications in polynomial optimization. 
\begin{prop} \label{abstract:duality} 
	Let $C \subseteq \R^m$ be a pointed closed convex cone, let $p \in \R^m$ and $q \in C \setminus \{0\}$. Then 
	\begin{equation} \label{duality:eq} 
		\sup \{ \lambda \in \R \colon p - \lambda q \in C \}  = \inf \{ \sprod{p}{v} \colon v \in C^\ast , \sprod{q}{v} = 1\}. 
	\end{equation} 
\end{prop} 
The proposition is borrowed from \cite[Proposition~2.4]{averkov2024convex}. Its proof is an easy exercise in convexity theory. 

\begin{rem} 
	\eqref{duality:eq} asserts the equivalence of \eqref{CP} and \eqref{CP-D} for a special choice of $A, b, c$ and the cones $C$ and $D$. 
\end{rem} 

Let us also establish an abstract ``sparse'' version of Proposition~\ref{abstract:duality}. 

\begin{prop} \label{sparse:abstract:duality} 
	Let $C_1,\ldots,C_N \subseteq \R^m$ be closed convex cones, which are subsets of a pointed closed convex cone $C$, let $p \in \R^m$ and $q \in (C_1 + \cdots + C_N) \setminus \{0\}$. Then 
	\begin{equation}
		\sup \{ \lambda \in \R \colon p-\lambda q \in C_1 + \cdots + C_N\} = \inf \{ \sprod{p}{v} \colon v \in C_1^\ast \cap \cdots \cap C_N^\ast, \ \sprod{q}{v} = 1\}. 
	\end{equation}  
\end{prop} 
\begin{proof} 
	Since $C_1 + \cdots + C_N \subseteq C$ and $C$ is a pointed closed convex cone, it follows that $C_1 + \cdots + C_N$, too, is a pointed closed convex cone. Hence, the assertion follows by using Proposition~\ref{abstract:duality} for the cone $C_1 + \cdots + C_N$ in combination with the well-known equality $(C_1 + \cdots + C_N)^\ast = C_1^\ast \cap \cdots \cap C_N^\ast$. 
\end{proof} 

\subsection{Duality for conic pattern relaxations}

We apply conic optimization and duality as laid out in Section~\ref{sect:conic:duality} to conic pattern relaxations. Let $\cP_A(X)$ denote the set of polynomials $f \in \R[x]_A$ that are non-negative on $X$. It is clear that $\cP_A(X)$ is a closed convex cone. Furthermore, if $X$ has non-empty interior, the cone $\cP_A(X)$ is pointed. Indeed, if $\cP_A(X) \cap ( - \cP_A(X))$, as the set of polynomials $f \in \R[x]_A$ that are zero on $X$, consists only of the zero polynomial. 

\begin{rem} 
	The set $\R[x]_A$ can be identified with $\R^A$ via the natural bijection $\sum_{\alpha \in A} f_\alpha x^\alpha \leftrightarrow (f_\alpha)_{\alpha \in A}$. 
	With this identification, $L_v(f)$ is interpreted as the scalar product of $f \in \R[x]_A \leftrightarrow \R^A$ and $v \in \R^A$. Consequently, for $C \subseteq \R[x]_A$ we have $C^\ast  = \{ v \in \R^A \colon L_v(f) \ge 0 \ \text{for all} \ f \in C\}$. Analogously, for $D \subseteq \R^A$, we can interpret $D^\ast$ as a cone in $\R[x]_A$ and use the equality $D^\ast = \{ f \in \R[x]_A \colon L_v(f) \ge 0 \ \text{for all} \ v \in D\}$. 
\end{rem} 

\begin{prop} \label{moment:vs:positive} 
	Let $A \subseteq \N^n$ be a finite exponent set and $X \subseteq \R^n$. Then $\cC_A(X)^\ast = \cM_A(X)^\ast = \cP_A(X)$. 
\end{prop} 
\begin{proof} 
	We have 
	\begin{align*}
		\cC_A(X)^\ast  & =\left\{ f \in \R[x]_A \colon \sum_{\alpha \in A} f_\alpha v_\alpha \ge 0 \ \text{for all} \ v \in \cC_A(X) \right\} 
		\\ & = \left\{ f \in \R[x]_A \colon \sum_{\alpha \in A} f_\alpha v_\alpha \ge 0 \ \text{for all} \ v \in \overline{ \cone \{ x^A \colon x \in X\} } \right\} 
		\\ & = \left\{ f \in \R[x]_A \colon \sum_{\alpha \in A} f_\alpha v_\alpha \ge 0 \ \text{for all} \ v \in  \{ x^A \colon x \in X\}  \right\} 
		\\ & = \left\{ f \in \R[x]_A \colon \sum_{\alpha \in A} f_\alpha x^\alpha \ge 0 \ \text{for all} \ x \in X \right\} 
		\\ & = \cP_A(X).
	\end{align*}
	Analogously, we obtain $\cM_A(X)^\ast = \cP_A(X)$. 
\end{proof} 

The duality $\cM_A(X)^\ast = \cP_A(X)$ corresponds to a pair of optimization problems, dual to each other. 

\begin{prop} \label{nonsparse:duality} 
	Let $A \subseteq \N^n$ with $0 \in A$ and $f \in \R[x]_A$. Then 
	\begin{align*}
		\sup \{ \lambda \in \R \colon f - \lambda \in \cP_A(X) \} & = \inf \{ L_v(f) \colon v \in \cC_A(X), \ v_0=1\} 
		\\ & = \inf \{ L_v(f) \colon v \in \cM_A(X) \}.   
	\end{align*} 
\end{prop} 
\begin{proof} 
	The assertion is trivial, but it also follows from Proposition~\ref{abstract:duality}. 
\end{proof} 

In Proposition~\ref{nonsparse:duality}, the formal supremum-problem is a declaration of the intention to find lower bounds on $f$, while the infimum-problem (written in two equivalent versions) is a declaration of the intention to convexify. 
Proposition~\ref{nonsparse:duality} serves as a starting point for establishing an analogous sparse version of duality. 


\begin{prop} \label{duality} 
	Let $P_1,\ldots, P_N, A \subseteq \N^n$ be finite sets of exponent vectors  satisfying  $\{0\} \cup A \subseteq B:= P_1 \cup \cdots \cup P_N$. Let $f \in \R[x]_A$ and let $X \subseteq \R^n$ have non-empty interior. Then 
	\begin{align*}
		& \sup \{ 
		\lambda \in \R \colon f - \lambda \in \cP_{P_1}(X) + \cdots + \cP_{P_N}(X) 
		\}  \\ = & 
		\inf \{ L_v(f) \colon v_0 =1, \ v_{P_i} \in \cC_{P_i}(X)  \ \text{for every} \ i \in [N]\}. 
	\end{align*} 
	Furthermore, if  $0 \in P_i$ for every $i \in [N]$, then the infimum is equal to 
	\begin{align*} 
		\inf \{ L_v(f) \colon v_{P_i} \in \cM_{P_i}(X)  \ \text{for every} \ i \in [N]\}. 
	\end{align*} 	
	If, additionally, the set $X$ is compact, then $\sup$ and $\inf$ are attained and can be replaced by $\max$ and $\min$, respectively. 	
\end{prop} 
\begin{proof} 	
	The cones $\cP_{P_i}(X)$ are subsets of the closed pointed convex cone  $\cP_B(X)$, which shows that the assumptions of Proposition~\ref{sparse:abstract:duality} are fulfilled for the cones $C_i = \cP_{P_i}(X)$. 
	
	As a direct consequence of Proposition~\ref{moment:vs:positive}, we obtain 
	\begin{align*}
		(\cC_A(X) \times \R^{B \setminus A})^\ast = (\cM_A(X) \times \R^{B \setminus A} )^\ast   = \cP_A(X) \subseteq \R[x]_B
	\end{align*} 
	which, inserted in Proposition~\ref{sparse:abstract:duality}, yields the main assertion. 
	
	Now, assume that $0 \in P_i$ for every $i \in [N]$.	Then one has $\cM_{P_i}(X) = \{ v \in \cC_{P_i}(X) \colon v_0 =1\}$. 
	
	Assume additionally that $X$ is compact. 
	The infimum is finite and a minimum, because $\RLX_\cF(X) = \{ v \in \R^B \colon v_{P_i} \in \cM_{P_i}(X) \ \text{for all} \in [N] \}$ is a compact set contained in $\Bx(l_B, u_B)$.  Hence the  supremum is finite as well. 
	To see that the (finite) supremum is attained it suffices to observe that $\cP_{P_1}(X) + \cdots + \cP_{P_n}(X)$ is a closed convex cone so that $ \{ f- \lambda \colon \lambda \in \R, \ f- \lambda \in \cP_{P_1}(X) + \cdots + \cP_{P_N}(X)\}$ is a ray with the end point of the ray corresponding to the $\lambda$, for which the supremum is attained. 
\end{proof} 

We call 
\begin{equation}
	\max \{ \lambda \in \R \colon f- \lambda \in \cP_{P_1}(X) + \cdots + \cP_{P_N}(X) \}. \tag{D-P-RLX} \label{D-P-RLX}
\end{equation} 
the \emph{dual pattern relaxation} for the pattern family $\cF = \{P_1,\ldots,P_N\}$. The constraint in \eqref{D-P-RLX} is a formal constraint $f - \lambda = g_1 + \cdots + g_N$, involving the real-valued variable $\lambda$, certifying the bound $f \ge \lambda$ on $f$ through the ``conic variables'' $g_i \in \cP_{P_i}(X)$.  Analogously to our discussion of the primal problem \eqref{P-RLXa}, the conditions $g_i \in \cP_{P_i}(X)$ are formal constraints, since \eqref{D-P-RLX} gives no  clues on how to model these constraints. It is just a declaration that if the $P_i$'s allow a sparse formulation of $g_i \in \cP_{P_i}(X)$ in some optimization paradigm this implies the same for the model \eqref{D-P-RLX}. For being able to use \eqref{D-P-RLX}, one either needs to provide exact descriptions of $\cP_{P_i}(X)$, which may be a description in the original space or a lifted description, or to model an \emph{inner} approximation of $\cP_{P_i}(X)$. We stress that relaxing \eqref{D-P-RLX} means finding formulations of inner approximations of $\cP_{P_i}(X)$. Indeed, \eqref{D-P-RLX} is a maximization problem, so that a strengthening of the constraints $v_{P_i} \in \cP_{P_i}(X)$ by replacing $\cP_{P_i}(X)$ with smaller cones makes the maximum over $\lambda$'s smaller.  Thus, relaxing \eqref{D-P-RLX} means to determine a cone of particular types of non-negative polynomials with the support in $P_i$ and use such polynomials for certifying bounds $f(x) \ge \lambda$ valid for $x \in X$.

\subsection{Sparse conic approaches from the pattern perspective}

\label{existing:sparse:conic}

One can develop a concrete sparse conic approach by starting with \eqref{C-P-RLX} and then deriving \eqref{D-P-RLX} via dualization, but one can also start with \eqref{D-P-RLX} and then derive \eqref{C-P-RLX} via dualization. If one starts with \eqref{C-P-RLX}, one emphasizes convexification and then passes to \eqref{D-P-RLX}, which is concerned with a certification of bounds using sparse non-negative polynomials. If one starts with \eqref{D-P-RLX} and then derives \eqref{C-P-RLX} as, e.g., in \cite{katthan2021unified}, one emphasizes certificates of bounds and then moves on to convexification. Emphasis on bound certification has been the favorite choice in the study of the SAGE, SOS, SDSOS, TS-SSOS and related approaches, which we are going to summarize below. The contents of the following subsections are mostly a survey of results about sparse approaches to polynomial optimization that are spread in the literature. Our intention was to provide a largely self-contained and unified presentation, in which specific approaches are derived as corollaries of general results. 
Our presentation of the duality for SONC and SAGE differs from the presentation of the duality in other sources such as \cite{katthan2021unified}: rather than starting with a formulation of \eqref{D-P-RLX} and obtaining a formulation of \eqref{C-P-RLX} via the dualization, we start with a formulation of \eqref{C-P-RLX} and dualize it to a formulation of \eqref{D-P-RLX}. 

\subsubsection{SAGE, SONC, and SDSOS} \label{sec_SAGE}

SAGE and SONC are closely related to each other \cite{chandrasekaran2016relative,katthan2021unified}. The patterns in both these approaches are simplicial circuits, but SAGE deals with the non-negativity on $\R_+^n$, while SONC with the non-negativity on $\R^n$. We start with a discussion of SAGE.

\begin{prop}  \label{prop:age}
	Let $\gamma(0), \cdots, \gamma(k) \in \N^n$ be affinely independent exponent vectors  and let $\beta \in \N^n$ be an exponent vector in the relative interior of the $k$-dimensional simplex $\conv(\gamma(0),\ldots,\gamma(k))$. Consider the coefficients $\lambda_i>0$ of the convex combination 
	\begin{align*}
		\beta & = \sum_{i=0}^k \lambda_i \gamma(i),
		\\	1 & = \sum_{i=0}^k \lambda_i
	\end{align*} Then, for the pattern $P = \{\beta,\gamma(0),\ldots,\gamma(k)\} $, one has 
	\begin{align} \label{AGE:convex} 
		\cC_P(\R_+^n) = \left\{ v \in \R_+^P \colon 0 \le v_\beta \le \prod_{i=0}^k v_{\gamma(i)}^{\lambda_i} \right\} 
	\end{align} 
\end{prop} 
\begin{proof} 
	We denote the right-hand side of \eqref{AGE:convex} as $C$. 
	It is well known that the weighted geometric mean $g: \R_+^{k+1} \to \R$
	\[
	g(t_0,\ldots,t_k) = \prod_{i=0}^k t_i^{\lambda_i}
	\]
	is a concave positively homogeneous function. Thus, $C$ is a closed convex cone defined as the region below the graph of $g$ and above the graph of the zero function on $\R_+^{k+1}$. To see that $\cC_P(\R_+^n) \subseteq C$ is true, it suffices to observe that for the choice $v_\beta = x^\beta$ and $v_{\gamma(i)} = x^{\gamma(i)}$ with $x \in \R_+^n$, the inequalities defining $C$ are true, as the upper bound on $v_\beta$ is attained with equality and the lower bound on $v_\beta$ is trivially fulfilled. 
	To verify the converse inclusion $C \subseteq \cC_P(\R_+^n)$, we start with the set 
	\[
	G = \left \{ v \in \R_{>0}^P \colon v_\beta = \prod_{i=0}^k v_{\gamma(i)}^{\lambda_i} \right\}, 
	\]
the relative interior of the graph of $g$. Pick an arbitrary $v \in G$, fix $y_i = \ln v_{\gamma(i)}$ and consider the matrix $\Gamma = (\gamma(0),\ldots, \gamma(k))$. Let $\mathbf{1}=(1,\ldots,1)^\top \in \R^{k+1}$. The matrix $M = (\Gamma^\top \ \mathbf{1})$ has rank $k+1$. Hence, the linear system $M u = y$ with the right-hand side $y=(y_0,\ldots,y_k)^\top$ has a solution $u = (u_1,\ldots,u_{n+1})^\top$. Let $w = (u_1,\ldots,u_n)^\top$. One has $y = M u = \Gamma^\top w + \mathbf{1} u_{n+1}$, which can be spelled out as $\sprod{\gamma(i)}{w} + u_{n+1} = y_i$. Exponentiation of the latter equality gives 
	$x^{\gamma(i)} e^{u_{n+1}} = v_{\gamma(i)}$  with $x = ( e^{u_1} ,\ldots, e^{u_n})^\top \in \R_{>0}^n$. 
	For this choice of $x$ and $c = e^{u_{n+1}}$, we obtain $c x^\beta = \prod_{i=0}^k ( c x^{\gamma(i)} )^{ \lambda_i}  = \sum_{i=0}^k v_{\gamma(i)} = v_\beta$ and thus $c x^P = v$. We have realized $v \in G$ as a point in $\cC_P(\R_+^n)$ up to a scaling factor $c>0$. Since $\cP_P(\R_+^n)$ is a cone, $v \in G$ itself is contained in $\cC_P(\R_+^n)$. This shows that $G \subseteq \cC_P(\R_+^n)$. But then the topological closure $\overline{G}$, which is the whole graph of $g$, is also a subset of $\cC_P(\R_+^n)$. To conclude the proof it thus suffices to show that $\conv( \overline{G}) = C$. It is clear that if  $(v_{\gamma(0)},\ldots, v_{\gamma(k)})$ is a standard unit vector and $v_\beta =0$, then $v$ is in the graph $\overline{G}$ of $g$. But then the graph $\R_+^{k+1} \times \{0\}$ of the zero function, defined on $\R_+^k$ is in $\conv(\overline{G})$. This shows that the graph $\overline{G}$ of $g$ and the graph of the zero function on $\R_+^k$ are both subsets of $C$. But since every point $v \in C$ lies above the graph of the zero function and below the graph of $g$, we conclude that $C = \conv(\overline{G})$. 
\end{proof} 

	We call a pattern $P$ as in \eqref{AGE:convex} a \emph{SAGE-circuit}. For a family $\cF = \{P_1,\ldots,P_N\}$ of SAGE circuits and the domain $X=\R_+^n$, we call the problem \eqref{C-P-RLX} the \emph{SAGE moment convexification} of the problem $\inf_{x \in X} f(x)$. 

For establishing the dual of the SAGE moment convexification, we need to dualize the cone $\cC_P(\R^n_+)$ from \eqref{AGE:convex}. 
\begin{defn} \label{def:GMC}
	For $\lambda = (\lambda_0,\ldots,\lambda_k) \in \R_{>0}^{k+1}$ with $1 = \sum_{i=0}^k \lambda_i$, we define the \emph{geometric-mean cone} for $\lambda$ by 
	\[
	\GMC_\lambda := \left \{ (y, t_0,\ldots,t_k) \in \R_+^{k+2} \colon 0 \le y \le \prod_{i=0}^k t_i^{\lambda_i} \right\}.  
	\]
\end{defn} 
Note that the geometric-mean cone is defined slightly differently in \mosek's documentation \cite{mosek}. 

\begin{prop} \label{prop:gmc:dual} 
	In the notation of Definition~\ref{def:GMC}, one has 
	\begin{equation} \label{gmc:dual} 
		(\GMC_\lambda)^\ast = \left\{ (z,s_0,\ldots,s_k) \in \R \times \R_+^{k+1} \colon z + \sum_{i=0}^k \left( \frac{s_i}{\lambda_i} \right)^{\lambda_i} \ge 0 \right\}.
	\end{equation} 
\end{prop} 
\begin{proof} 
	We denote $\GMC_\lambda$ as $C$ and the right-hand side of \eqref{gmc:dual} as $D$. The scalar product of $(y,t_0,\ldots,t_k) \in C$ and $(z,s_0,\ldots,s_k) \in D$ satisfies 
	\begin{align*} 
		z y + \sum_{i=0}^k s_i t_i & =  zy + \sum_{i=0}^k \lambda_i \frac{s_i}{\lambda_i} t_i
		\\ & \ge z y + \prod_{i=0}^k \left( \frac{s_i}{\lambda_i} \right)^{\lambda_i} \prod_{i=0}^k t_i^{\lambda_i}   &  & \text{(by the AM-GM inequality)}
		\\ & \ge \left( z + \prod_{i=0}^k \left( \frac{s_i}{\lambda_i} \right)^{\lambda_i} \right)  y
		\\ & \ge 0.
	\end{align*} 
	This shows that $D \subseteq C^\ast$. 
Assume now $C^\ast \not\subseteq D$. Then there exists $(z,s_0,\ldots,s_k) \in C^\ast$ with $s_0,\ldots,s_k \ge 0$ and $z  + \prod_{i=0}^k \left( \frac{s_i}{\lambda_i} \right)^{\lambda_i} < 0$. The continuity of this function allows to find slightly increased values $s_0,\ldots, s_k > 0$ still resulting in $z  + \prod_{i=0}^k \left( \frac{s_i}{\lambda_i} \right)^{\lambda_i} < 0$. Then, fixing $(y,t_0,\ldots,t_k) \in C$ with $t_i = \frac{\lambda_i}{s_i}$ and $y = \prod_{i=0}^k t_i^{\lambda_i}$, for the scalar product of $(y,t_0,\ldots,t_k)$ and $(z, s_0,\ldots,s_k)$, we obtain 
	\begin{align*} 
		z y + \sum_{i=0}^k s_i t_i & = zy + \sum_{i=0}^k \lambda_i 
		\\ & = zy + 1
		\\ & < - \prod_{i=0}^k \left( \frac{s_i}{\lambda_i} \right)^{\lambda_i} y  + 1
		\\ & = - \prod_{i=0}^k \left( \frac{s_i}{\lambda_i} \right)^{\lambda_i} t_i^{\lambda_i}  + 1
		\\ & = 0. 
	\end{align*} 
	The scalar product is negative, which contradicts the fact that one of the vectors is in $C^\ast$ and  the other one is in $C$. This implies that $D = C^\ast$. 
\end{proof} 

As a a direct consequence of Proposition~\ref{prop:gmc:dual}, we obtain 
\begin{prop} \label{prop:sage:pol:cone}
	In the notation of Proposition~\ref{prop:age}, one has 
	\[
	\cP_P(\R_+^n) = \left\{ f \in \R[x]_P \colon f_\beta + \prod_{i=0}^k \left( \frac{f_{\gamma(i)}}{\lambda_i} \right)^{\lambda_i} \ge 0  \right\}. 
	\]
\end{prop} 
\begin{proof} 
	The assertion follows from Propositions~\ref{prop:age}, \ref{prop:gmc:dual}, and the duality relation $\cC_P(\R_+^n)^\ast = \cP_P(\R_+^n)$. 
\end{proof} 

Proposition~\ref{prop:sage:pol:cone} yields an explicit formulation of the problem \eqref{D-P-RLX} on the domain $X =\R_+^n$ in the case when $\cF = \{P_1,\ldots,P_N\}$ is a family of SAGE-circuits. We call such a problem \eqref{D-P-RLX} the \emph{SAGE relaxation} of $\inf_{x \in \R_+^n} f(x)$ with respect to the family of SAGE-circuits $\cF$. The polynomials of from $\cP_{P_1}(\R_+^n) + \cdots + \cP_{P_N}(\R_+^n)$, where $P_1,\ldots,P_N$ are SAGE-circuits, are called \emph{SAGE polynomials}. 

\begin{rem} \label{compute:SAGE} 
	Let us briefly discuss computational options for the SAGE moment convexification and the SAGE relaxation. The problems involve the cones $\GMC_\lambda$ and $(\GMC_\lambda)^\ast$ for different choices of $\lambda$. The cones $\GMC_\lambda$ and $(\GMC_\lambda)^\ast$ are similar in nature, because both involve the weighted geometric-mean in their description. It is known that $\GMC_\lambda$ and $(\GMC_\lambda)^\ast$ are second-order cone representable \cite{ben2001lectures}.  The Mosek documentation \cite{aps2020mosek} provides hints on how to model cones like $\GMC_\lambda$ and its dual. There are also further options, some of which rely on the substitution of variables $v_\alpha = e^{\nu_\alpha}$ for the SAGE moment convexification and  $f_\alpha  = e^{ \phi_\alpha}$ for the SAGE relaxation. 
\end{rem} 

We proceed with a similar derivation of primal and dual SONC relaxations.
\begin{prop} \label{sonc:convexification} 
	Let the assumptions of Proposition~\ref{prop:age} hold and assume that additionally $\gamma(0),\ldots,\gamma(k) \in 2 \N^n$. Then, for $\beta \in 2 \N^n$, one has 
	\begin{equation} \label{even:case} 
		\cC_P( \R^n)  = \left\{ v \in \R_+^P \colon 0 \le v_\beta \le \prod_{i=0}^k v_{\gamma(i)}^{\lambda_i} \right\}
	\end{equation}
	and, for $\beta \not\in 2 \N^n$, one has 
	\begin{equation} \label{odd:case} 
		\cC_P( \R^n)  = \left\{ v \in \R_+^P \colon - \prod_{i=0}^k v_{\gamma(i)}^{\lambda_i} \le v_\beta \le \prod_{i=0}^k v_{\gamma(i)}^{\lambda_i} \right\}
	\end{equation} 
\end{prop} 
\begin{proof} 
	If $\beta \in 2\N^n$, then $P \subseteq 2 \N^n$. This implies that the value $x^P$ does not depend on the signs of the coordinates $x_1,\ldots,x_n$ but only on the absolute values $|x_1|,\ldots, |x_n|$. This implies $\cC_P( \R^n) = \cC_P(\R_+^n)$ so that the assertion follows from Proposition~\ref{prop:age}. 
	
	If $\beta \notin 2 \N^n$, then the sign of $x^{\gamma(i)}$ still does not depend on the signs of $x_1,\ldots,x_n$ but the sign of $x^{\beta}$ depend on the signs of $x_1,\ldots,x_n$ and switching the sign of $x_i$, for which $\beta_i$ is odd, generates a sign switch of $x^\beta$. This shows that $\cC_P(\R^n)$ contains both $C' := \cC_P(\R_+^n)$ and its reflection $C'' := \{ v \in \R^P \colon (- v_\beta, v_{\gamma(0)},\ldots, v_{\gamma(k)}) \in \cC_P(\R_+^n)\}$.  Hence, $\cC_P(\R^n) \supseteq C' \cup C''$, where, in view of Proposition~\ref{prop:age}, $C:=C' \cup C''$ is the right-hand side of \eqref{odd:case}. To see that $\cC_P(\R^n)$ is a subset of $C$ it suffices to observe that, for every $x \in \R^n$, one has $x^P = ( - |x^\beta| , x^{\gamma(0)},\ldots, x^{\gamma(k)})$ or $x^P = (|x^\beta| , x^{\gamma(0)},\ldots, x^{\gamma(k)})$, we see that for $v = x^P$, one of the two convex constraints that define $C$ is attained with equality. This concludes the proof of the equality $\cC_P(\R^n) = C$. 
\end{proof} 

We can dualize the cones from Proposition~\ref{sonc:convexification}. 

\begin{prop} \label{SONC:non:neg} 
	In the notation of Proposition~\ref{sonc:convexification}, for $\beta \in 2 \N^n$, one has 
	\[
	\cP_P(\R^n) = \left\{ f \in \R[x]_P \colon f_\beta + \prod_{i=0}^k \left( \frac{f_{\gamma(i)}}{\lambda_i} \right)^{\lambda_i} \ge 0  \right\}
	\] 
	and, for $\beta \not\in 2 \N^n$, one has 
	\[
	\cP_P(\R^n) = \left\{ f \in \R[x]_P \colon -|f_\beta| + \prod_{i=0}^k \left( \frac{f_{\gamma(i)}}{\lambda_i} \right)^{\lambda_i} \ge 0  \right\}.
	\]
\end{prop} 
\begin{proof} 
	If $\beta \in 2 \N^n$, then we have $\cC_P(\R^n) = \cC_P(\R_+^n)$. Hence $\cP_P(\R^n) = \cC_P(\R^n)^\ast = \cC_P(\R_+^n)^\ast = \cP_P(\R_+^n)$, where $\cP_P(\R_+^n)$ was determined in Proposition~\ref{prop:sage:pol:cone}. 
	
	Assume $\beta \not\in 2 \N^n$. We recall  that $x^{\gamma(i)}$ does not depend on the signs of $x_1,\ldots,x_n$ because $\gamma(i) \in 2 \N^n$, while $x^{\beta}$ has a sign that can be switched by switching the sign of an $x_i$ with odd $\beta_i$. 
	This shows that $f_\beta x^\beta + \sum_{i=0}^k f_{\gamma(i)} x^{\gamma(i)}$ is non-negative on $\R^n$ if and only if $ - |f_\beta| x^\beta + \sum_{i=0}^k f_{\gamma(i)} x^{\gamma(i)}$ is non-negative on $\R_+^n$. 
\end{proof} 

If $\gamma(0,\ldots,\gamma(k)$ are affinely independent exponent vectors that belong to $2 \N^n$ and $\beta \in \N^n$ lies in the relative interior of the simplex with the vertices $\gamma(0),\ldots,\gamma(k)$, we call $P = \{\beta, \gamma(0),\ldots,\gamma(k)\}$ a \emph{SONC circuit}. We call \eqref{C-P-RLX} on the domain $X =\R^n$ where $\cF = \{P_1,\ldots,P_N\}$ is a family of SONC circuits a \emph{SONC moment convexification} of the unconstrained problem $\inf_{x \in \R^n} f(x)$, and we call the respective problem \eqref{D-P-RLX} a \emph{SONC relaxation} of the unconstrained problem $\inf_{x \in \R^n} f(x)$. Polynomials from $\cP_{P_1}(\R^n) + \cdots + \cP_{P_N}(\R^n)$, where $P_1,\ldots,P_N$ are SONC circuits, are called \emph{SONC polynomials}. 

From Proposition~\ref{SONC:non:neg}, we obtain an explicit description of the SONC relaxations for the problem $\inf_{x \in \R^n} f(x)$. 

\begin{rem}
	Comments in Remark~\ref{compute:SAGE} regarding the computational options for SAGE relaxations also apply for SONC relaxations. 
\end{rem} 

We call a pattern of the form $\{ 2 \alpha , \alpha + \beta, 2 \beta \}$ with $\alpha, \beta \in \N^n$ and $\alpha \ne \beta$ an SDSOS pattern. This is a special SONC pattern with $\alpha + \beta = \lambda_0 (2 \alpha) + \lambda_1 (2 \beta)$ and $\lambda_0= \lambda_1 = \frac{1}{2}$. 

\begin{prop}
	Let $P = \{2 \alpha, \alpha + \beta, 2 \beta\} \subseteq \N^n$ be an SDSOS pattern. Then, for $\alpha + \beta \in 2 \N^n$, one has  
	\begin{align*} 
		\cC_P(\R^n) & = \{ v \in \R_+^P \colon 0 \le v_{\alpha+\beta}, \ v_{\alpha+\beta}^2 \le v_{2 \alpha} v_{2 \beta} \}, 
		\\ \cP_P(\R^n) & = \{ f \in \R[x]_P \colon f_{\alpha+\beta} + 2 (f_{2\alpha} \cdot f_{2 \beta})^{1/2} \ge 0 \}. 
	\end{align*} 
	Furthermore, for $\alpha + \beta \not\in 2 \N^n$, one has 
	\begin{align*} 
		\cC_P(\R^n) & = \{ v \in \R_+^P \colon v_{\alpha+\beta}^2 \le v_{2 \alpha} v_{2 \beta} \}, 
		\\	\cP_P(\R^n) & =  \{ v \in \R^P \colon f^2_{\alpha+\beta} \le   2 f_{2\alpha} \cdot f_{2 \beta}  \}.
	\end{align*} 
\end{prop} 
\begin{proof} 
	The assertion follows from Propositions~\ref{sonc:convexification} and \ref{SONC:non:neg} applied to the special circuit $P$ with $k=1$ and $\lambda_0 = \lambda_1 = \frac{1}{2}$. 
\end{proof} 

Analogously to SONC relaxations, we define \emph{SDSOS moment convexifications} and \emph{SDSOS relaxations} of the problem $\inf_{x \in \R^n} f(x)$. If $P_1,\ldots, P_N$ are SDSOS patterns, then we call a polynomial from $\cP_{P_1}(\R^n) + \cdots + \cP_{P_N}(\R^n)$ an \emph{SDSOS polynomial}. 

\subsubsection{Sparse SOS certificates for unconstrained problems} 

Let $\bS_+^m$ be the cone of positive semidefinite matrices within the vector space $\bS^m$ of symmetric matrices of size $m$.  We introduce the Euclidean structure on $\bS^m$ by endowing it with the trace scalar product $\sprod{U}{V} = \tr(UV)$. The following is well known: 

\begin{prop} \label{sdp:self-dual}
	The cone $\bS_+^m$ is self-dual. That is, $(\bS_+^m)^\ast = \bS_+^m$
\end{prop} 
\begin{proof} 
	$\bS_+^m$ is the conic hull of the rank-one matrices $u u^\top$ with $u \in \R^m$. Hence, $(\bS_+^m)^\ast$ consists of matrices $V$ that satisfy $\sprod{V}{u u^\top} \ge 0$ for every $u \in \R^m$. Since $\sprod{V}{u u^\top} = u^\top V u$, the latter condition is exactly the definition of positive semidefiniteness, which gives the assertion. 
\end{proof} 

A polynomial that is a sum of squares of polynomials from $\R[x]$ is called a \emph{SOS polynomial}. We introduce the closed convex cone 
\[
\Sigma_B := \left\{ \sum_{i=1}^t f_i^2 \colon t \in \N, \ f_1,\ldots,f_t \in \R[x]_B\right\}
\]
of sparse SOS polynomials with the sparsity determined by $B$. Clearly, $\Sigma_B \subseteq \cP_{B+B}(\R^n)$, but the inclusion is strict in general.

\begin{prop} \label{sos:vs:moments} 
	Let $B \subseteq \N^n$ be a finite set of $m$ exponent vectors. Then, $\Sigma_B$ is the image 
	\begin{equation} \label{sigma:image:sdp}
		\Sigma_B = \{ (x^B)^\top M x^B \colon M \in \bS_+^m \} 
	\end{equation}
	of $\bS_+^m$ under the linear map $M \in \bS_+^m \mapsto (x^B)^\top M x^B \in \R[x]_{B+B}$, while $\Sigma_B^\ast$ is given with $M_B = x^B (x^B)^\top$ by 
	\begin{equation} \label{sigma:dual} 
		\Sigma_B^\ast = \left\{ v \in \R^{B+B} \colon L_v(M_B) \succeq 0 \right\},
	\end{equation}
	which means that $\Sigma_B^\ast$ is linearly isomorphic to a cross-section of $\Sigma_+^m$ by a linear sub-space, with the isomorphy determined by the injective map $v \in \R^{B+B} \mapsto L_v(M_B) \in \bS_+^m$. 
\end{prop} 
\begin{proof} 
	A square of a polynomial $f(x) = f^\top x^B \in \R[x]_B$ can be written as $$f(x)^2 = (f^\top x^B) ((x^B)^\top f) = f^\top x^B (x^B)^\top f = f^\top M_B f = \sprod{M_B}{f f^\top},$$
	where in vector expressions involving $f$, we use the identification $\R[x]_B \leftrightarrow \R^B$. 
	Consequently, $f_1^2 + \cdots + f_t^2 = \sprod{M_B}{Q}$, where $Q =\sum_{i=1}^t f_i f_i^\top \succeq 0$. Conversely, every symmetric positive-semidefinite matrix $Q$ of size $|B|$ can be decomposed into $\sum_{i=1}^t f_i f_i^\top$ for some $f_1,\ldots,f_t$, which shows that $\sprod{M_B}{Q} = \sum_{i=1}^t f_i^2 \in \Sigma_B$. This shows \eqref{sigma:image:sdp}.
	
	Since $\bS_+^m$ is self-dual, a symmetric matrix $M$ is positive semidefinite if and only if $\sprod{M}{Q} \ge 0$ holds for every symmetric positive semidefinite matrix $Q$. We thus obtain 
	\begin{align*}
		\left\{ v \in \R^{B+B} \colon L_v(M_B) \succeq 0 \right\} & = \left\{ v \in \R^{B+B} \colon \sprod{L_v(M_B)}{Q} \ge 0 \ \text{for all} \ Q \succeq 0 \right\}
		\\ & = \left\{ v \in \R^{B+B} \colon L_v(\sprod{M_B)}{Q}) \ge 0 \ \text{for all} \ Q \succeq 0 \right\}
		\\ & = \left\{ v \in \R^{B+B} \colon L_v ( g) \ge 0 \ \text{for all} \ g \in \Sigma_B\right\}
		\\ & = \Sigma_B^\ast. 
	\end{align*} 
	which concludes the proof. 
\end{proof} 

In most cases, we do not have an explicit description in any of the available optimization paradigms for the cones $\cP_{B+B}(\R^n)$ which we use in \eqref{D-P-RLX}. In contrast, we have semidefinite representations for the cones $\Sigma_{B}$, as explained by Proposition~\ref{sos:vs:moments}. Thus, SOS cones along with the so-called positivstellens\"atze, which explain how positivity of polynomials can be certified in terms of SOS polynomials, provide a theoretical foundation for semidefinite approaches to polynomial optimization. Such approaches were pioneered by Lasserre \cite{lasserre2001global}. 

For $f \in \R[x]_{B+B}$, we call $\sup \{ f \in \R[x]_{B+B} \colon f -\lambda \in \Sigma_B \}$ an \emph{SOS relaxation} of the unconstrained problem $\inf_{x \in \R^n} f(x)$ with respect to the exponent set $B$. Furthermore, we call $\sup \{ L_v(f) \colon v_0 = 1, \ L_v( M_B) \succeq 0 \}$ the \emph{semidefinite moment relaxation} of $\inf_{x \in \R^n} f(x)$ with respect to the exponent set $B$. The dense versions of these relaxations are obtained by taking $B = \N_d^n$. 

We formulate the duality between SOS and a semidefinite moment relaxation for the unconstrained polynomial optimization problem. 

\begin{prop} \label{sos:moments:duality} 
	For a finite exponent set $B \subseteq \N^n$ and $f \in \R[x]_{B+B}$, we have 
	\[
	\sup \{ \lambda \in \R \colon f -\lambda \in \Sigma_B \}  = \inf \{ L_v(f) \colon v_0=1, \ L_v(M_B) \succeq 0 \}. 
	\]
\end{prop} 
\begin{proof} 
	The assertion follows by applying Proposition~\ref{abstract:duality} to the pointed closed convex cone $C = \Sigma_B$ and using the description \eqref{sigma:dual} of $(\Sigma_B)^\ast$. 
\end{proof} 

From a computational perspective, the issue with the primal-dual pair of problems in Proposition~\ref{sos:moments:duality} is that they involve semidefinite constraints with matrices of size $|B|$. When $|B|$ is large, the respective computations can become intractable, as discussed in Section~\ref{sec_paradigms}. To cope with this, one can establish a sparse version of Proposition~\ref{sos:moments:duality}. 

\begin{prop} \label{sparse:sos}
	Let $B_1,\ldots,B_N \in \N^n$ be finite exponent sets and let $f \in \R[x]_A$ with $A = (B_1 +B_1) \cup \cdots \cup (B_N+B_N)$. Then 
	\[
	\sup \{ \lambda \in \R \colon f - \lambda \in \Sigma_{B_1} + \cdots + \Sigma_{B_N} \}  = \inf \{ L_v(f) \colon v_0 = 1,  L_v(M_{B_i}) \succeq 0  \ \forall \ i\in [N]\}. 
	\]
\end{prop} 
\begin{proof} 
	Clearly, $\Sigma_{B_1} + \cdots + \Sigma_{B_N}$ is a subset of the pointed convex cone $\cP_A(\R^n)$, so that that the assumptions of Proposition~\ref{sparse:abstract:duality} are true for $C_i = \Sigma_{B_i}$. Hence, the assertion follows from Proposition~\ref{sparse:abstract:duality} by taking into  account the expression \eqref{sigma:dual} for the dual of the SOS cone. Note that we need to dualize $\Sigma_{B_i}$ as a cone in $\R[x]_A$. In this case, \eqref{sigma:dual} is only slightly modified to $(\Sigma_{B_i})^\ast = \{ v \in \R^A \colon L_v(M_{B_i}) \succeq 0\}$, by using $\R^A$ rather than $\R^{B_i+B_i}$. 
\end{proof} 

The primal-dual pair from Proposition~\ref{sparse:sos} is a relaxed model for the pair of the problems \eqref{D-P-RLX} and \eqref{C-P-RLX} using the cones $\Sigma_{B_i} \subseteq \cP_{B_i+B_i}(\R^n)$ and $\{v \in \R^{B_i+B_i} \colon L_v(M_{B_i}) \succeq 0\} \supseteq \cC_{B_i+B_i}(\R^n)$. Thus, Proposition~\ref{sparse:sos} corresponds to a pattern relaxation for the pattern family $\cF = \{B_i+B_i \colon i \in [N]\}$. 

We call the two problems in Proposition~\ref{sparse:sos} the \emph{sparse SOS relaxation} and the \emph{sparse semidefinite moment relaxation} of $\inf_{x \in \R^n} f(x)$ with respect to the exponent sets $B_1,\ldots,B_N$. 

\subsubsection{TSSOS} \label{sec_TSSOS}

The intention of the TSSOS approach \cite{MR4198579} is to sparsify SOS relaxations without losing quality. That is, one wants to replace an SOS relaxation by an equivalent sparse SOS relaxation. 
Assume we want to solve the unconstrained problem $\inf_{x \in \R^n} f(x)$ with $f \in \R[x]_A$. We fix an SOS relaxation with respect to some set $B$ satisfying $A \subseteq B+B$. Having $A$ and $B$, we now establish a sparse SOS relaxation with respect to some $B_1,\ldots,B_N$ that is equivalent to the SOS relaxation with respect to $B$ for all polynomials $f$ with the support contained in $A$.  That is, one wants to have 
\begin{equation} \label{TSSOS:wish}
	\inf \{ \lambda \in \R \colon f - \lambda \in \Sigma_B \} = \inf \{ \lambda \in \R \colon f - \lambda \in \Sigma_{B_1} + \cdots + \Sigma_{B_N} \}  \ \forall \; \ f \in R[x]_A.
\end{equation} 

Whenever the maximum size of the $B_i$'s is smaller than the size of $B$, the sparse SOS relaxation is usually more tractable than the SOS relaxation with respect to $B$. 

The TSSOS approach is an iterative algorithm that gets $A$ and $B$ with $A \subseteq B+B$ as an input and performs updates on a partition $\cB = \{B_1,\ldots,B_N\}$ of $B$. In each iteration, the partition gets coarser (that is, $N$ gets smaller), and the algorithm terminates in finitely many steps, providing a partition $\cB$ that satisfies \eqref{TSSOS:wish}. 

To explain the algorithm, we use undirected graphs and we view a graph as a set of its edges. For example, $G = \{ \{0,1\}, \{1,2\} ,\{1,3\} \}$ is a graph with three edges $\{0,1\}$, $\{1,2\}$ and $\{1,3\}$ between nodes from the vertex set  $\{0,1,2,3\}$. 

The algorithm starts by fixing $S_0 := A \cup 2 B$. In the $i$-th iteration one considers the graph $G_i := \{ \{\alpha,\beta\} \colon \alpha,\beta \in \N^n, \ \alpha+ \beta \in S_{i-1}\}$. One defines 
\[
S_i = \{ \alpha + \beta \colon \alpha,\beta  \ \text{nodes of $G_i$ connected by a path} \},
\]
and sets the partition $\cB_i$ to be the family of connected components (vertex sets) of $G_i$. $\cB_i$ stabilizes after finitely many iterations and upon stabilization one has \eqref{TSSOS:wish} with $\cB_i = \{B_1,\ldots,B_N\}$. 

\begin{rem}
	How much one has gained by employing TSSOS, depends crucially on whether the size of the sets $B_1,\ldots,B_N$ is small. Understanding which choices of $A$ lead to sparse TSSOS reformulations $\Sigma_B \cap \R[x]_A$ is an interesting open problem.  
\end{rem} 

\begin{rem} 
	The basic TSSOS approach was generalized to constrained polynomial optimization and modified to the so-called Chordal-TSSOS approach \cite{MR4198579} and the CS-TSSOS approach. 
\end{rem} 

\begin{rem} 
	A refined version of the TSSOS method has recently been suggested \cite{Shaydurova2024}. In rTSSOS, refined partitions are determined via integer programming, allowing a better fine-tuning between approximation quality and computational costs. 
\end{rem} 

\subsubsection{Optimization of sparse univariate polynomials over $\R_+$}

Averkov and Scheiderer have shown that optimization of sparse univariate polynomials over $\R_+$ can be formulated exactly via a sparse SOS relaxation \cite[Proposition~4.3]{averkov2024convex}. 

\begin{thm} \label{thm_averkovscheiderer}
	For every univariate polynomial $f \in \R[x]_A$, where 
	$A \subseteq \N$ is a finite set of exponents containing $0$, of odd size $|A| = 2 k+ 1$,	and with $d = \max (A) > 2k$ one has 
	\begin{align*} 
		\inf_{x \in \R_+} f(x) & = \sup \{ \lambda \colon f- \lambda \in x^0 \Sigma_B + \dots + x^{d-2k} \Sigma_B \} 
		\\ & = \inf \{ L_v(f) \colon v_0 =1, \ L_v( x^i M_B ) \succeq 0 \ \text{for all} \ i \in 0,\ldots,d-2k \}, 
	\end{align*} 
	where $B = \{0,\ldots,k\}$. 
\end{thm} 

Theorem~\ref{thm_averkovscheiderer} implies that if a univariate polynomial $f$ has $2k+1$ terms and an arbitrarily large degree $d$, then the pattern relaxation with respect to  the family of patterns $\{i,..,i+2k\}$, where $0\le i \le d - 2k$, for the optimization of $f$ on the domain $\R_+$ gives the exact optimal value of \eqref{POP}.

\subsection{Lifted conic formulations} 

In Subsection~\ref{existing:sparse:conic}, we discussed conic relaxations of polynomial optimization problems that are based on conic lifted representations. In this subsection, we present some abstract framework of lifted formulations and provide a summary of how the different liftings from Subsection~\ref{existing:sparse:conic} are related to each other. 

\subsubsection{Abstract lifted formulations} 

\begin{defn} \label{def:lift} 
	Consider a closed convex cone $D \subseteq \R^n$.
	We say that a subset $S$ of the finite-dimensional vector space $\R^m$ over $\R$ admits a $D$-lifted representation if 
	\begin{equation} \label{lift} 
		S = \{ M x \colon x \in D, \ A x = b \}
	\end{equation} 
	for some linear maps $M \in \R^{m \times n}$, $A \in \R^{k \times m}$ and $b \in \R^k$.  Geometrically, this means that $S$ is a linear image of an affine slice of $D$. 
\end{defn} 

\begin{rem}
	Representing $S$ as an image $S = M D = \{ M x \colon x \in D\}$ of some cone $D$ is a special case of a $D$-lifted formulation with $k=0$.  Furthermore, representing $S$ as a cross-section $S = \{ x \in D \colon Ax =b\}$ of $D$ by the affine subspace $\{x \colon Ax = b\}$ is a special $D$-lifted formulation with $M$ being the identity matrix. 

	Another way of writing the constraints $x \in D, \ A x =b$ is by parametrizing the affine space $\{ x \colon A x = b\}$ as $x = v - U y $ for some matrix $U$ and a vector $v$. This allows to represent the constraint $x \in D, \ A x = b$ as $v - U y \in D$ and $M x$ is represented as $M x = M v - M U y$. 

When $S$ in Definition~\ref{def:lift} is a convex cone, one can replace $b$ by the zero vector in the $D$-lifted representation of $S$ so that $S$ is represented as a linear image of a slice of $D$ by a \emph{linear} subspace. 
\end{rem} 

In Subsection~\ref{existing:sparse:conic} we had a number of situations, in which we replaced the cones $\cC_{P_i}(X)$ in \eqref{C-P-RLX} by smaller cones $C_i$ that are images of some cones $D_i$ under a linear map. The following abstract proposition underpins the duality for the respective sparse conic relaxation of \eqref{C-P-RLX}. 

\begin{prop} \label{abstract:sparse:duality:images} 
	Let $p,q \in \R^n$, let $C_1,\ldots,C_N \subseteq \R^n$ be pointed closed convex cones and $C:= C_1 +  \cdots + C_N$, a pointed  closed convex cone that does not contain $q$. Assume that for each $i \in [N]$ the cone  $C_i$ admits a formulation $C_i := M_i D_i := \{ M_i y \colon y \in D_i\}$ as an image of a closed convex cone $D_i$ under a matrix $M_i$. Then the problem $\sup \{ \lambda \in \R \colon p-\lambda q \in C\}$ can be formulated as 
	\begin{align*} 
		& \sup \{ \lambda \in \R \colon p - \lambda q \in M_1 D_1 + \cdots + M_N D_N \} 
		\\ = & \inf \{ \sprod{p}{v}  \colon \sprod{q}{v} = 1, \  M_i^\top v \in D_i^\ast \ \forall \;  i \in [N] \}.
	\end{align*} 
\end{prop} 
\begin{proof} 
	The assertion follows from Proposition~\ref{sparse:abstract:duality} by taking into account that $(M_i D_i)^\ast = \{ v \in \R^n \colon M_i^\top v \in D_i^\ast \}$. 
\end{proof} 

\subsubsection{Examples of primal dual pairs with lifting} 

We look at some special cases of duality.

\begin{prop} \label{primal-dual:lp} 
	Let polynomials $g_1,\ldots,g_s, f \in \R[x]_B$ be given for some finite set $B \subseteq \N^n$ with $0 \in B$. Then
	\begin{align*}
		& \sup \{ \lambda \in \R \colon f - \lambda \in g_1  \R_+ \cdots + g_s \R_+ \} 
		\\ = &  \inf \{ L_v(f) \colon v \in \R^B, \ v_0 = 1, \ L_v(g_i) \ge 0 \ \forall \; i \in [s] \}
	\end{align*} 
	and this primal-dual pair of problems provides a lower bound on the problem of the minimization of $f(x)$ subject to $g_1 \ge 0,\ldots, g_s \ge 0$. 
\end{prop} 
\begin{proof} 
	We use Proposition~\ref{abstract:sparse:duality:images}, which explains  general sparse duality with lifting, for $A_i = g_i$ and $D_i = \R_+$, where we see $g_i(x)$ as a column vector in $\R^B$. The minimization problem involves the constraints  $A_i^\top v \in D_i^\ast$, where $A_i^\top v = \sprod{g_i}{v} = L_v(g_i)$ and $D_i^\ast = (\R_+)^\ast = \R_+$.  This gives the assertion. 
\end{proof} 

\begin{rem}
	The primal-dual pair in Proposition~\ref{primal-dual:lp} is a pair of linear problems. If one knows some inequalities $h_1,\ldots,h_N$ valid on $X \subseteq \R^n$, one can use these linear problems to relax $\inf_{x \in X} f(x)$. In particular, it is possible to add the dual linear problem to the relaxations arising from the expression trees, see Section~\ref{sec_expression}.
\end{rem} 

\begin{rem}
	Proposition~\ref{primal-dual:lp} in combination with Handelman's theorem~\ref{handelman:thm} gives a primal-dual pair of linear problems for the minimization $\inf_{x \in K} f(x)$ of a polynomial $f$ over a polytope $K \subseteq \R^n$. If $g_1 \ge 0,\ldots, g_m \ge 0$ is a system of linear inequalities that describes $K$, one can choose $h_i = g^{\alpha(i)}$ with $\alpha(i) \in \N^n$. Handelman's theorem guarantees that when $\{\alpha(i) \colon i \in [N]\} = \N_d^n$ and $d$ is large enough, the primal-dual pair approximates $\inf_{x \in K} f(x)$ arbitrarily well. 
\end{rem} 

\begin{rem}
	For the primal and dual SAGE relaxation, based on a family of SAGE patterns $P_1,\ldots,P_N$ with $B := P_1 \cup \cdots \cup P_N$, 
	the cones $D_i$ are GMC cones $\GMC_{\lambda(i)}$ with $\lambda(i) \in \R_{>0}^{k_i+2}$, and the matrix $A_i$ provides a coordinate embedding of $\R^{k_i+2} \widehat{=} \R^{P_i}$ into $\R^{B}$ that bijectively sends $\R^{P_i}$ to $\R^{P_i} \times \{0\}^{B \setminus P_i}$ by appending zero components. Accordingly, $A_i^\top$ describes the coordinate projection of $\R^B$ onto $\R^{P_i} \widehat{=} \R^{k_i+2}$. Analogous comments also apply to the SONC, SDSOS, and sparse SOS relaxations. The underlying cone for the SDSOS relaxation is $\GMC_{1/2,1/2}$. 
	
For sparse SOS relaxations, Proposition~\ref{abstract:sparse:duality:images} provides an abstract version of  Proposition~\ref{sparse:sos}, where $A_i$ corresponds to the linear map $Q \in \bS^{m_i} \mapsto (x^{B_i})^\top Q x^{B_i} \in \R[x]_A$, with $m_i = |B_i|$, and $A_i^\top$ corresponds to the adjoint of this map, which is the map   $v \in \R^A \rightarrow L_v( M_{B_i}) \in \bS^{m_i}$. 
\end{rem} 

\subsubsection{Relevant conic-optimization paradigms}

\begin{rem} 
	Conic programming with respect to the cone $\R_+^n$ is linear programming. A $\R_+^n$-lifted formulation is also called a \emph{linear extended formulation}. 
	It is known that a set $X$ admits a linear extended formulation if and only if $X$ is a polyhedron. This follows from the fact that polyhedrality is preserved under taking linear images. Thus, linear programming convexification techniques correspond to conic optimization for the cone $\R_+^n$. 
\end{rem} 

\begin{defn} 
	We call $(\bS_+^2)^m$ a \emph{second-order cone} and we call a set $S$ \emph{second-order cone representable} if it admits a $(\bS_+^2)^m$-lift. We define second-order cone programming as optimization of a linear objective function subject to finitely many LMIs $M_1(x) \succeq 0,\ldots,M_m(x) \succeq 0$ with matrices of size $2$. 
\end{defn} 

\begin{rem} 
	Polyhedra are second-order cone representable, since every linear inequality $\sprod{a}{x} + b \ge 0$ in the variables $x$ can be rewritten as a linear matrix inequality $\begin{pmatrix} \sprod{a}{x} + b & 0 \\ 0 & 0\end{pmatrix} \succeq 0$ with a matrix of size $2$. On the other hand, not every second-order cone representable set is a polyhedron: consider for example the unit disc, described as  $\begin{pmatrix} 1 + x_1 & x_2 \\ x_2 & 1 -x_2 \end{pmatrix} \succeq 0$. 
\end{rem} 

\begin{defn} 
	A $\bS_+^m$-lifted formulation is also called a \emph{semidefinite lifted formulation} or a \emph{semidefinite extended formulation}. We call a set that admits a semidefinite extended formulation \emph{semidefinitely representable}. 
\end{defn} 

\begin{rem} 
	Every second-order cone representable set is semidefinitely representable, because $(\bS_+^2)^m$ corresponds to a cross-section of $\bS_+^{2m}$. On the other hand, Hamza Fawzi showed that there exist semidefinitely representable sets that are not second-order cone representable \cite{fawzi2019representing}. This result was generalized by Averkov \cite{averkov2019optimal} and later generalized even further by Saunderson \cite{saunderson2020limitations}. 
\end{rem} 

There is a well-developed theory of computational methods for solving linear, second-order cone and semidefinite optimization problems, with the methods implemented by various solvers. So, for each relaxation of a polynomial optimization problem that relies on linear, second-order cone or semidefinite lifted formulations, it is possible to use the available computational machinery to carry out computations in practice. 

\begin{rem} 
For SAGE and SONC relaxations we use lifts with respect to the Cartesian products 
\[
\GMC_{\gamma(1)} \times \cdots \times \GMC_{\gamma(m)}
\]  
of the GMC cones. For conic programming with respect to GMC cones computational approaches are also available. Alternatively, one can pass from GMC cones to second-order cones, using a result of Ben-Tal and Nemirovski \cite{ben2001lectures}, which shows that $\GMC_{\lambda}$ is second-order cone reprentable. 
\end{rem}

%% file: 4_computations.tex
\section{Computations} \label{sec_computations}

In this section we shall investigate numerically the impact of various forms of convexification and sparsification on computational runtime and on the quality of approximations. We 
describe which algorithms were evaluated with which criteria in Section~\ref{sec_setup}, and describe the benchmark test set of polynomial optimization problems in Section~\ref{sec_benchmark}. The numerical results are then presented and discussed in Section~\ref{sec_results}.

\subsection{Convexification and sparsification approaches} \label{sec_setup}

We are interested in different ways of convexification and in a comparison to existing solvers for polynomial optimization problems based on either primal or dual approaches. In all of our numerical comparisons we used the following software. The abbreviations will be used in all following figures.

\begin{itemize} 
  \item[\underline{Method}] \underline{Description} 
	\item[B] We used the software \baron~1.8.9 \cite{sahinidis:baron:17.8.9,tawarmalani2005polyhedral} with default settings, called from \matlab. \baron is a mature solver, based on branch-and-bound and polyhedral relaxations. We enforced a 1000 second time limit, as discussed further below. 
	\item[R] Root node relaxation of \baron, indicating the computational costs and approximation quality of the polyhedral relaxation without branch-and-bound based refinements. 
	\item[Y] We used the software \yalmip 20200930 \cite{yalmip} as a \matlab toolbox. It allows to compute the moment as well as the SOS relaxation of \eqref{POP}. We used \yalmip's \texttt{solvesos} at the lowest possible level of the SOS hierarchy with \mosek 9.2.32 \cite{mosek}.
	\item[SOS] In study \ref{subsec:num-res-sparseB}, method Y could not provide a solution due to memory issues. We replaced it with a custom implementation of Lasserre's SDP relaxation at the lowest possible level of the hierarchy. It gave qualitatively similar results on other test problems (data not shown).
	\item[CS] We used the software \cstssos 1.00, based on the (dual) TSSOS approach \cite{MR4198579}, compare Section~\ref{sec_TSSOS}. \cstssos is a \julia package that allows to exploit correlative sparsity and term sparsity simultaneously. We called the first level of the hierarchy by running the command \texttt{cs\_tssos\_first} with settings \texttt{order} $=\lceil\tfrac{\deg(f)}{2}\rceil$ and \texttt{TS="MD"}.
	\item[SIG] The \python package \sageopt 0.5.3 \cite{chandrasekaran2016relative,katthan2021unified} allows to compute relaxations based on SONC polynomials as well as SAGE relaxations based on signomials. For SONC relaxation no finite lower bounds could be obtained for any of the instances, all shown results correspond to SAGE relaxations, compare Section~\ref{sec_SAGE}. We used the method \texttt{sig\_constrained\_relaxation} on the lowest possible hierarchy level. 
\end{itemize}

\medskip
The rationale behind the selection of these methods was the variety of primal and dual approaches: \baron for global NLP approaches relying on polyhedral relaxations, \yalmip as an implementation of the dense Lasserre's semidefinite relaxation, \cstssos as a particular way to sparsify the latter, and \sageopt as an alternative approach relying on the circuit patterns.
Note that more solvers to solve \eqref{POP} globally exist, such as \texttt{SCIP} \cite{scip}, \texttt{COUENNE} \cite{belotti2015couenne}, or \texttt{LINDOGlobal}. As they are based on polyhedral relaxations, too, we concentrated on a comparison with \baron.
We are not aware of any other source, where such a broad comparison across a wide range of approaches was carried out.

In different numerical studies in Section~\ref{sec_results} we compared the approaches listed above to implementations of pattern relaxations from Section~\ref{sec_specific} as specified below. The code for solving the pattern relaxation \eqref{P-RLXa} was implemented and run in \matlab~9.10.0.1669831 (R2021a) Update 5 \cite{matlab}, consists of roughly 3500 lines of code, and uses \mosek 9.2.32 \cite{mosek} to solve the relaxations. We used the following pattern relaxations in order of appearance.

\begin{itemize} 
  \item[\underline{Method}] \underline{Description} 
	\item[M] Multilinear pattern relaxation using $\cF_A^{m} := \{ \{0,\alpha_1\} \times \cdots \times \{0,\alpha_n\} \colon \alpha \in A\}$.  Here, and in what follows, as ap preprocessing step we only keep the inclusion-maximal patterns of the family and leave out all the other patterns. 
	\item[C] Chain relaxation for the pattern family $\cF_A^{c} := \{ C_{\alpha}\alpha \colon \alpha \in A \} $ given by 
	\[
	C_\alpha := \left \{ k \frac{\alpha}{\gcd(\alpha)} \colon k=0,\ldots, 2 \lceil \gcd(\alpha)/2 \rceil  \right \},
	\]
	where $\gcd(\alpha)$ denotes the greatest common divisor of the components of $\alpha$. 
	\item[H] Pattern relaxation for the pattern family $\cF_A^h$ constructed as follows. Choose the minimal $d \in \N$ such that $A \subseteq \{0,\ldots, 2d \}^n$ and add to $\cF_A^h$ the patterns 
$	\{ k e_i \colon k=0,\ldots, 2d\} $
	for each $i\in[n]$, and $ \{ (k,\ldots,k) \colon k=0,\ldots, 2d\}$, and $\{0,k\} \times \cdots \times \{0,k\}$ for every $k=0,\ldots, 2d$, and all patterns from $\cF_A^m$. 
	\item[S] Pattern relaxation for the family $\cF_A^{s} := \{ S_{\alpha,i} \colon i \in [n], \alpha \in A\}$   
	\[
	S_{\alpha,i } := \{ (\alpha_1,\ldots, \alpha_{i-1},  k , \alpha_{i+1},\ldots,\alpha_n) \colon k =0,\ldots, 2 \lceil \alpha_i/2 \rceil \}
	\]
	of shifted axis-parallel chains containing elements of $A$.
	\item[MC] Pattern relaxation for the pattern family $\cF_A^{mc} := \cF_A^m \cup \cF_A^c$. 
	\item[T] Pattern relaxation for the pattern family $\cF_A^t$ consisting of truncated submonoid patterns constructed as follows. Fix a truncated submonoid pattern $P = 2 \, \N_{2 \lceil \deg(A) / 4 \rceil} ^n$. For each $\alpha \in \N^n$, consider the pattern 
	\(
	P_\alpha = \Gamma \N_{2 \lceil \deg(A) / 2 \rceil}^n
	\)
	given by the matrix $\Gamma = (\gamma(1),\ldots,\gamma(n))$ with $\gamma(i) = e_i$ if $\alpha_i >0$ and $\gamma(i) = 0$ otherwise. Define $\cF_A^t := \{P \} \cup \{P_\alpha \colon \alpha \in A \setminus P \}$. This pattern is designed to give more tractable Lasserre-type relaxations for low degrees $d$.   
\end{itemize} 

\subsection{Test set and evaluation criteria} \label{sec_benchmark}

Establishing a meaningful test set for evaluating approaches to \eqref{POP} is a nontrivial task. Regarding the choice of the objective $f$, there are low-degree and high-degree instances, there are dense instances, instances with a specific structural sparsity and instances with a random sparsity, instances with many variables and instance with a small number of variables, structured instances occurring in particular applications and randomly generated instances. There are many ways to combine the above properties so that a thorough evaluation would require a large number of different types of instances. Evaluation on such a diverse test set would allow us to understand on which types of instances a particular approach should be applied. 

We are not aware of any library of polynomial optimization instances that exhibits such a desired degree of diversity. Thus, we randomly generated a test set of instances of different types and complemented it with some structured examples. 
Our test sets evaluates $12$ choices of $A$, classified into four types: dense exponent sets, sparse sets, sparse sets with large dimension $n$ and small degree $d$, and specifically structured sets. The latter is composed of exponent sets that are designed to give specific insight for some of the tested relaxation approaches.

In all calculations, we fix the feasible set $X$ to the box $K= [0,1]^n$.
For each $A$, we create a sample of twenty polynomials by sampling coefficients for $f$ using a uniform distribution in $[-1,1]^A$. 
However, samples are only included in the test set, when they are sufficiently difficult. To quantify difficulty, we used the performance of method B. Only when either the minimization problem $\min_{x \in K} f(x)$ or the maximization problem $\min_{x \in K} - f(x)$ can not be solved completely by \baron within the time limit of 1000 seconds, do we keep $f$. This procedure explains why the reported runtimes of \baron are always at least 1000 seconds and only allow to indicate the potential of improvement of \baron for difficult cases. This preselection is motivated by the basic idea that we primarily need to improve the performance of methods on instances that are currently hard for established solution techniques.

The first evaluation criterium is the approximation quality of the relaxation. To consider different directions of optimization over the feasible set $K$, we consider minimization and maximization problem. We use the normalized value
\begin{equation} \label{eq_criterium}
\triv_{\cF}(f,K) = \frac{\max_{v \in \RLX_\cF(K)} L_v(f) - \min_{v \in \RLX_\cF(K)} L_v(f)}{\max_{v \in \Bx(l_A,u_A)} L_v(f) - \min_{v \in \Bx(l_A,u_A)} L_v(f)} \in [0,1] 
\end{equation}
as an indicator of the tightness of the calculated relaxations in comparison to the trivial bounds obtained from the bounds $(l_A,u_A)$ on the monomials occuring in $f$. Values close to $1$ indicate that the bounds obtained via relaxation $\RLX_\cF(K)$ do not improve much over the trivial bounds in the denominator, whereas values close to $0$ indicate that $\RLX_\cF(K)$ is a good relaxation.



The second reported criteria is computational time, provided as the wall-clock time in seconds, averaged over all 20 samples.
The solvers were run on a compute server with 4 Intel(R) Xeon(R) Gold 6138 CPUs with 20 cores of 2 threads and 1 TB RAM each under Ubuntu 20.04.3.
Each solver-instance pair was assigned to one such job, i.e., the solvers themselves did not use the parallel structure.
In order to distribute the solver-instance pairs to the 80 cores we used \cite{tange_2020_4118697}.

\subsection{Numerical results} \label{sec_results}

We investigate the performance of different methods with respect to approximation quality evaluated via \eqref{eq_criterium} and computational run time. We structure the analysis according to properties of the exponent sets $A$.

\subsubsection{Dense Exponent Sets}\label{subsec:dense} 
We first consider dense exponent sets $A=\Nnd$ for $n\in\{2,4\}$ and $d=10$. Results are shown in Figure~\ref{boxplot:dense}.

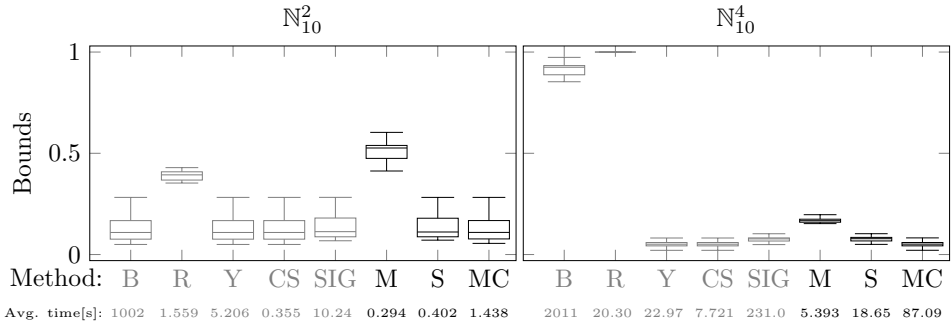
\begin{figure}[h!!!]
	\centering
	\scalebox{0.96}{  \input{figures/fig-boxplot-dense.tex}}
	\caption{Results for dense exponent sets $A=\Nnd$. Shown is the median and standard variation for 20 sample coefficient vectors $f$ of the evaluation criterium \eqref{eq_criterium} for the methods specified in Section~\ref{sec_setup}. Note that a value close to $0$ indicates a good relaxation. The average computational time in seconds is provided below each method in the bottom row. Note that only instances are selected for which \baron needs the full time limit on either minimization or maximization problem.} \label{boxplot:dense}
\end{figure}

The root relaxation R of \baron can be solved in 1.5 and 20 seconds, respectively. The solution B using branch-and-bound is constrained on purpose by the imposed time limit of 1000 seconds for minimization and maximization problem. During this time a certain number of nodes can be processed, which leads to a significant reduction in the case $\N^2_{10}$. For the higher dimensional polynomials in $\N^4_{10}$ the solution time per node is increased, though. The reduced number of nodes that can be processed in the time limit do only result in a modest improvement of the bounds.

The methods Y, CS, SIG, S, and MC provide good bounds. In comparison, the runtime for SIG is not very competetive.
The multilinear polyhedral relaxation M can be solved fast, but results in weaker bounds, as expected. The bounds are better for $n=4$ compared to R, though, probably due to a larger pattern size and hence more connections between monomial variables compared to the relaxation \baron uses.  

\subsubsection{Sparse Exponent Sets}\label{subsec:num-res-sparse}
We use randomly generated sparse exponent sets $A=S(n,d)$ to test pattern families that do not assume any structure of $A$.
$S(n,d)$ is generated by randomly picking $\left\lceil\sqrt{\tbinom{n+d}{d}}\right\rceil$ exponents via \texttt{randperm} from $\N^{n}_{d}$. Results are shown in Figure~\ref{boxplot:sparse}.

\begin{figure}[h!!!]
	\centering
	\scalebox{0.96}{  \input{figures/fig-boxplot-sparse.tex}}
	\caption{As in Figure~\ref{boxplot:dense}, but for sparse exponent sets $S(n,d)$.}\label{boxplot:sparse}
\end{figure}
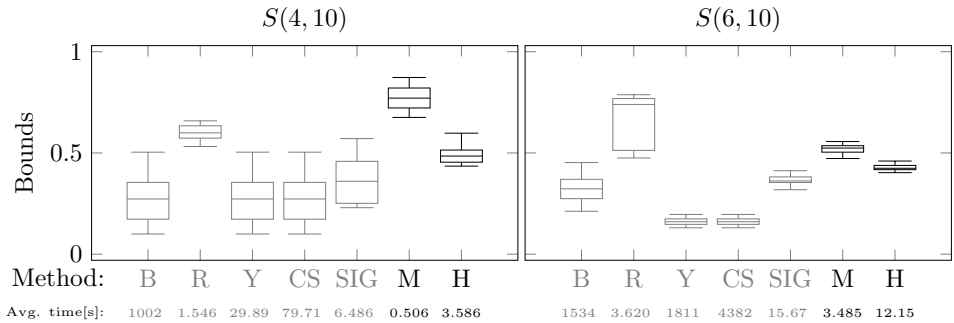

In comparison to \baron, the semidefinite relaxations Y and CS provide very good results also in $n=6$ dimensions, albeit at a high computational price. In comparison, SIG, M, and H provide weaker bounds, but up to three orders of magnitude faster. Method M based on multilinear patterns results in weak bounds, which can be strengthened considerably with method H at some computational expense by additionally enforcing indirect connections between moment variables via $n+1$ chains and $d$ multilinear patterns.

\subsubsection{Sparse Exponent Sets with $n > d$}\label{subsec:num-res-sparseB}

As in Section~\ref{subsec:num-res-sparse} we use sparse exponent sets, however now with a high number of variables $n=20,25,30,40$ and low degree $d=4$. 
Results are shown in Figure~\ref{boxplot:sparse-d4}.

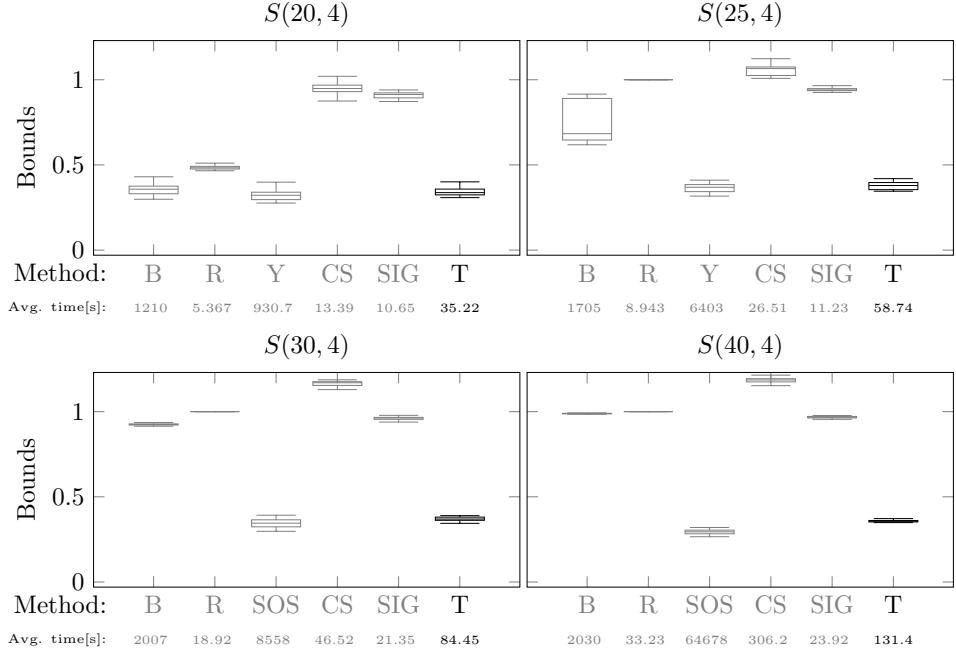
\begin{figure}[h!!!]
	\centering
	\scalebox{0.96}{  \input{figures/fig-boxplot-sparse-d4-1.tex}}
	\scalebox{0.96}{  \input{figures/fig-boxplot-sparse-d4-2.tex}}
	\caption{As in Figure~\ref{boxplot:dense}, but for sparse exponent sets $S(n,d)$ with $n > d$.}\label{boxplot:sparse-d4}
\end{figure}

The difficulty of these instances becomes apparent from the weak bounds close to $1$ provide by method R and for $n=40$ also for B. Remember that a value of $1$ in \eqref{eq_criterium} corresponds to relaxations that have no advantage over the trivial lower bounds on the monomial variables.
Excellent bounds are provided by the Y / SOS methods. However, the computational time increases drastically and memory issues arise for Y, validating the theoretical considerations in Section~\ref{sec_paradigms}. It becomes clear that alternatives are necessary for problems in higher dimensions.
It is interesting that methods CS and SIG run comparitively fast, but do not result in good bounds.

One such candidate is method T, which is tailored for problems with low degree and many variables and profits computationally from the small size of the involved LMIs. The very good bounds provided by T can be calculated two orders of magnitude faster than with \yalmip. This allowed to also compute nontrivial bounds for instances with exponent sets $S(80,4)$ in approximately 400 seconds.

\subsubsection{Specifically structured test sets} \label{sec_structuredresults}

To complement the results on dense and sparse exponent sets, we consider specific exponent sets with particular properties. We start with adversarial instances. 
Assume for example that we want to solve \eqref{POP} with $A = \{(k,\ldots,k ) \colon k=0,\ldots, 2d \}$, where $d \in \N$. This is actually a disguised one-dimensional problem, since we can make a substitution $y = x_1 \cdots x_n$. Thus, we need patterns that would link the monomial variables with the exponents in $A$.
Since $A$ is a chain, choosing $A$ as a pattern is optimal. This means, C is an optimal method for this $A$.  But if we use M, we would link the monomials $x_1^i \cdots x_n^i$ with the variables $x_1,\ldots,x_n$ by convex constraints, which is not helpful. So the above support is an adversarial choice for method M. By looking at such kinds of special $A$'s we are able to understand which are the preferred choices of $A$ for the methods we evaluate and on which choices of $A$ the methods have difficulties. We test four exponent sets based on chain pattern families,
\begin{align*} 
A_5 & = \{ (k, k) \colon k=0,\ldots, 10\}  \subset \N^2_{10} \\
A_6 & = \{ (k, k, k, k) \colon k=0,\ldots, 10\}  \subset \N^4_{10} \\
A_7 & = \{ k \alpha \colon k=0,\ldots, 10, \ \alpha \in \{e_1, e_2, e_1 + e_2\} \} \subset \N^2_{10} \\
A_8 & = \{ k \alpha \colon k=0,\ldots, 10, \ \alpha \in \{e_1, e_2, e_3, e_4, e_1 + e_2 + e_3 + e_4\} \} \subset \N^4_{10}
\end{align*}
Results are shown in Figure~\ref{boxplot:adversary}.

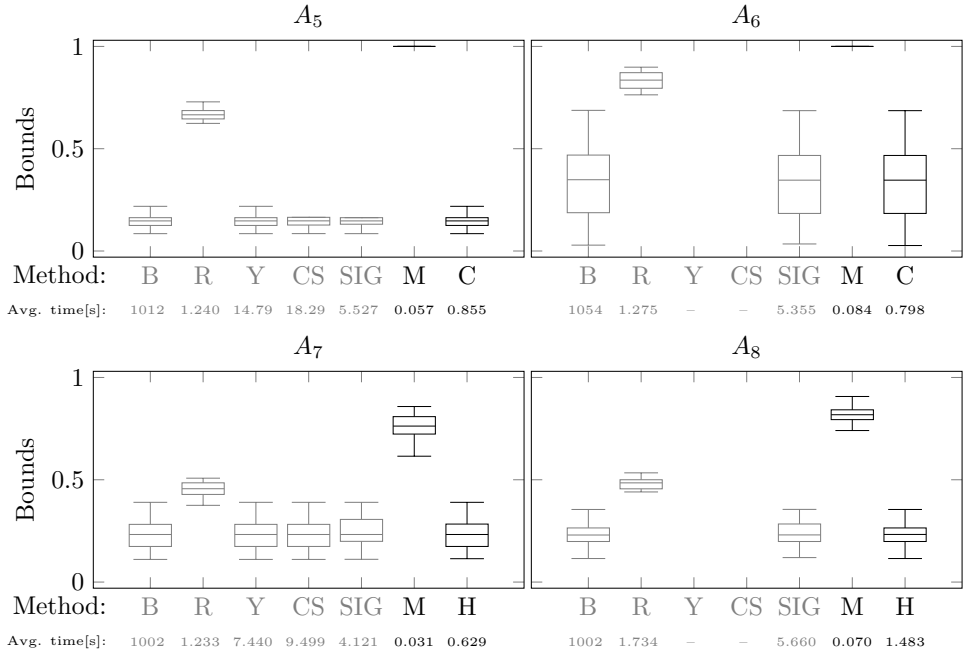
\begin{figure}[htbp]
	\centering
	\scalebox{0.97}{\input{figures/fig-boxplot-adversary-chains.tex}}
	\scalebox{0.97}{\input{figures/fig-boxplot-adversary-dpa.tex}}
	\caption{As in Figure~\ref{boxplot:dense}, but for specific adversarial exponent sets $A_5, A_6, A_7, A_8$.}\label{boxplot:adversary}
\end{figure}

As a main take-away, the exploitation of chain patterns in C and H brings a huge advantage in comparison to the multilinear pattern method M. The resulting bounds are similar to those produced by B and SIG, but significantly faster. This shows that the structure of the exponent set may be decisive for the performance of particular methods. Detecting such structures a priori in an automatized, adaptive way, seems to be one of the future challenges in polynomial optimization.

Findally, we study the exponent set $\Aex$ as defined in Example~\ref{rmk:pattern-plot} and visualized in Figure~\ref{fig_patterns}. 
In addition to the solvers and methods M and H, we also apply three custom pattern families $F^1$, $F^2$, and $F^3$ as introduced in the fifth row of Figure~\ref{fig_patterns}. Results are shown in Figure~\ref{boxplot:custom}.

\begin{figure}[h!!!]
	\centering
	\input{figures/fig-boxplot-custom.tex}
	\caption{As in Figure~\ref{boxplot:dense}, but for the exponent set $\Aex$ and additional custom pattern families $F^1$, $F^2$, and $F^3$ as introduced in Example~\ref{rmk:pattern-plot}.}\label{boxplot:custom}
\end{figure}
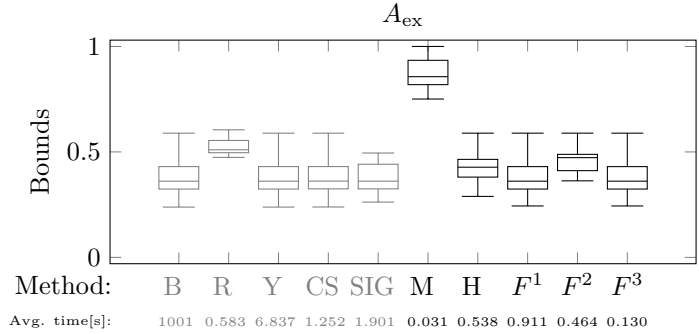

One observes that particular choices of pattern families, such as $F^3$, may bring additional advantages with respect to runtime and approximation quality when compared to method H.

%% file: figures/fig-boxplot-dense.tex
\begin{tikzpicture}
\begin{groupplot}[group style={
                  group size=2 by 1,
	              xticklabels at = edge bottom,
	              yticklabels at = edge left,
	              ylabels at     = edge left,
	              xlabels at     = edge bottom,
	              horizontal sep = 0.1cm, 
	              vertical sep   = 0.3cm},
                  xlabel style={align=right},
	              ylabel= Bounds,
	              boxplot/hide outliers/.code={\def\pgfplotsplothandlerboxplot@outlier{}
	              },
                  ymin = -0.03,
                  ymax = 1.03,
                  xmin = 0.2, 
                  xmax = 8.53,
	              xticklabel style       = {align=center},
                  height  = \textwidth*(0.35),
                  width   = \textwidth*(0.575),
                  boxplot/draw direction = y,
                  title style={yshift=-1ex,},
                  every boxplot/.style={mark=none,every mark/.append style={mark size=1.5pt}}
                  ]
  \nextgroupplot[
  xlabel style = {at={(axis description cs:-0.085,0.005)},}, 
  xlabel = {Method: \\ \tiny{Avg. time[s]:}},
  title={$\N^{2}_{10}$}, 
  xtick={1,2,3,4,5,6,7,8,9}, 
  xticklabels={
  {\color{gray}B    \\ \color{gray}\tiny{1002}	},
  {\color{gray}R    \\ \color{gray}\tiny{1.559}	},
  {\color{gray}Y    \\ \color{gray}\tiny{5.206}	},
  {\color{gray}CS   \\ \color{gray}\tiny{0.355}	},
  {\color{gray}SIG  \\ \color{gray}\tiny{10.24} },
  {\color{black}M   \\ \tiny{0.294}     			},
  {\color{black}S   \\ \tiny{0.402}     			}, 
  {\color{black}MC  \\ \tiny{1.438}     			}, 
  {\color{black}H   \\ \tiny{1.093}     			}}]
 
  \addplot+ [boxplot, mark=none, solid, gray]  table[row sep=newline,y index=0] {figures/data/dual-monorels/dense/dense_n2d10_B1000_width.dat};
  \addplot+ [boxplot, mark=none, solid, gray]  table[row sep=newline,y index=0] {figures/data/dual-monorels/dense/dense_n2d10_BRoot_width.dat};
  \addplot+ [boxplot, mark=none, solid, gray]  table[row sep=newline,y index=0] {figures/data/dual-monorels/dense/dense_n2d10_YAL_width.dat};
  \addplot+ [boxplot, mark=none, solid, gray]  table[row sep=newline,y index=0] {figures/data/dual-monorels/dense/dense_n2d10_CSTS_width.dat};
  \addplot+ [boxplot, mark=none, solid, gray]  table[row sep=newline,y index=0] {figures/data/dual-monorels/dense/dense_n2d10_SOPT_sig_width.dat};
  \addplot+ [boxplot, mark=none, solid, black] table[row sep=newline,y index=0] {figures/data/dual-monorels/dense/dense_n2d10_M_width.dat};
  \addplot+ [boxplot, mark=none, solid, black] table[row sep=newline,y index=0] {figures/data/dual-monorels/dense/dense_n2d10_S_width.dat};
  \addplot+ [boxplot, mark=none, solid, black] table[row sep=newline,y index=0] {figures/data/dual-monorels/dense/dense_n2d10_MCS_width.dat};
  \addplot+ [boxplot, mark=none, solid, black] table[row sep=newline,y index=0] {figures/data/dual-monorels/dense/dense_n2d10_Str1_width.dat};
  
  \nextgroupplot[
  title={$\N^{4}_{10}$}, 
  xtick={1,2,3,4,5,6,7,8,9}, 
  xticklabels={
  {\color{gray}B    \\ \color{gray}\tiny{2011}	},
  {\color{gray}R    \\ \color{gray}\tiny{20.30}	},
  {\color{gray}Y    \\ \color{gray}\tiny{22.97}	},
  {\color{gray}CS   \\ \color{gray}\tiny{7.721}	},
  {\color{gray}SIG  \\ \color{gray}\tiny{231.0}	},
  {\color{black}M   \\ \tiny{5.393}     			},
  {\color{black}S   \\ \tiny{18.65}     			},
  {\color{black}MC  \\ \tiny{87.09}     			}, 
  {\color{black}H   \\ \tiny{4.817}     			}}]
 
  \addplot+ [boxplot, mark=none, solid, gray]  table[row sep=newline,y index=0] {figures/data/dual-monorels/dense/dense_n4d10_B1000_width.dat};
  \addplot+ [boxplot, mark=none, solid, gray]  table[row sep=newline,y index=0] {figures/data/dual-monorels/dense/dense_n4d10_BRoot_width.dat};
  \addplot+ [boxplot, mark=none, solid, gray]  table[row sep=newline,y index=0] {figures/data/dual-monorels/dense/dense_n4d10_YAL_width.dat};
  \addplot+ [boxplot, mark=none, solid, gray]  table[row sep=newline,y index=0] {figures/data/dual-monorels/dense/dense_n4d10_CSTS_width.dat};
  \addplot+ [boxplot, mark=none, solid, gray]  table[row sep=newline,y index=0] {figures/data/dual-monorels/dense/dense_n4d10_SOPT_sig_width.dat};
  \addplot+ [boxplot, mark=none, solid, black] table[row sep=newline,y index=0] {figures/data/dual-monorels/dense/dense_n4d10_M_width.dat};
  \addplot+ [boxplot, mark=none, solid, black] table[row sep=newline,y index=0] {figures/data/dual-monorels/dense/dense_n4d10_S_width.dat};
  \addplot+ [boxplot, mark=none, solid, black] table[row sep=newline,y index=0] {figures/data/dual-monorels/dense/dense_n4d10_MCS_width.dat};
  \addplot+ [boxplot, mark=none, solid, black] table[row sep=newline,y index=0] {figures/data/dual-monorels/dense/dense_n4d10_Str1_width.dat};
\end{groupplot}
\end{tikzpicture}

%% file: figures/fig-boxplot-sparse.tex
\begin{tikzpicture}
\begin{groupplot}[group style={
                  group size=2 by 1,
	              xticklabels at = edge bottom,
	              yticklabels at = edge left,
	              ylabels at     = edge left,
	              xlabels at     = edge bottom,
	              horizontal sep = 0.1cm, 
	              vertical sep   = 0.3cm},
                  xlabel style={align=right},
	              ylabel= Bounds,
                  height  = \textwidth*(0.35),
                  width   = \textwidth*(0.575),
                  boxplot/draw direction = y,
                  ymin    = -0.03,
                  ymax    = 1.03,
	              xticklabel style       = {align=center},
                  title style={yshift=-1ex,},
                  every boxplot/.style={mark=none,every mark/.append style={mark size=1.5pt}}
                  ]
  \nextgroupplot[
  xlabel style = {at={(axis description cs:-0.085,0.005)},}, 
  xlabel = {Method: \\ \tiny{Avg. time[s]:}}, 
  title={$\rmS(4,10)$}, xtick={1,2,3,4,5,6,7}, 
  xticklabels={
  {\color{gray}B  \\ \color{gray}\tiny{1002}  },
  {\color{gray}R  \\ \color{gray}\tiny{1.546} },
  {\color{gray}Y  \\ \color{gray}\tiny{29.89} }, 
  {\color{gray}CS \\ \color{gray}\tiny{79.71} }, 
  {\color{gray}SIG\\ \color{gray}\tiny{6.486} }, 
  {\color{black}M \\ \tiny{0.506}             },
  {\color{black}H \\ \tiny{3.586}             }}]

  \addplot+ [boxplot, mark=none, solid, gray]   table[row sep=newline,y index=0] {figures/data/dual-monorels/sparse/sparse_sqrt_n4d10_B1000_width.dat};
  \addplot+ [boxplot, mark=none, solid, gray]   table[row sep=newline,y index=0] {figures/data/dual-monorels/sparse/sparse_sqrt_n4d10_BRoot_width.dat};
  \addplot+ [boxplot, mark=none, solid, gray]   table[row sep=newline,y index=0] {figures/data/dual-monorels/sparse/sparse_sqrt_n4d10_YAL_width.dat};
  \addplot+ [boxplot, mark=none, solid, gray]   table[row sep=newline,y index=0] {figures/data/dual-monorels/sparse/sparse_sqrt_n4d10_CSTS_width.dat};
  \addplot+ [boxplot, mark=none, solid, gray]   table[row sep=newline,y index=0] {figures/data/dual-monorels/sparse/sparse_sqrt_n4d10_SOPT_sig_width.dat};
  \addplot+ [boxplot, mark=none, solid, black]  table[row sep=newline,y index=0] {figures/data/dual-monorels/sparse/sparse_sqrt_n4d10_M_width.dat};
  \addplot+ [boxplot, mark=none, solid, black]  table[row sep=newline,y index=0] {figures/data/dual-monorels/sparse/sparse_sqrt_n4d10_Str1_width.dat};
  
  \nextgroupplot[
  title={$\rmS(6,10)$}, 
  xtick={1,2,3,4,5,6,7}, 
  xticklabels={
  {\color{gray}B  \\ \color{gray}\tiny{1534}  	},
  {\color{gray}R  \\ \color{gray}\tiny{3.620} 	},
  {\color{gray}Y  \\ \color{gray}\tiny{1811} 	}, 
  {\color{gray}CS \\ \color{gray}\tiny{4382}  	},
  {\color{gray}SIG \\ \color{gray}\tiny{15.67}  },
  {\color{black}M \\ \tiny{3.485}             	},
  {\color{black}H \\ \tiny{12.15}             	}}]
 
  \addplot+ [boxplot, mark=none, solid, gray]   table[row sep=newline,y index=0] {figures/data/dual-monorels/sparse/sparse_sqrt_n6d10_B1000_width.dat};
  \addplot+ [boxplot, mark=none, solid, gray]   table[row sep=newline,y index=0] {figures/data/dual-monorels/sparse/sparse_sqrt_n6d10_BRoot_width.dat};
  \addplot+ [boxplot, mark=none, solid, gray]   table[row sep=newline,y index=0] {figures/data/dual-monorels/sparse/sparse_sqrt_n6d10_YAL_width.dat};
  \addplot+ [boxplot, mark=none, solid, gray]   table[row sep=newline,y index=0] {figures/data/dual-monorels/sparse/sparse_sqrt_n6d10_CSTS_width.dat};
  \addplot+ [boxplot, mark=none, solid, gray]   table[row sep=newline,y index=0] {figures/data/dual-monorels/sparse/sparse_sqrt_n6d10_SOPT_sig_width.dat};
  \addplot+ [boxplot, mark=none, solid, black]  table[row sep=newline,y index=0] {figures/data/dual-monorels/sparse/sparse_sqrt_n6d10_M_width.dat};
  \addplot+ [boxplot, mark=none, solid, black]  table[row sep=newline,y index=0] {figures/data/dual-monorels/sparse/sparse_sqrt_n6d10_Str1_width.dat};
\end{groupplot}
\end{tikzpicture}

%% file: figures/fig-boxplot-sparse-d4-1.tex
\begin{tikzpicture}
\begin{groupplot}[group style={
                  group size=2 by 1,
	              xticklabels at = edge bottom,
	              yticklabels at = edge left,
	              ylabels at     = edge left,
	              xlabels at     = edge bottom,
	              horizontal sep = 0.1cm, 
	              vertical sep   = 0.3cm},
                  xlabel style={align=right},
	              ylabel= Bounds,
                  height  = \textwidth*(0.35),
                  width   = \textwidth*(0.575),
                  boxplot/draw direction = y,
                  ymin    = -0.03,
                  ymax    = 1.23,
	              xticklabel style       = {align=center},
                  title style={yshift=-1ex,},
                  every boxplot/.style={mark=none,every mark/.append style={mark size=1.5pt}}
                  ]
  \nextgroupplot[
  xlabel style = {at={(axis description cs:-0.085,0.005)},}, 
  xlabel = {Method: \\ \tiny{Avg. time[s]:}}, 
  title={$\rmS(20,4)$}, xtick={1,2,3,4,5,6}, 
  xticklabels={
  {\color{gray}B    \\ \color{gray}\tiny{1210}	},
  {\color{gray}R    \\ \color{gray}\tiny{5.367}	},
  {\color{gray}Y    \\ \color{gray}\tiny{930.7}	},
  {\color{gray}CS   \\ \color{gray}\tiny{13.39}	},
  {\color{gray}SIG  \\ \color{gray}\tiny{10.65}	}, 
  {\color{black}T   \\             \tiny{35.22}	}}]
 
  \addplot+ [boxplot, mark=none, solid, gray]   table[row sep=newline,y index=0] {figures/data/dual-monorels/sparse_big/sparse_sqrt_n20d04_B1000_width.dat};
  \addplot+ [boxplot, mark=none, solid, gray]   table[row sep=newline,y index=0] {figures/data/dual-monorels/sparse_big/sparse_sqrt_n20d04_BRoot_width.dat};
  \addplot+ [boxplot, mark=none, solid, gray]   table[row sep=newline,y index=0] {figures/data/dual-monorels/sparse_big/sparse_sqrt_n20d04_YAL_width.dat};
  \addplot+ [boxplot, mark=none, solid, gray]   table[row sep=newline,y index=0] {figures/data/dual-monorels/sparse_big/sparse_sqrt_n20d04_CSTS_width.dat};
  \addplot+ [boxplot, mark=none, solid, gray]   table[row sep=newline,y index=0] {figures/data/dual-monorels/sparse_big/sparse_sqrt_n20d04_SOPT_sig_width.dat};
  \addplot+ [boxplot, mark=none, solid, black]  table[row sep=newline,y index=0] {figures/data/dual-monorels/sparse_big/sparse_sqrt_n20d04_TSD4_width.dat};
  
  \nextgroupplot[
  title={$\rmS(25,4)$}, 
  xtick={1,2,3,4,5,6}, 
  xticklabels={
  {\color{gray}B    \\ \color{gray}\tiny{1705}	},
  {\color{gray}R    \\ \color{gray}\tiny{8.943}	},
  {\color{gray}Y    \\ \color{gray}\tiny{6403}	},
  {\color{gray}CS   \\ \color{gray}\tiny{26.51}	},
  {\color{gray}SIG  \\ \color{gray}\tiny{11.23}	}, 
  {\color{black}T   \\             \tiny{58.74}	}}]
 
  \addplot+ [boxplot, mark=none, solid, gray]   table[row sep=newline,y index=0] {figures/data/dual-monorels/sparse_big/sparse_sqrt_n25d04_B1000_width.dat};
  \addplot+ [boxplot, mark=none, solid, gray]   table[row sep=newline,y index=0] {figures/data/dual-monorels/sparse_big/sparse_sqrt_n25d04_BRoot_width.dat};
  \addplot+ [boxplot, mark=none, solid, gray]   table[row sep=newline,y index=0] {figures/data/dual-monorels/sparse_big/sparse_sqrt_n25d04_YAL_width.dat};
  \addplot+ [boxplot, mark=none, solid, gray]   table[row sep=newline,y index=0] {figures/data/dual-monorels/sparse_big/sparse_sqrt_n25d04_CSTS_width.dat};
  \addplot+ [boxplot, mark=none, solid, gray]   table[row sep=newline,y index=0] {figures/data/dual-monorels/sparse_big/sparse_sqrt_n25d04_SOPT_sig_width.dat};
  \addplot+ [boxplot, mark=none, solid, black]  table[row sep=newline,y index=0] {figures/data/dual-monorels/sparse_big/sparse_sqrt_n25d04_TSD4_width.dat};

\end{groupplot}
\end{tikzpicture}

%% file: figures/fig-boxplot-sparse-d4-2.tex
\begin{tikzpicture}
\begin{groupplot}[group style={
                  group size=2 by 1,
	              xticklabels at = edge bottom,
	              yticklabels at = edge left,
	              ylabels at     = edge left,
	              xlabels at     = edge bottom,
	              horizontal sep = 0.1cm, 
	              vertical sep   = 0.3cm},
                  xlabel style={align=right},
	              ylabel= Bounds,
                  height  = \textwidth*(0.35),
                  width   = \textwidth*(0.575),
                  boxplot/draw direction = y,
                  ymin    = -0.03,
                  ymax    = 1.23,
	              xticklabel style       = {align=center},
                  title style={yshift=-1ex,},
                  every boxplot/.style={mark=none,every mark/.append style={mark size=1.5pt}}
                  ]
  
  \nextgroupplot[
  xlabel style = {at={(axis description cs:-0.085,0.005)},}, 
  xlabel = {Method: \\ \tiny{Avg. time[s]:}}, 
  title={$\rmS(30,4)$}, xtick={1,2,3,4,5,6}, 
  xticklabels={
  {\color{gray}B    \\ \color{gray}\tiny{2007}	},
  {\color{gray}R    \\ \color{gray}\tiny{18.92}	},
  {\color{gray}SOS  \\ \color{gray}\tiny{8558}	},
  {\color{gray}CS   \\ \color{gray}\tiny{46.52}	}, 
  {\color{gray}SIG  \\ \color{gray}\tiny{21.35}	},
  {\color{black}T   \\             \tiny{84.45}	}}]
 
  \addplot+ [boxplot, mark=none, solid, gray]   table[row sep=newline,y index=0] {figures/data/dual-monorels/sparse_big/sparse_sqrt_n30d04_B1000_width.dat};
  \addplot+ [boxplot, mark=none, solid, gray]   table[row sep=newline,y index=0] {figures/data/dual-monorels/sparse_big/sparse_sqrt_n30d04_BRoot_width.dat};
  \addplot+ [boxplot, mark=none, solid, gray]   table[row sep=newline,y index=0] {figures/data/dual-monorels/sparse_big/sparse_sqrt_n30d04_SOS_width.dat};
  \addplot+ [boxplot, mark=none, solid, gray]   table[row sep=newline,y index=0] {figures/data/dual-monorels/sparse_big/sparse_sqrt_n30d04_CSTS_width.dat};
  \addplot+ [boxplot, mark=none, solid, gray]   table[row sep=newline,y index=0] {figures/data/dual-monorels/sparse_big/sparse_sqrt_n30d04_SOPT_sig_width.dat};
  \addplot+ [boxplot, mark=none, solid, black]  table[row sep=newline,y index=0] {figures/data/dual-monorels/sparse_big/sparse_sqrt_n30d04_TSD4_width.dat};
  
  \nextgroupplot[
  title={$\rmS(40,4)$}, 
  xtick={1,2,3,4,5,6}, 
  xticklabels={
  {\color{gray}B    \\ \color{gray}\tiny{2030}	},
  {\color{gray}R    \\ \color{gray}\tiny{33.23}	},
  {\color{gray}SOS  \\ \color{gray}\tiny{64678}	},
  {\color{gray}CS   \\ \color{gray}\tiny{306.2}	}, 
  {\color{gray}SIG  \\ \color{gray}\tiny{23.92}	}, 
  {\color{black}T   \\             \tiny{131.4}	}}]
 
  \addplot+ [boxplot, mark=none, solid, gray]   table[row sep=newline,y index=0] {figures/data/dual-monorels/sparse_big/sparse_sqrt_n40d04_B1000_width.dat};
  \addplot+ [boxplot, mark=none, solid, gray]   table[row sep=newline,y index=0] {figures/data/dual-monorels/sparse_big/sparse_sqrt_n40d04_BRoot_width.dat};
  \addplot+ [boxplot, mark=none, solid, gray]   table[row sep=newline,y index=0] {figures/data/dual-monorels/sparse_big/sparse_sqrt_n40d04_SOS_width.dat};
  \addplot+ [boxplot, mark=none, solid, gray]   table[row sep=newline,y index=0] {figures/data/dual-monorels/sparse_big/sparse_sqrt_n40d04_CSTS_width.dat};
  \addplot+ [boxplot, mark=none, solid, gray]   table[row sep=newline,y index=0] {figures/data/dual-monorels/sparse_big/sparse_sqrt_n40d04_SOPT_sig_width.dat};
  \addplot+ [boxplot, mark=none, solid, black]  table[row sep=newline,y index=0] {figures/data/dual-monorels/sparse_big/sparse_sqrt_n40d04_TSD4_width.dat};

\end{groupplot}
\end{tikzpicture}

%% file: figures/fig-boxplot-adversary-chains.tex
\begin{tikzpicture}[]
\begin{groupplot}[
                  group style={
                  group size=2 by 1,
	              xticklabels at = edge bottom,
	              yticklabels at = edge left,
	              ylabels at     = edge left,
	              xlabels at     = edge bottom,
	              horizontal sep = 0.1cm, 
	              vertical sep   = 0.3cm
                  },
                  xlabel style={align=right},
	              ylabel= Bounds,
	              xticklabel style   = {align=center},
                  height  = \textwidth*(0.35),
                  width   = \textwidth*(0.575),
                  boxplot/draw direction = y,
                  ymin = -0.03,
                  ymax = 1.03,
                  title style={yshift=-1ex,},
                  every boxplot/.style={mark=none,every mark/.append style={mark size=1.5pt}}
                  ]
  \nextgroupplot[
  xlabel style = {
  at={(axis description cs:-0.085,0.005)},}, 
  xlabel = {Method: \\ \tiny{Avg. time[s]:}}, 
  title={$A_5$}, 
  xtick={1,...,7}, 
  xticklabels={
  {\color{gray}B  \\ \color{gray}\tiny{1012}  },
  {\color{gray}R  \\ \color{gray}\tiny{1.240} },
  {\color{gray}Y  \\ \color{gray}\tiny{14.79} },
  {\color{gray}CS \\ \color{gray}\tiny{18.29} },
  {\color{gray}SIG\\ \color{gray}\tiny{5.527} },
  {\color{black}M \\ \tiny{0.057}             },
  {\color{black}C \\ \tiny{0.855}             }}] 

  \addplot+ [boxplot, mark=none, solid, gray]   table[row sep=newline,y index=0] {figures/data/dual-monorels/d_chain/chain_n2d10_B1000_width.dat};
  \addplot+ [boxplot, mark=none, solid, gray]   table[row sep=newline,y index=0] {figures/data/dual-monorels/d_chain/chain_n2d10_BRoot_width.dat};
  \addplot+ [boxplot, mark=none, solid, gray]   table[row sep=newline,y index=0] {figures/data/dual-monorels/d_chain/chain_n2d10_YAL_width.dat};
  \addplot+ [boxplot, mark=none, solid, gray]   table[row sep=newline,y index=0] {figures/data/dual-monorels/d_chain/chain_n2d10_CSTS_width.dat};
  \addplot+ [boxplot, mark=none, solid, gray]   table[row sep=newline,y index=0] {figures/data/dual-monorels/d_chain/chain_n2d10_SOPT_sig_width.dat};
  \addplot+ [boxplot, mark=none, solid, black]  table[row sep=newline,y index=0] {figures/data/dual-monorels/d_chain/chain_n2d10_M_width.dat};
  \addplot+ [boxplot, mark=none, solid, black]  table[row sep=newline,y index=0] {figures/data/dual-monorels/d_chain/chain_n2d10_C_width.dat};
  
  \nextgroupplot[
  title={$A_6$}, 
  xtick={1,...,7}, 
  xticklabels={
  {\color{gray}B  \\ \color{gray}\tiny{1054}  },
  {\color{gray}R  \\ \color{gray}\tiny{1.275} },
  {\color{gray}Y  \\ \color{gray}\tiny{--}    },
  {\color{gray}CS \\ \color{gray}\tiny{--}    },
  {\color{gray}SIG\\ \color{gray}\tiny{5.355} },
  {\color{black}M \\ \tiny{0.084}             },
  {\color{black}C \\ \tiny{0.798}             }}] 
  \addplot+ [boxplot, mark=none, solid, gray]   table[row sep=newline,y index=0] {figures/data/dual-monorels/d_chain/chain_n4d10_B1000_width.dat};
  \addplot+ [boxplot, mark=none, solid, gray]   table[row sep=newline,y index=0] {figures/data/dual-monorels/d_chain/chain_n4d10_BRoot_width.dat};
  \addplot+ [boxplot, mark=none, solid, gray]   table[row sep=newline,y index=0] {figures/data/dual-monorels/d_chain/chain_n4d10_YAL_width.dat};
  \addplot+ [boxplot, mark=none, solid, gray]   table[row sep=newline,y index=0] {figures/data/dual-monorels/d_chain/chain_n4d10_CSTS_width.dat};
  \addplot+ [boxplot, mark=none, solid, gray]   table[row sep=newline,y index=0] {figures/data/dual-monorels/d_chain/chain_n4d10_SOPT_sig_width.dat};
  \addplot+ [boxplot, mark=none, solid, black]  table[row sep=newline,y index=0] {figures/data/dual-monorels/d_chain/chain_n4d10_M_width.dat};
  \addplot+ [boxplot, mark=none, solid, black]  table[row sep=newline,y index=0] {figures/data/dual-monorels/d_chain/chain_n4d10_C_width.dat};
 
\end{groupplot}
\end{tikzpicture}

%% file: figures/fig-boxplot-adversary-dpa.tex
\begin{tikzpicture}[]
\begin{groupplot}[
                  group style={
                  group size=2 by 2,
	              xticklabels at = edge bottom,
	              yticklabels at = edge left,
	              ylabels at     = edge left,
	              xlabels at     = edge bottom,
	              horizontal sep = 0.1cm, 
	              vertical sep   = 0.3cm
                  },
                  xlabel style={align=right},
	              ylabel= Bounds,
	              xticklabel style   = {align=center},
                  height  = \textwidth*(0.35),
                  width   = \textwidth*(0.575),
                  boxplot/draw direction = y,
                  ymin = -0.03,
                  ymax = 1.03,
                  title style={yshift=-1ex,},
                  every boxplot/.style={mark=none,every mark/.append style={mark size=1.5pt}}
                  ]
  \nextgroupplot[
  xlabel style = {
  at={(axis description cs:-0.085,0.005)},}, 
  xlabel = {Method: \\ \tiny{Avg. time[s]:}   }, 
  title={$A_7$}, 
  xtick={1,...,7}, 
  xticklabels={
  {\color{gray}B  \\ \color{gray}\tiny{1002}   },
  {\color{gray}R  \\ \color{gray}\tiny{1.233}  },
  {\color{gray}Y  \\ \color{gray}\tiny{7.440}  },
  {\color{gray}CS \\ \color{gray}\tiny{9.499}  },
  {\color{gray}SIG\\ \color{gray}\tiny{4.121}  },
  {\color{black}M \\ \tiny{0.031}              },
  {\color{black}H \\ \tiny{0.629}              }}] 

  \addplot+ [boxplot, mark=none, solid, gray]   table[row sep=newline,y index=0] {figures/data/dual-monorels/dpa_chain/dpa_chain_n2d10_B1000_width.dat};
  \addplot+ [boxplot, mark=none, solid, gray]   table[row sep=newline,y index=0] {figures/data/dual-monorels/dpa_chain/dpa_chain_n2d10_BRoot_width.dat};
  \addplot+ [boxplot, mark=none, solid, gray]   table[row sep=newline,y index=0] {figures/data/dual-monorels/dpa_chain/dpa_chain_n2d10_YAL_width.dat};
  \addplot+ [boxplot, mark=none, solid, gray]   table[row sep=newline,y index=0] {figures/data/dual-monorels/dpa_chain/dpa_chain_n2d10_CSTS_width.dat};
  \addplot+ [boxplot, mark=none, solid, gray]   table[row sep=newline,y index=0] {figures/data/dual-monorels/dpa_chain/dpa_chain_n2d10_SOPT_sig_width.dat};
  \addplot+ [boxplot, mark=none, solid, black]  table[row sep=newline,y index=0] {figures/data/dual-monorels/dpa_chain/dpa_chain_n2d10_M_width.dat};
  \addplot+ [boxplot, mark=none, solid, black]  table[row sep=newline,y index=0] {figures/data/dual-monorels/dpa_chain/dpa_chain_n2d10_Str1_width.dat};
  
  \nextgroupplot[
  title={$A_8$}, 
  xtick={1,...,7}, 
  xticklabels={
  {\color{gray}B  \\ \color{gray}\tiny{1002}		},
  {\color{gray}R  \\ \color{gray}\tiny{1.734}	}, 
  {\color{gray}Y  \\ \color{gray}\tiny{--}  		},
  {\color{gray}CS \\ \color{gray}\tiny{--}  		},
  {\color{gray}SIG\\ \color{gray}\tiny{5.660}	},
  {\color{black}M \\ \tiny{0.070}           		},
  {\color{black}H \\ \tiny{1.483}           		}}] 
  
  \addplot+ [boxplot, mark=none, solid, gray]   table[row sep=newline,y index=0] {figures/data/dual-monorels/dpa_chain/dpa_chain_n4d10_B1000_width.dat};
  \addplot+ [boxplot, mark=none, solid, gray]   table[row sep=newline,y index=0] {figures/data/dual-monorels/dpa_chain/dpa_chain_n4d10_BRoot_width.dat};
  \addplot+ [boxplot, mark=none, solid, gray]   table[row sep=newline,y index=0] {figures/data/dual-monorels/dpa_chain/dpa_chain_n4d10_YAL_width.dat};
  \addplot+ [boxplot, mark=none, solid, gray]   table[row sep=newline,y index=0] {figures/data/dual-monorels/dpa_chain/dpa_chain_n4d10_CSTS_width.dat};
  \addplot+ [boxplot, mark=none, solid, gray]   table[row sep=newline,y index=0] {figures/data/dual-monorels/dpa_chain/dpa_chain_n4d10_SOPT_sig_width.dat};
  \addplot+ [boxplot, mark=none, solid, black]  table[row sep=newline,y index=0] {figures/data/dual-monorels/dpa_chain/dpa_chain_n4d10_M_width.dat};
  \addplot+ [boxplot, mark=none, solid, black]  table[row sep=newline,y index=0] {figures/data/dual-monorels/dpa_chain/dpa_chain_n4d10_Str1_width.dat};

\end{groupplot}
\end{tikzpicture}

%% file: figures/fig-boxplot-custom.tex
\begin{tikzpicture}
\begin{groupplot}[group style={
                  group size= 1 by 1,
	              xticklabels at = edge bottom,
	              yticklabels at = edge left,
	              ylabels at     = edge left,
	              xlabels at     = edge bottom,
	              horizontal sep = 0.3cm, 
	              vertical sep   = 0.3cm},
                  xlabel style={align=right},
	              ylabel= Bounds,
	              boxplot/hide outliers/.code={\def\pgfplotsplothandlerboxplot@outlier{}},
                  ymin = -0.03,
                  ymax = 1.03,
	              xticklabel style       = {align=center},
                  height  = \textwidth*(0.35),
                  width   = \textwidth*(0.72),
                  boxplot/draw direction = y,
                  title style={yshift=-1ex,},
                  every boxplot/.style={mark=none,every mark/.append style={mark size=1.5pt}}
                  ]
  \nextgroupplot[
  xlabel style = {at={(axis description cs:-0.085,-0.01)},}, 
  xlabel = {Method: \\ \tiny{Avg. time[s]:} },
  title={$\Aex$}, 
  xtick={1,2,3,4,5,6,7,8,9,10}, 
  xticklabels={
  {\color{gray}B$\phantom{^1}$  \\ \color{gray}\tiny{1001}    },
  {\color{gray}R$\phantom{^1}$  \\ \color{gray}\tiny{0.583}   },
  {\color{gray}Y$\phantom{^1}$  \\ \color{gray}\tiny{6.837}   },
  {\color{gray}CS$\phantom{^1}$ \\ \color{gray}\tiny{1.252}   },
  {\color{gray}SIG$\phantom{^1}$\\ \color{gray}\tiny{1.901}   },
  {\color{black}M$\phantom{^1}$ \\ \tiny{0.031}               },
  {\color{black}H$\phantom{^1}$ \\ \tiny{0.538}               },
  {\color{black}$\rmF^1$        \\ \tiny{0.911}               },
  {\color{black}$\rmF^2$        \\ \tiny{0.464}               },
  {\color{black}$\rmF^3$        \\ \tiny{0.130}               }}]

  \addplot+ [boxplot, mark=none, solid,  gray]  table[row sep=newline,y index=0] {figures/data/dual-monorels/custom/A00n2_B1000_width.dat};
  \addplot+ [boxplot, mark=none, solid,  gray]  table[row sep=newline,y index=0] {figures/data/dual-monorels/custom/A00n2_BRoot_width.dat};
  \addplot+ [boxplot, mark=none, solid,  gray]  table[row sep=newline,y index=0] {figures/data/dual-monorels/custom/A00n2_YAL_width.dat};
  \addplot+ [boxplot, mark=none, solid,  gray]  table[row sep=newline,y index=0] {figures/data/dual-monorels/custom/A00n2_CSTS_width.dat};
  \addplot+ [boxplot, mark=none, solid,  gray]  table[row sep=newline,y index=0] {figures/data/dual-monorels/custom/A00n2_SOPT_sig_width.dat};
  \addplot+ [boxplot, mark=none, solid,  black] table[row sep=newline,y index=0] {figures/data/dual-monorels/custom/A00n2_M_width.dat};
  \addplot+ [boxplot, mark=none, solid,  black] table[row sep=newline,y index=0] {figures/data/dual-monorels/custom/A00n2_Str1_width.dat};
  \addplot+ [boxplot, mark=none, solid,  black] table[row sep=newline,y index=0] {figures/data/dual-monorels/custom/A00n2_P1_width.dat};
  \addplot+ [boxplot, mark=none, solid,  black] table[row sep=newline,y index=0] {figures/data/dual-monorels/custom/A00n2_P2_width.dat};
  \addplot+ [boxplot, mark=none, solid,  black] table[row sep=newline,y index=0] {figures/data/dual-monorels/custom/A00n2_P3_width.dat};  
\end{groupplot}
\end{tikzpicture}

%% file: 5_appendix.tex
\clearpage
\section*{Appendix: Notation}

We denote the convex hull and the convex conic hull of sets by $\conv$ and $\cone$, respectively. The transposition operation is denoted as $\top$. By default, vectors are interpreted as columns, but when the choice of the column vs. row convention does not matter, we write vectors as rows to avoid unnecessary transposition. 

\begin{nomenclature}
  \nclentry{$[n]$}{\{1, \ldots, n\}}
  \nclentry{$|A|$}{Cardinality of set $A$}
  \nclentry{$A$}{Finite exponent set, $A \subset \N^n$}
  \nclentry{$B$}{Finite exponent set in $\N^n$, $B \supseteq A$ with $B = P_1 \cup \cdots \cup P_N $}
  \nclentry{$A+B$}{Minkowski sum, $\{ \alpha + \beta \colon \alpha \in A, \beta \in B \}$}
  \nclentry{$\BF(B)$}{Polyhedral relaxation based on \eqref{eq_BF}}
  \nclentry{$\Bx(l,u)$}{Box of lower and upper bounds $[l_1,u_1] \times \ldots \times [l_n, u_n]$}
  \nclentry{$\cC_A(X)$}{Moment cone $\overline {  \cone \{ x^A \colon x \in X\} }$}
  \nclentry{$\CRLX_\cF(X)$}{Conic moment relaxation $\{ v \in \R^B \colon v_{P_i} \in \cC_{P_i}(X) \ \text{for all} \in [N]\}$}
  \nclentry{$\deg(A)$}{Degree, $ \max \{ \deg(x^\alpha) \colon \alpha \in A\}$}
  \nclentry{$\cM_A(X)$}{Moment body $\cM_A(X) := \conv \{ (x^\alpha)_{\alpha \in A} \colon x \in X\}$}
  \nclentry{$\cF$}{Pattern familiy $\cF = \{P_1,\ldots,P_N\}$}
  \nclentry{$f(x)$}{Objective function, polynomial}
  \nclentry{$g$}{Vector of polynomial inequalities $g(x) \ge 0$ defining $X$}
  \nclentry{$h$}{Vector of polynomial equalities $h(x) = 0$ defining $X$}
  \nclentry{$K$}{Special case of feasible set $X$, either box or polytope}
  \nclentry{$L_v(f)$}{Linearization map $L_v \Bigl(\sum_{\alpha \in A} f_\alpha x^\alpha \Bigr) := \sum_{\alpha \in A} f_\alpha v_\alpha$}
  \nclentry{$M_B(x)$}{$M_B(x) = x^B (x^B)^\top$}
  \nclentry{$\MR(\cB_0,\dots,\cB_{k})$}{$\{ v \in \R^P \colon L_v(g_i M_B) \succeq 0  \ \forall \; \ i \in [k] \ \text{and} \ B \in \cB_i\}$}
  \nclentry{$\N_d^n$}{Space of exponents of maximum degree $d$ for polynomials in $\R^n$}
  \nclentry{$P_i$}{Pattern, as a member of the pattern family $\cF$}
  \nclentry{$\R[x]$}{Ring of polynomials in variables $x \in \R^n$ with real coefficients}
  \nclentry{$\R[x]_A$}{As $\R[x]$, but with monomials whose exponent vectors are in $A$}
  \nclentry{$\R^A$}{Real vector space isomorphic to $\R^{|A|}$ with vectors indexed by $A$}
  \nclentry{$\bS^m$}{Space of symmetric matrices of size $m$ over reals}
  \nclentry{$\bS_+^m$}{Space of positive semidefinite matrices  in $\bS^m$}
  \nclentry{$s$}{Number of polynomial inequality constraints in $X$}
  \nclentry{$t$}{Number of polynomial equality constraints in $X$}
  \nclentry{$v$}{Monomial variable vector with entries $v_\alpha = x^\alpha \in \R$}
  \nclentry{$x$}{Variables of original problem \eqref{POP}, $x \in \R^n$}
  \nclentry{$x^\alpha$}{Product $x^\alpha = \prod_{i=1}^n x_i^{\alpha_i}$}
  \nclentry{$X$}{Feasible set, $x \in X \subseteq \R^n$}
  \nclentry{$\overline{Y}$}{Topological closure of set $Y$}
  \nclentry{$y_A$}{If $y \in \R^n$ and $B \subseteq \N^n$, then for $A \subseteq B$ the vector $y_A$ is the}
  \nclentry{     }{projection of $y$ on coordinates indexed by $A$: $y_A = (y_\alpha)_{\alpha \in A} \in \R^A$}
  \nclentry{$Y_A$}{Projection of a set $Y$ on the $A$ coordinates: $Y_A  = \{ y_A \colon y \in Y \}$}
\end{nomenclature}

%% file: averkov.bbl
\begin{thebibliography}{10}

\bibitem{Adjiman1998}
{\sc C.~Adjiman, I.~Androulakis, and C.~Floudas}, {\em {A} global optimization
  method, {$\alpha$BB}, for general twice-differentiable constrained {NLPs} --
  {II.} {I}mplementation and computational results}, {C}omputers \& {C}hemical
  {E}ngineering, 22 (1998), pp.~1159--1179,
  \url{https://doi.org/http://dx.doi.org/10.1016/S0098-1354(98)00218-X},
  \url{http://www.sciencedirect.com/science/article/pii/S009813549800218X}.

\bibitem{ahmadi2023sums}
{\sc A.~A. Ahmadi, C.~Dibek, and G.~Hall}, {\em Sums of separable and quadratic
  polynomials}, Mathematics of Operations Research, 48 (2023), pp.~1316--1343.

\bibitem{ahmadi2019dsos}
{\sc A.~A. Ahmadi and A.~Majumdar}, {\em Dsos and sdsos optimization: more
  tractable alternatives to sum of squares and semidefinite optimization}, SIAM
  Journal on Applied Algebra and Geometry, 3 (2019), pp.~193--230.

\bibitem{Andersen2011}
{\sc M.~Andersen, J.~Dahl, Z.~Liu, L.~Vandenberghe, S.~Sra, S.~Nowozin, and
  S.~Wright}, {\em Interior-point methods for large-scale cone programming},
  Optimization for machine learning, 5583 (2011).

\bibitem{lasserre2011}
{\sc M.~F. Anjos and J.~B. Lasserre}, eds., {\em Handbook on semidefinite,
  conic and polynomial optimization}, vol.~166 of International Series in
  Operations Research \& Management Science, Springer, New York, 2012,
  \url{https://doi.org/10.1007/978-1-4614-0769-0},
  \url{https://doi.org/10.1007/978-1-4614-0769-0}.

\bibitem{aps2020mosek}
{\sc M.~ApS}, {\em Mosek modeling cookbook}, 2020.

\bibitem{averkov2019optimal}
{\sc G.~Averkov}, {\em Optimal size of linear matrix inequalities in
  semidefinite approaches to polynomial optimization}, SIAM Journal on Applied
  Algebra and Geometry, 3 (2019), pp.~128--151.

\bibitem{averkov2024convex}
{\sc G.~Averkov and C.~Scheiderer}, {\em Convex hulls of monomial curves, and a
  sparse positivstellensatz}, Mathematical Programming,  (2024), pp.~1--19.

\bibitem{bao2015global}
{\sc X.~Bao, A.~Khajavirad, N.~V. Sahinidis, and M.~Tawarmalani}, {\em Global
  optimization of nonconvex problems with multilinear intermediates},
  Mathematical Programming Computation, 7 (2015), pp.~1--37.

\bibitem{belotti2015couenne}
{\sc P.~Belotti}, {\em Couenne, an exact solver for nonconvex {MINLP}s}, 2015,
  \url{https://projects.coin-or.org/Couenne/}.

\bibitem{ben2001lectures}
{\sc A.~Ben-Tal and A.~Nemirovski}, {\em Lectures on modern convex
  optimization: analysis, algorithms, and engineering applications}, SIAM,
  2001.

\bibitem{MR3075433}
{\sc G.~Blekherman, P.~A. Parrilo, and R.~R. Thomas}, eds., {\em Semidefinite
  optimization and convex algebraic geometry}, vol.~13 of MOS-SIAM Series on
  Optimization, Society for Industrial and Applied Mathematics (SIAM),
  Philadelphia, PA; Mathematical Optimization Society, Philadelphia, PA, 2013.

\bibitem{Boland2017}
{\sc N.~Boland, S.~S. Dey, T.~Kalinowski, M.~Molinaro, and F.~Rigterink}, {\em
  Bounding the gap between the mccormick relaxation and the convex hull for
  bilinear functions}, Mathematical Programming, 162 (2017), pp.~523--535.

\bibitem{chandrasekaran2016relative}
{\sc V.~Chandrasekaran and P.~Shah}, {\em Relative entropy relaxations for
  signomial optimization}, SIAM J. Optim., 26 (2016), pp.~1147--1173,
  \url{https://doi.org/10.1137/140988978},
  \url{https://doi.org/10.1137/140988978}.

\bibitem{dalkiran2013boundfactor}
{\sc E.~Dalkiran and H.~D. Sherali}, {\em Theoretical filtering of {RLT}
  bound-factor constraints for solving polynomial programming problems to
  global optimality}, J. Global Optim., 57 (2013), pp.~1147--1172,
  \url{https://doi.org/10.1007/s10898-012-0024-z},
  \url{https://doi.org/10.1007/s10898-012-0024-z}.

\bibitem{Diedam2018}
{\sc H.~Diedam and S.~Sager}, {\em Global optimal control with the direct
  multiple shooting method}, Optimal Control Applications and Methods, 39
  (2018), pp.~449--470, \url{https://doi.org/10.1002/oca.2324}.

\bibitem{dressler2017}
{\sc M.~Dressler, S.~Iliman, and T.~de~Wolff}, {\em A {P}ositivstellensatz for
  sums of nonnegative circuit polynomials}, SIAM J. Appl. Algebra Geom., 1
  (2017), pp.~536--555, \url{https://doi.org/10.1137/16M1086303},
  \url{https://doi.org/10.1137/16M1086303}.

\bibitem{Esposito2000}
{\sc W.~Esposito and C.~Floudas}, {\em {D}eterministic {G}lobal {O}ptimization
  in {N}onlinear {O}ptimal {C}ontrol {P}roblems}, {J}ournal of {G}lobal
  {O}ptimization, 17 (2000), pp.~97--126,
  \url{https://doi.org/10.1023/A:1026578104213},
  \url{http://titan.princeton.edu/research.htm}.

\bibitem{Esposito2000a}
{\sc W.~Esposito and C.~Floudas}, {\em {G}lobal {O}ptimization for the
  {P}arameter {E}stimation of {D}ifferential-{A}lgebraic {S}ystems},
  {I}ndustrial and {E}ngineering {C}hemistry {R}esearch, 39 (2000),
  pp.~1291--1310, \url{https://doi.org/10.1021/ie990486w},
  \url{http://titan.princeton.edu/research.htm}.

\bibitem{fawzi2019representing}
{\sc H.~Fawzi}, {\em On representing the positive semidefinite cone using the
  second-order cone}, Mathematical Programming, 175 (2019), pp.~109--118.

\bibitem{gartner2012approximation}
{\sc B.~G{\"a}rtner and J.~Matousek}, {\em Approximation algorithms and
  semidefinite programming}, Springer Science \& Business Media, 2012.

\bibitem{handelman1988representing}
{\sc D.~Handelman}, {\em Representing polynomials by positive linear functions
  on compact convex polyhedra}, Pacific Journal of Mathematics, 132 (1988),
  pp.~35--62.

\bibitem{katthan2021unified}
{\sc L.~Katth{\"a}n, H.~Naumann, and T.~Theobald}, {\em A unified framework of
  sage and sonc polynomials and its duality theory}, Mathematics of
  Computation, 90 (2021), pp.~1297--1322.

\bibitem{Land1960}
{\sc A.~Land and A.~Doig}, {\em {A}n automatic method of solving discrete
  programming problems}, {E}conometrica, 28 (1960), pp.~497--520.

\bibitem{lasserre2001global}
{\sc J.~B. Lasserre}, {\em Global optimization with polynomials and the problem
  of moments}, SIAM Journal on optimization, 11 (2001), pp.~796--817.

\bibitem{lasserre2009moments}
{\sc J.~B. Lasserre}, {\em Moments, positive polynomials and their
  applications}, vol.~1, World Scientific, 2009.

\bibitem{lasserre2015}
{\sc J.~B. Lasserre}, {\em An introduction to polynomial and semi-algebraic
  optimization}, Cambridge Texts in Applied Mathematics, Cambridge University
  Press, Cambridge, 2015, \url{https://doi.org/10.1017/CBO9781107447226},
  \url{https://doi.org/10.1017/CBO9781107447226}.

\bibitem{laurent2009sums}
{\sc M.~Laurent}, {\em Sums of squares, moment matrices and optimization over
  polynomials}, in Emerging applications of algebraic geometry, vol.~149 of IMA
  Vol. Math. Appl., Springer, New York, 2009, pp.~157--270,
  \url{https://doi.org/10.1007/978-0-387-09686-5_7},
  \url{https://doi.org/10.1007/978-0-387-09686-5_7}.

\bibitem{Little1963}
{\sc J.~Little, K.~Murty, D.~Sweeney, and C.~Karel}, {\em {A}n {A}lgorithm for
  the {T}raveling {S}alesman {P}roblem}, {O}perations {R}esearch, 11 (1963),
  pp.~972--989, \url{https://doi.org/10.1287/opre.11.6.972},
  \url{http://dx.doi.org/10.1287/opre.11.6.972},
  \url{https://arxiv.org/abs/http://dx.doi.org/10.1287/opre.11.6.972}.

\bibitem{yalmip}
{\sc J.~L{\"{o}}fberg}, {\em {YALMIP}: A toolbox for modeling and optimization
  in {MATLAB}}, in In Proceedings of the CACSD Conference, Taipei, Taiwan,
  2004.

\bibitem{magron2023sparse}
{\sc V.~Magron and J.~Wang}, {\em Sparse polynomial optimization: theory and
  practice}, World Scientific, 2023.

\bibitem{marshall2008}
{\sc M.~Marshall}, {\em Positive polynomials and sums of squares}, vol.~146 of
  Mathematical Surveys and Monographs, American Mathematical Society,
  Providence, RI, 2008, \url{https://doi.org/10.1090/surv/146},
  \url{https://doi.org/10.1090/surv/146}.

\bibitem{matlab}
{\sc The Mathworks, Inc.}, {\em {MATLAB} version 9.6.0.1174912 (R2019a) Update
  5}, Natick, Massachusetts, 2019.

\bibitem{Mitsos2009}
{\sc A.~Mitsos, B.~Chachuat, and P.~Barton}, {\em {McC}ormick-{B}ased
  {R}elaxations of {A}lgorithms}, {SIAM} {J}ournal on {O}ptimization, 20
  (2009), pp.~573--601, \url{https://doi.org/10.1137/080717341},
  \url{http://dx.doi.org/10.1137/080717341},
  \url{https://arxiv.org/abs/http://dx.doi.org/10.1137/080717341}.

\bibitem{Mittelmann2010}
{\sc H.~D. Mittelmann and F.~Vallentin}, {\em High-accuracy semidefinite
  programming bounds for kissing numbers}, Experimental Mathematics, 19 (2010),
  pp.~175--179.

\bibitem{Moore1966}
{\sc R.~Moore}, {\em {I}nterval analysis}, Prentice-Hall, Englewood Cliffs, NJ,
  1966.

\bibitem{mosek}
{\sc {MOSEK ApS}}, {\em The MOSEK optimization toolbox for {MATLAB} manual.
  Version 9.0.}, 2019, \url{http://docs.mosek.com/9.0/toolbox/index.html}.

\bibitem{Neumaier2004}
{\sc A.~Neumaier}, {\em {C}omplete {S}earch in {C}ontinuous {G}lobal
  {O}ptimization and {C}onstraint {S}atisfaction}, Cambridge University Press,
  2004, pp.~271--369.

\bibitem{nie2023moment}
{\sc J.~Nie}, {\em Moment and polynomial optimization}, 2023.

\bibitem{Papamichail2002}
{\sc I.~Papamichail and C.~Adjiman}, {\em {A} {R}igorous {G}lobal
  {O}ptimization {A}lgorithm for {P}roblems with {O}rdinary {D}ifferential
  {E}quations}, {J}ournal of {G}lobal {O}ptimization, 24 (2002), pp.~1--33,
  \url{https://doi.org/10.1023/A:1016259507911},
  \url{http://dx.doi.org/10.1023/A%3A1016259507911}.

\bibitem{Peters2021a}
{\sc B.~Peters}, {\em Monomial Patterns in Polynomial Optimization}, PhD
  thesis, Otto von Guericke University Magdeburg, 2021,
  \url{https://mathopt.de/publications/Peters2021.pdf}.

\bibitem{putinar1993positive}
{\sc M.~Putinar}, {\em Positive polynomials on compact semi-algebraic sets},
  Indiana University Mathematics Journal, 42 (1993), pp.~969--984.

\bibitem{sahinidis:baron:17.8.9}
{\sc N.~V. Sahinidis}, {\em {BARON} 17.8.9: Global Optimization of
  Mixed-Integer Nonlinear Programs, Users Manual}, 2017.

\bibitem{saunderson2020limitations}
{\sc J.~Saunderson}, {\em Limitations on the expressive power of convex cones
  without long chains of faces}, SIAM Journal on Optimization, 30 (2020),
  pp.~1033--1047.

\bibitem{schneider2014}
{\sc R.~Schneider}, {\em Convex bodies: the {B}runn-{M}inkowski theory},
  vol.~151 of Encyclopedia of Mathematics and its Applications, Cambridge
  University Press, Cambridge, expanded~ed., 2014.

\bibitem{Scott2011}
{\sc J.~Scott, M.~Stuber, and P.~Barton}, {\em {G}eneralized {McC}ormick
  relaxations}, {J}ournal of {G}lobal {O}ptimization, 51 (2011), pp.~569--606,
  \url{https://doi.org/10.1007/s10898-011-9664-7},
  \url{http://dx.doi.org/10.1007/s10898-011-9664-7}.

\bibitem{seidler2018experimental}
{\sc H.~Seidler and T.~de~Wolff}, {\em An experimental comparison of {SONC} and
  {SOS} certificates for unconstrained optimization}, arXiv preprint
  arXiv:1808.08431,  (2018).

\bibitem{Shaydurova2024}
{\sc D.~Shaydurova, V.~Kaibel, and S.~Sager}, {\em Refined tssos}, 2024,
  \url{https://arxiv.org/abs/2402.05444}.

\bibitem{sherali2013reformulation}
{\sc H.~D. Sherali and W.~P. Adams}, {\em A reformulation-linearization
  technique for solving discrete and continuous nonconvex problems}, vol.~31,
  Springer Science \& Business Media, 2013.

\bibitem{smith2001symbolic}
{\sc E.~Smith, C.~Pantelides, and G.~Reklaitis}, {\em A symbolic reformulation
  spatial branch-and-bound algorithm for the global optimization of nonconvex
  {MINLP}s}, Computers \& Chemical Engineering, 25 (2001), pp.~1399--1401,
  \url{https://doi.org/10.1016/S0098-1354(01)00733-5}.

\bibitem{tange_2020_4118697}
{\sc O.~Tange}, {\em Gnu parallel 20201022}, Oct 2020,
  \url{https://doi.org/10.5281/zenodo.4118697},
  \url{https://doi.org/10.5281/zenodo.4118697}.
\newblock GNU Parallel is a general parallelizer to run multiple serial command
  line programs in parallel without changing them.

\bibitem{Tawarmalani2005}
{\sc M.~Tawarmalani and N.~V. Sahinidis}, {\em {A} polyhedral branch-and-cut
  approach to global optimization}, {M}athematical {P}rogramming, 103 (2005),
  pp.~225--249.

\bibitem{tawarmalani2005polyhedral}
{\sc M.~Tawarmalani and N.~V. Sahinidis}, {\em A polyhedral branch-and-cut
  approach to global optimization}, Math. Program., 103 (2005), pp.~225--249.

\bibitem{tawarmalani2013convexification}
{\sc M.~Tawarmalani and N.~V. Sahinidis}, {\em Convexification and global
  optimization in continuous and mixed-integer nonlinear programming: theory,
  algorithms, software, and applications}, vol.~65, Springer Science \&
  Business Media, 2013.

\bibitem{trefethen2019approximation}
{\sc L.~N. Trefethen}, {\em Approximation Theory and Approximation Practice,
  Extended Edition}, SIAM, 2019.

\bibitem{vandenberghe2015chordal}
{\sc L.~Vandenberghe, M.~S. Andersen, et~al.}, {\em Chordal graphs and
  semidefinite optimization}, Foundations and Trends in Optimization, 1 (2015),
  pp.~241--433.

\bibitem{scip}
{\sc S.~Vigerske and A.~Gleixner}, {\em {SCIP}: global optimization of
  mixed-integer nonlinear programs in a branch-and-cut framework}, Optim
  Methods Softw, 33 (2018), pp.~563--593,
  \url{https://doi.org/10.1080/10556788.2017.1335312},
  \url{https://doi.org/10.1080/10556788.2017.1335312}.

\bibitem{MR4198579}
{\sc J.~Wang, V.~Magron, and J.-B. Lasserre}, {\em Chordal-{TSSOS}: {A}
  {M}oment-{SOS} {H}ierarchy {T}hat {E}xploits {T}erm {S}parsity with {C}hordal
  {E}xtension}, SIAM J. Optim., 31 (2021), pp.~114--141,
  \url{https://doi.org/10.1137/20M1323564},
  \url{https://doi.org/10.1137/20M1323564}.

\bibitem{wolkowicz2012handbook}
{\sc H.~Wolkowicz, R.~Saigal, and L.~Vandenberghe}, {\em Handbook of
  semidefinite programming: theory, algorithms, and applications}, vol.~27,
  Springer Science \& Business Media, 2012.

\end{thebibliography}
